\documentclass[reqno,a4paper]{amsart}

\usepackage{header}
\renewcommand{\S}{\mathbb{S}}
\definecolor{green}{rgb}{0.0,0.6,0.0}
\definecolor{blue}{rgb}{0.0,0.0,1.0}
\definecolor{red}{rgb}{1.0,0.0,0.0}

\newcommand{\nablap}{{\nabla_+}}
\newcommand{\nablam}{{\nabla_-}}
\newcommand{\uk}{{\underline k}}
\newcommand{\KK}{{\opn{K}}}
\renewcommand{\CC}{{\opn{C}}}
\newcommand{\CPvar}{\CD}
\newcommand{\NablaP}{{\nabla_+}}
\newcommand{\NablaM}{{\nabla_-}}
\newcommand{\boxx}{{\scalebox{0.4}{$\blacksquare$}}}
\newcommand{\DeltaV}{{\Delta}}
\renewcommand{\H}{\mathbb{H}}
\begin{document}
\parindent0mm

\title[Toric sheaves and polyhedra]
{Toric sheaves and polyhedra}

\author[K.~Altmann]{Klaus Altmann
}
\address{Institut f\"ur Mathematik,
FU Berlin,
K\"onigin-Luise-Str.~24-26,
D-14195 Berlin
}
\email{altmann@math.fu-berlin.de}
\author[A.~Hochenegger]{Andreas Hochenegger}
\address{
Dipartimento di Matematica ``Francesco Brioschi'',
Politecnico di Milano,
via Bonardi 9,
20133 Milano 
}
\email{andreas.hochenegger@polimi.it}
\author[F.~Witt]{Frederik Witt}
\address{
Fachbereich Mathematik,
U Stuttgart,
Pfaffenwaldring 57,
D-70569 Stuttgart
}
\email{witt@mathematik.uni-stuttgart.de}

\thanks{{\bf MSC 2020:}
14C20, 
14F06, 
14F08,  
14M25, 
18G10, 
52C07, 
55N05  
\hfill\newline
\indent{\bf Key words:} \v Cech type cohomology, cohomology of sheaves, derived categories, extension problem, polyhedra, reflexive sheaves, right derived functors, sheaves on poset spaces, toric varieties}

\begin{abstract}
Over a smooth projective toric variety we study toric sheaves, that is, reflexive sheaves equivariant with respect to the acting torus, from a polyhedral point of view. One application is the explicit construction of the torus invariant universal extension of two nef line bundles via polyhedral inclusion/exclusion sequences. 

Second, we link the cohomology of toric sheaves to the cohomology of certain constructible sheaves explicitly built out of the associated polyhedra. For the latter we define a concrete double complex and a spectral sequence which computes the cohomology of toric sheaves from the reduced cohomology of polyhedral subsets living in the realification of the character lattice of the toric variety.
\end{abstract}

\maketitle
\setcounter{tocdepth}{1}
\tableofcontents

\section{Introduction}
\label{sec:Intro}
Let $\kk$ be an algebraically closed field of characteristic $0$ and $N$ be a free abelian group of finite rank. This gives rise to the algebraic torus
\[
\TX:=N\otimes_\Z\kk^\times;
\]
together with a fan $\Sigma$ inside the realification $N_\R=N\otimes_\Z\R$ of $N$, we obtain the toric variety $X=\toric(\Sigma)$. In principle, every invariant of $X$ can be expressed in combinatorial terms involving solely $N$ and $\Sigma$. In practice, however, it is often more convenient to work with the dual group $M=\Hom_\Z(N,\Z)$ and polyhedra inside $M_\R=M\otimes_\Z\R$. This applies for instance if $X$ is projective, that is, its defining fan arises as the normal fan $\mc N(\Delta)$ of a polytope $\Delta\subseteq M_\R$. In this case, we also write $X=\PP(\Delta)$ for emphasis.

\medskip

Let us illustrate this point by considering an invertible sheaf $\CL$. For simplicity, we assume here (and mostly throughout this paper) that $X$ is smooth. If $j_\TX\colon\TX\into X$ denotes the open embedding of the torus, we can think of $\CL$ as a subsheaf of $j_{\TX*}\CO_\TX$ where the latter corresponds to the semigroup algebra $\kk[M]$ defined by $M$. In particular, this yields a unique $\TX$-invariant divisor $D$ with $\CL=\CO_X(D)$. Moreover, $D$ is given by a finite sum of $\TX$-invariant prime divisors $D_\rho$ indexed by the one-dimensional cones $\rho\in\Sigma$, cf.\ Subsection~\ref{subsec:Div}. 

\medskip

On the other hand, we can write $D$ as a (non-unique) difference $D_+-D_-$ of two toric nef (or even ample) Cartier divisors. This corresponds to a pair of two lattice polytopes $\nabla_+$ and $\nabla_-$ sitting inside $M_\R$ and whose normal fans are refined by $\Sigma$. In fact, we think of such pairs as formal differences in the sense of Grothendieck groups and refer to $\nabla=\nabla_+-\nabla_-$ as a {\em virtual polytope}; see Subsection~\ref{subsec:VP} for details. Differences of polyhedra also appear for instance in the construction of the polytope algebra, see~\cite{brion1} and~\cite{brion2} for its connections to toric and convex geometry.

\medskip

Many properties of invertible sheaves can be easily read off from this point of view. For instance, by results from~\cite{immaculate} and~\cite{dop} we have
\begin{equation}
\label{eq:AP}
\gH^\ell(X,\CO_X(D))_0=\widetilde{\opn H}^{\raisebox{-3pt}{\scriptsize$\ell\!-\!1$}}\big(\nabla_-\setminus\nabla_+\big),
\end{equation}
where $\widetilde{\opn H}^{\raisebox{-3pt}{\scriptsize$\ell\!-\!1$}}$ 
denotes the 
reduced singular cohomology with coefficients in $\kk$ of the set-theoretic difference $\nabla_-\setminus\nabla_+$, and the subscript $0$ refers to the degree $0$ part of the grading induced by the torus action. Note that reduced cohomology can be indeed nontrivial in degree $-1$: Given a topological space $Z$, then $\widetilde{\opn H}^{\raisebox{-3pt}{\scriptsize$-1$}}(Z)=0$ if $Z\not=\varnothing$, and equals $\kk$ if $Z=\varnothing$.

\subsection{Toric sheaves and Weil decorations}
\label{subsec:TSWD}
One of the major themes of this article is to generalise the previous statements for line bundles to {\em toric sheaves} of arbitrary rank. By definition, these are reflexive sheaves that are
equivariant with respect to the acting torus $\TX$, see for instance~\cite[Chapter 1.3]{mumfordGIT}. 

\medskip

The restriction of a toric sheaf $\CE$ to $\TX$ is determined by a free graded $\kk[M]$-module whose torus invariant part is a $\kk$-vector space $E$. A {\em Weil decoration} for $\CE$ is a map $\CD\colon E\setminus\{0\}\to\Div_\TX(X)$ into the $\TX$-invariant divisors, which factorises over the projectivisation $\P(E)$ and satisfies
\[
\CD(e+e')\geq\CD(e)\wedge\CD(e')
\]
for all $e$, $e'\in E\setminus\{0\}$ provided $e+e'\not=0$. Here, $\wedge$ denotes the natural meet on $\Div_\TX(X)$ given by taking the minimum of the coefficients in $\CD(e)$ and $\CD(e')$. 
As we shall explain in Section~\ref{sec:WDTS}, with every toric sheaf $\CE$ we can associate a Weil decoration $\CD_\CE$ which completely characterises $\CE$ as a subsheaf of $E\otimes (j_\TX)_*\CO_\TX$; the divisors $\CD(e)$ come from the invertible sheaves $\CE\cap \big(e\cdot(j_\TX)_*\CO_\TX\big)$, see~\ssect{subsec:TS}.

\begin{example}
\label{exam:LB}
Let $\CE=\mc O(D)\subseteq(j_\TX)_*\CO_\TX=\kk[M]$ be a reflexive sheaf of rank $1$. Then $E=\kk$ with Weil decoration determined by $\CD_\CE(1)=D$. In this way, Weil decorations provide a natural reformulation of Weil divisors.

\smallskip

Similarly, the trivial higher rank sheaf $\CE=\mc O_X(D)\otimes E$ is associated with $E$ and $\CD_\CE(e)=D$ for all $e\in E\setminus\{0\}$.  
\end{example}

\begin{example}
\label{exam:SumLB}
Next, let $\CE=\CO_X(D_1)\oplus\CO_X(D_2)$ be the direct sum of reflexive sheaves of rank $1$. The torus invariant decomposition of $\CE$ yields $E=e_1\cdot\kk\oplus e_2\cdot\kk$, where $e_1$ and $e_2\in E$ correspond to the summands of $\CE$. The associated Weil decoration for $e \in E\setminus\{0\}$ is given by
\[
\CD_\CE(e)=\left\{\begin{array}{cl}D_1&\text{if }e\in \kk\cdot e_1\\[3pt]D_2&\text{if }e\in \kk\cdot e_2\\[3pt]D_1\wedge D_2&\text{if otherwise}\end{array}\right..
\]
\end{example}

\bigskip

By pooling together in the set $\S(D)$ the points of $E\setminus\{0\}$ giving rise to the same divisor $D$ and adding the origin, we actually obtain a stratification $\Strat_\CE$ of $E$ defining a join semilattice. Furthermore, the Weil decoration $\CD_\CE$ descends to an anti-semilattice morphism $\CD_\CE\colon\Strat_\CE\to\Div_\TX(X)$, which is best visualised in terms of a Hasse diagram. For instance, $\CO_X(D_1)\oplus\CO_X(D_2)$ yields
\[
\begin{tikzcd}
& \eta:=\S(D_1\wedge D_2)=\S(D_1)\vee\S(D_2)
& \\[3pt]
\S(D_1)=\kk^\times\cdot e_1\ar[ru, no head] & & \S(D_2)=\kk^\times\cdot e_2\ar[lu, no head]
\end{tikzcd}
\]
where $\eta=E\setminus(e_1\cdot\kk\,\cup\,e_2\cdot\kk)$.

\medskip

Besides providing a handy description of toric sheaves, we can also read off various global data directly from the Weil decoration. For instance, if $X$ is smooth and projective, the $\TX$-equivariant Euler characteristic of a toric sheaf $\CE$ is given by
\[
\chi^\TX(X,\CE)=\sum_{\S\leq\T}\dim(\S)\cdot\Inc_\CE^{-1}\cdot\chi^\TX\big(X,\CO_X(\CD_\CE(\T))\big),
\]
where $\Inc_\CE$ is the incidence matrix of the lattice $\Strat_\CE$. Furthermore, for a toric sheaf $\CE$ whose Weil divisors $\CD(\S)$ are ample, let $\S_{\mr{glob},m}$, $m\in M$, be maximal among the strata whose polytope associated with $\CD_\CE(\S)$ contains $m$. Then the space of global sections of degree $m$ is given by
\[
\gH^0(X,\CE)_m=\overline\S_{\mr{glob},m}\subseteq E.
\]
A similar formula holds for general $\CE$. See \ssect{subsec:EulerGlobSec} for details.

\medskip

Lattices in the order theoretic sense have played an important r\^ole in the theory of toric sheaves ever since the seminal work of Klyachko~\cite{klyachko}. In fact, we recover the Klyachko filtrations of $\CE$ from its Weil decoration $\CD_\CE$ via
\begin{equation}
\label{eq:KlyachkoFilt}
E^\ell_\rho=\{e\in E \setminus \{0\}\mid\,\text{the coefficient of }D_\rho\text{ in }\CD_\CE(e)\text{ is }\geq\ell\}\cup\{0\},
\end{equation}
where $D_\rho$ denotes the toric prime divisor associated with the ray $\rho\in\Sigma(1)$. Assuming that $X$ is smooth, it also induces a piecewise linear valuation in the sense of Kaveh and Manon~\cite{tropical2} after identifying divisors with piecewise linear functions. Finally, we recover the parliament of polytopes of di Rocco, Jabbusch and Smith~\cite{parliaments} by mapping the divisors $\CD_\CE(e)$ to their polytope of sections $\gH^0_\R(\CD_\CE(e))$, see~\eqref{eq:H0R} on page~\pageref{eq:H0R}.

\subsection{The universal extension}
\label{subsec:UnivExt}
With every lattice polytope $\nabla$ compatible with the fan $\Sigma$ we can associate a nef line bundle $\CO_X(\nabla)$. As a first application of Weil decorations, we revisit the problem of constructing the {\em toric universal extension} $\CE(\nablam,\nablap)$ of $\CO_X(\nablam)$ by $\CO_X(\nablap)$, cf.~\cite{ka48-dispExt}. Namely, for the $\TX$-equivariant extensions 
\[
\gExt:=\gExt^1_{\CO_X}\!\!\big(\CO_X(\nablam),\CO_X(\nablap)\big)_0 
\]
we wish to construct a short exact sequence
\[
\begin{tikzcd}
0\ar[r]&\gExt^\vee\otimes_\kk\CO_X(\nablap)\ar[r]&\CE(\nablam,\nablap)\ar[r]&\CO_X(\nablam)\ar[r]&0
\end{tikzcd}
\]
such that the extension induced by pushout along $t\in\gExt$ is actually the one specified by $t$.

\begin{example}{\cite[Section 1]{ka48-dispExt}} 
Consider the Hirzebruch surface $\F_1=\toric(\Sigma)$ and its nef polytopes $\nablap$ and $\nablam$ from Figure~\ref{fig:F1Nef}. 
\begin{figure}[ht]
\begin{tikzpicture}[scale=0.3]
\draw[thick, color=black]
  (0,0) -- (4,0) (0,-4) -- (0,4) (0,0) -- (-3.0,-3.0);
\draw[thick, color=black]
  (2,2) node{$\sigma_{12}$};
\draw[thick, color=black]
  (-2,1) node{$\sigma_{23}$};
\draw[thick, color=black]
  (-1.2,-2.7) node{$\sigma_{34}$};
\draw[thick, color=black]
  (2.5,-2) node{$\sigma_{41}$};
\draw[thick, color=black]
  (5,0) node{$\rho_1$} (0,5) node{$\rho_2$} (-3.8,-3.8) node{$\rho_3$}
  (1.0,-3.8) node{$\rho_4$};
\end{tikzpicture}
\qquad
\begin{tikzpicture}[scale=0.85]
\draw[color=oliwkowy!40] (-0.3,-0.3) grid (2.3,2.3);
\draw[thick, color=green]
  (0,0) -- (2,0) -- (0,2) -- cycle;
\draw[line width=1mm, color=blue]
  (0,1) -- (1,1);
\draw[very thick, color=blue]
  (1.2,1.3) node{$\nablap$};
\fill[color=red]
  (0,1) circle (4pt);
\draw[thick, color=green]
  (2.3,0.3) node{$\nablam$};
\end{tikzpicture}
\caption{The Hirzebruch surface $\F_1=\toric(\Sigma)$ given by the fan on the left-hand side. The red dot on the right-hand side indicates the origin in $M_\R$ and fixes the position of the polytopes.}
\label{fig:F1Nef}
\end{figure}
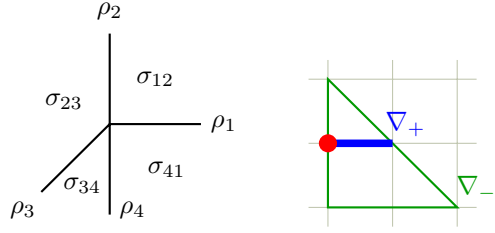

\medskip

The $\TX$-equivariant extensions are given by $\gExt=\widetilde{\opn H}^{\raisebox{-3pt}{\scriptsize$0$}}\big(\nablam\setminus\nablap\big)$, cf.~Equation~\eqref{eq:AP}. In our specific situation, $\gExt\cong\kk$, and the nontrivial extension is encapsulated in the inclusion/exclusion sequence of polyhedra displayed in Figure~\ref{fig:F1PolRes}. Here, $\nabla_0:=C_0\cup\nablap$ and $\nabla_1:=C_1\cup\nablap$, where $C_0$ and $C_1$ denote the connected components of $\nablam\setminus\nablap$. 
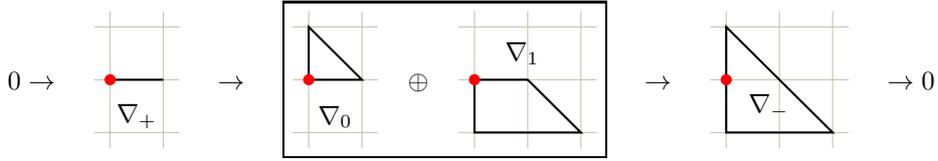
\begin{figure}[ht]
\newcommand{\scaleA}{0.7}
\newcommand{\scaleB}{1.0}
\newcommand{\raiseC}{5.2ex}
\newcommand{\spaceA}{\hspace*{1.1em}}
\raisebox{\raiseC}{$0 \rightarrow$}
\spaceA
\begin{tikzpicture}[scale=\scaleA]
\draw[color=oliwkowy!40] (-0.3,-0.3) grid (1.3,2.3);
\draw[thick, color=black]
  (0,1) -- (1,1);
\fill[thick, color=red]
  (0,1) circle (3pt);
\draw[thick, color=black]
  (0.5,0.3) node{$\nablap$};
\end{tikzpicture}
\spaceA
\raisebox{\raiseC}{$\rightarrow$}
\spaceA
\fbox{$
\begin{tikzpicture}[scale=\scaleA]
\draw[color=oliwkowy!40] (-0.3,-0.3) grid (1.3,2.3);
\draw[thick, color=black]
  (0,1) -- (1,1) -- (0,2) -- cycle;
\fill[thick, color=red]
  (0,1) circle (3pt);
\draw[thick, color=black]
  (0.5,0.3) node{$\NablaA$};
\end{tikzpicture}
\spaceA
\raisebox{\raiseC}{$\oplus$}
\spaceA
\begin{tikzpicture}[scale=\scaleA]
\draw[color=oliwkowy!40] (-0.3,-0.3) grid (2.3,2.3);
\draw[thick, color=black]
  (0,0) -- (2,0) -- (1,1) -- (0,1) -- cycle;
\fill[thick, color=red]
  (0,1) circle (3pt);
\draw[thick, color=black]
  (0.9,1.5) node{$\NablaB$};
\end{tikzpicture}
$}
\spaceA
\raisebox{\raiseC}{$\rightarrow$}
\spaceA
\begin{tikzpicture}[scale=\scaleA]
\draw[color=oliwkowy!40] (-0.3,-0.3) grid (2.3,2.3);
\draw[thick, color=black]
  (0,0) -- (2,0) -- (0,2) -- cycle;
\fill[thick, color=red]
  (0,1) circle (3pt);
\draw[thick, color=black]
  (0.8,0.5) node{$\nablam$};
\end{tikzpicture}
\spaceA
\raisebox{\raiseC}{$\rightarrow 0$}
\caption{The polyhedral resolution of $\nablam\setminus\nablap$}
\label{fig:F1PolRes}
\end{figure}
This sequence actually translates into the non-split extension (albeit defined by a split bundle)
\begin{equation}
\label{eq:ExacSeq}
0\to\CO_X(\nablap)\to\fbox{$\CE(\nablam,\nablap):=\CO_X(\NablaA)\oplus\CO_X(\NablaB)$}\to\CO_X(\nablam)\to0.
\end{equation}

\medskip

More generally, the same method applies for more than two components as long as $\nablap\subseteq\nablam$. This viewpoint, 
however, breaks down if $\nablap\cap\nablam$ is no longer a lattice polytope compatible with the fan (see Figure~\ref{fig:NonIntegralInter} on page~\pageref{fig:NonIntegralInter} for an example), and~\cite{ka48-dispExt} falls back to Klyachko filtrations. 
\end{example}

\bigskip

In our approach, we specify $\CE(\nablam,\nablap)$ by its Weil decoration. 
In the former situation where $\nablap\subseteq\nablam$, 
the sheaf $\CE(\nablam,\nablap)$ splits and we have $E=\kk e_1\oplus\kk e_2$. This yields the Hasse diagram
\[
\begin{tikzcd}
& \eta=
\S(\nabla_0\cap\nabla_1)=\S(\nablap) &\\[3pt]
\S(\nabla_0)=\kk^\times\cdot e_0\ar[ru, no head] & & \S(\nabla_1)=\kk^\times\cdot e_1\ar[lu, no head]
\end{tikzcd}
,
\]
cf.\ Example~\ref{exam:SumLB}. If, on the other hand, $\nablam\cap\nablap$ is not a lattice polytope, the intersection makes still sense as a {\em virtual intersection} $\nablam\vcap\nablap$. We therefore regard $\nabla_0$ and $\nabla_1$ as virtual polytopes with $\nabla_0\vcap\nabla_1=\nablam\vcap\nablap$ and add the stratum $\S(\nablap)$, that is,
\[
\begin{tikzcd}
& \eta=\S(\nabla_0\vcap\nabla_1)&\\[3pt]
&&\S(\nablap)=\kk^\times\cdot(e_1-e_0)\ar[lu,no head]\\[3pt]
\S(\nabla_0)=\kk^\times\cdot e_0\ar[ruu, no head] & \S(\nabla_1)=\kk^\times\cdot e_1\ar[uu, no head]&
\end{tikzcd}
.
\]
As we discuss in Section~\ref{sec:ExtNLB}, this yields an inclusion/exclusion 
exact sequence in the sense specified in Section~\ref{sec:TorMor}. The case of $n$ (virtual) connected components is treated similarly, which yields the

\begin{theorem}
For a smooth projective toric variety $X=\P(\Delta)$, the toric sheaf $\CE(\nablam,\nablap)$ is the $\TX$-invariant universal extension of $\CO_X(\nablam)$ by $\CO_X(\nablap)$. 
\end{theorem}

See Section \ref{sec:ExtNLB} for details. This surprisingly catchy structure makes the universal toric extensions $\CE(\nablam,\nablap)$ akin to tangent bundles and more generally to Kaneyama bundles~\cite{kaneyama}.

\subsection{The cohomology of toric sheaves}
As exemplified in the previous example, the polyhedral language provides a handy tool to compute the cohomology of line bundles. We generalise this approach to a toric sheaf $\CE$ over a smooth projective variety $X=\P(\Delta)$ by using its $\Pol(\Sigma)$-valued Weil decoration $\CD$. 

\medskip

Replacing $\Delta$ by a suitable ample polytope if necessary, we assume the Weil decoration $\CD^+$ of
\[
\CE^+:=\CE(\Delta):=\CE\otimes\CO_X(\Delta)
\]
to consist of ample polytopes only. Then we can define a constructible $\kk$-linear sheaf $\CF(\CE)$ over $\Delta$ by
\begin{equation}
\label{eq:DefCFCE}
\CF(\CE)(U):=\{e\in E\mid U\subseteq\CPvar^+(e)\},\quad U\subseteq\Delta\text{ open},
\end{equation}
which is a subsheaf of the constant sheaf $\underline E_\Delta$.

\begin{theorem}
The $\TX$-invariant piece of the cohomology of $\CE$ is given by
\[
\gH^\ell(X,\CE)_0\cong\gH^\ell\hspace{-0.5ex}\big(\Delta,\CF(\CE)\big).
\]
\end{theorem}

\medskip

Shifting $\Delta$ by $m\in M$ also yields the $m$-graded piece of the cohomology, see Theorem~\ref{thm:CohomSheafE} and Remark~\ref{rem:GenDegree} for details. 

\medskip

The link between $\CE$ and $\CF$ seems to be an instance of a broader correspondence between coherent and constructible sheaves over toric varieties~\cite{zaslow},~\cite{treumann}. 

\medskip

Our goal, however, is to gain a concrete way of visualising and computing the cohomology of toric sheaves in polyhedral terms. Theorem \ref{thm:FRes} constructs a natural double complex 
from which we obtain the cohomology of $\CE$ as the limit of a spectral sequence. The $E_1$-terms are essentially given by the reduced cohomology of polyhedral albeit nonconvex subsets of $M_\R$, see Theorem~\ref{thm:SpecSeq} for details. 

\medskip

On the practical side, we note that the length of the complex of Theorem \ref{thm:FRes} is governed by the height of the stratification $\Strat_\CE$, while it is $\dim X$ for the usual Klyachko-\v Cech complex, 
cf.~\cite{klyachko}.
Expressing the cohomology of $\CE$ as a mapping cone in the bounded derived category, we can compute the cohomology for rank $2$ bundles or the universal extensions considered in \ssect{subsec:UnivExt}, see~\ssect{subsec:HeightOne}. For illustration, we apply this method in \ssect{subsubsec:cohomSheafTang} and  \ssect{subsec:TangBundle} to the twisted tangent sheaves $\CT_{\PP^2}(\ell)$ .

\subsection*{Acknowledgements}
It is a pleasure to thank David Ploog for his very valuable comments on a preliminary version of this article.

\section{Toric geometry}
\label{sec:TG}
We briefly recall some general features of toric geometry we shall use in the sequel. A more detailed introduction can be found in~\cite{fultonToric}.

\subsection{Toric varieties}
\label{subsec:TV}
Consider the perfect pairing $\langle\cdot\,,\cdot\rangle\colon M\times N\to\Z$ between $N\cong\Z^\kdimX$ and $M = \Hom(N,\Z)$. Using the $\kdimX$-dimensional algebraic torus $\TX=N\otimes_\Z\kk^\times$ we can think of $N$ as $\Hom_\Z(\kk^\times,\TX)$, the {\em $1$-parameter subgroups of $\TX$}. On the other hand, $M$ arises as the dual {\em character group} $\Hom_\Z(\TX,\kk^\times)$. 

\medskip

Let $\sigma$ be a {\em cone} inside $N_\R=N\otimes_\Z\R$, a finitely generated convex set $\sum_{\mr{finite}}\R_{\geq0}\cdot\rho_i$, which is rational, i.e., $\rho_i\in N$, and pointed, i.e., $\sigma\cap(-\sigma)=\{0\}$. Any such cone defines an {\em affine toric variety}
\[
\toric(\sigma):=\Spec\kk[\sigma^\vee\cap M]
\]
for the so-called {\em dual cone}
\[
\sigma^\vee:=\{u\in M_\R\mid\langle u,\sigma\rangle\geq0\}.  
\]
Each face $\tau\leq\sigma$ provides a natural open embedding $\toric(\tau)\into\toric(\sigma)$. In particular, $\toric(\sigma)$ contains the torus
\[
\TX=\Spec\kk[M]=\toric(0).
\]

\medskip

A finite set $\Sigma$ of cones is called a {\em fan} if it is closed under taking faces, and $\sigma\cap\sigma'$ is a common face for any two cones $\sigma$, $\sigma'\in\Sigma$. The subset $|\Sigma|=\bigcup_{\sigma\in\Sigma}\sigma$ of $N_\R$ is called the {\em support} of $\Sigma$. To avoid certain degenerate cases, we always assume that
\begin{equation}
\label{eq:GenAss}
|\Sigma|\text{ \bf is convex and full-dimensional.}
\end{equation}
For instance, any cone $\sigma\in\Sigma$ defines the fan $[\sigma]$ consisting of the faces of $\sigma$. As the face relation yields a natural partial order on $\Sigma$ we obtain the {\em toric variety}
\[
X=\toric(\Sigma):=\varinjlim_{\sigma\in\Sigma}\toric(\sigma)
\]
by glueing the family of affine toric varieties specified by the fan. In particular, $\toric([\sigma])=\toric(\sigma)$. The image of the open embedding $\toric(\sigma)\into\toric(\Sigma)$ will be denoted $U_\sigma$. In particular, the torus
\[
j_\TX\colon\TX\stackrel{\sim}{\longrightarrow}U_0\subseteq X
\]
embeds as a dense open subset, and the action on itself extends to all of $X$. Therefore, we have a natural $M$-grading on every algebraic structure functorially attached to $\toric(\Sigma)$. 

\medskip

An important class of fans arises by taking the {\em normal fan} $\mc N(\Delta)$ of a full dimensional lattice polyhedron inside $M_\R$. In this case, we write
\[
X=\PP(\Delta):=\toric\big(\mc N(\Delta)\big);
\]
a toric variety arising this way is automatically quasi-projective and polarised. Actually,~\eqref{eq:GenAss} implies that it is even {\em semi-projective}, namely, $X$ is projective over the affine scheme defined by its ring of global functions.

\subsection{Divisors}
\label{subsec:Div}
A salient feature of toric geometry is the explicit description of its class
group. Let $\Sigma(d)$, $d\in\N$, denote the subset of $d$-dimensional cones
of $\Sigma$. Cones in $\Sigma(1)$ will be called {\em rays}; we tacitly
identify these with their primitive generators. Rays define one-codimensional torus orbits $\opn{orb}(\rho):=\Spec\kk[\rho^\perp\cap M]$ that are closed inside $U_\rho\subseteq X$. Their closures in $X$, namely
\[
D_\rho:=\overline{\opn{orb}(\rho)},\quad\rho\in\Sigma(1),
\]
are precisely the $\TX$-invariant prime divisors of $X$. Furthermore, we have the celebrated exact sequence
\begin{equation}
\label{eq:MotherSeq}
\begin{tikzcd}
0\ar[r]&M\ar[r,"\div"]&
\fbox{$\Z^{\Sigma(1)}=\bigoplus_{\rho\in\Sigma(1)}\Z\cdot D_\rho$}
\ar[r,"{[\cdot]}"]&\opn{Cl}(X)\ar[r]&0  
\end{tikzcd}
\end{equation}
where $\div(m):=\div(x^m)=\sum\langle m,\rho\rangle D_\rho$ and $[\cdot]$ sends $D_\rho$ to its class $[D_\rho]$. In particular, the classes of $\TX$-invariant or {\em toric} divisors
\[
\Div(\Sigma):=\Div_\TX(X)=
\big\{\!\sum_{\rho\in\Sigma(1)}\!a_\rho D_\rho\mid a_\rho\in\Z\big\}
\]
generate the full class group.

\begin{remark}
Since we will consider toric divisors only, we usually drop the qualifier ``toric'' in the sequel and simply speak of Weil divisors in $\Div(\Sigma)$.  
\end{remark}

\subsection{Virtual polyhedra}
\label{subsec:VP}
In this section, we recall the basic correspondence between toric divisors and virtual polyhedra, assuming that $X$ is smooth of dimension $\kdimX$. Since we will also work locally, we consider a quasi-projective $X=\P(\Delta)$ given by a (not necessarily compact) polyhedron $\Delta$ giving rise to the ample sheaf $\CO_X(\Delta)$. 

\subsubsection{The Grothendieck group $\Pol(\Sigma)$}
Any polyhedron $\nabla\subseteq M_\R$ can be written as the Minkowski sum of a polytope and its tail cone
\[
\opn{tail}(\nabla):=\{u\in M_\R\mid u+\nabla\subseteq\nabla\}.
\]

\begin{definition}
A lattice polyhedron $\nabla$ is {\em compatible with} $\Sigma$ if its normal fan $\mc N(\nabla)$ is refined by $\Sigma$. In particular, $\opn{tail}(\nabla)=|\Sigma|^\vee$ is the dual of the support of $\Sigma$. We set 
\[
\Pol^+(\Sigma):=\text{ the set of compatible lattice polyhedra.}
\]
With Minkowski sum as addition, $\Pol^+(\Sigma)$ becomes a cancellative semigroup. The associated Grothendieck group will be denoted $\Pol(\Sigma)$ and referred to as the {\em group of virtual polyhedra}. We think of the latter as formal differences $\nabla=\nabla_+-\nabla_-$ with $\nabla_\pm\in\Pol^+(\Sigma)$.
\end{definition}

\medskip

Compatibility entails that for every cone $\sigma\in\Sigma$ there is a uniquely determined face $\nabla(\sigma)$ of $\nabla\in\Pol^+(\Sigma)$ satisfying
\begin{equation}
\label{eq:sigmaCorner}
\langle\nabla(\sigma),a\rangle=\min\langle\nabla,a\rangle\quad\text{for all }a\in\sigma;
\end{equation}
the notation is justified by the fact that the left-hand side is a constant value. If $\sigma$ is full-dimensional, then $\nabla(\sigma)$ is a vertex; 
distinct $\sigma$ may lead to the same vertex if $\nabla$ is not ample. At any rate, defining the set $\{\nabla_\sigma\}_{\sigma\in\Sigma}$ of {\em local polyhedra of} $\nabla$ through
\[
\nabla_\sigma:=\nabla(\sigma)+\sigma^\vee=\nabla+\sigma^\vee
\]
we obtain
\begin{equation}
\label{eq:NablaInter}
\nabla=\opn{conv}\{\nabla(\sigma)\mid\sigma\in\Sigma(\kdimX)\}+|\Sigma|^\vee=\bigcap_{\sigma\in\Sigma(\kdimX)}\nabla_\sigma.
\end{equation}
.

\medskip

The vertices $\nabla(\sigma)$ associated with $\sigma\in\Sigma(\kdimX)$ induce the monomials $x^{\nabla(\sigma)}\in\kk[M]$ which define an invertible sheaf $\CO_X(\nabla)$ by
\[
\CO_X(\nabla)|_{U_\sigma}:=x^{\nabla(\sigma)}\cdot\CO_{U_\sigma}=\kk[\nabla_\sigma\cap M]\subseteq\kk[M].
\]
We therefore think of $\CO_X(\nabla)$ as a subsheaf 
\begin{equation}
\label{eq:Subsheaf} 
\CO_X(\nabla)\subseteq(j_\TX)_*\CO_\TX=\kk[M],
\end{equation}
slightly abusing notation for the latter equality. Furthermore,
\[
x^{\nabla(\sigma)}\in\gH^0(X,\CO_X(\nabla))=\kk[\nabla\cap M],
\]
which implies that $\CO_X(\nabla)$ is globally generated. Equivalently, its associated divisor determined by the embedding~\eqref{eq:Subsheaf}, namely
\[
D_\nabla=-\sum_{\rho\in\Sigma(1)}\min\langle\nabla,\rho\rangle\cdot D_\rho,
\]
is basepoint free, which for toric varieties is the same as being nef. Finally,
we have
\begin{equation}
\label{eq:Tensor}
\CO_X(\nabla)\otimes\CO_X(\nabla')=\CO_X(\nabla+\nabla')
\end{equation}
for any pair of compatible lattice polyhedra. In particular, we may define for a virtual polyhedron $\nabla=\nablap-\nablam$ the line bundle $\CO_X(\nabla)=\CO_X(\nablap)\otimes\CO_X(\nablam)^{-1}$.

\medskip

Conversely, starting with a (toric) divisor $D$, we consider the {\em section polyhedron} 
\begin{equation}
\label{eq:H0R}
\gH^0_\R\!\big(\hspace{-3pt}\sum_{\rho\in\Sigma(1)}a_\rho D_\rho\big):=\big\{u\in M_\R\mid\langle u,\rho\rangle\geq-a_\rho,\,\rho\in\Sigma(1)\big\}.   
\end{equation}
The integral points $\gH^0_\R(D)\cap M$ correspond to a $\kk$-basis 
$\{x^m\}$ of global sections of $\CO_X(D)$ whence the notation. 
Furthermore, $\gH^0_\R(D_\nabla)=\nabla$ for an honest, i.e., non-virtual polyhedron. In particular, the restriction to nef divisors $\gH^0_\R\colon\opn{Nef}(\Sigma)\to\Pol^+(\Sigma)$ defines a semigroup isomorphism by~\eqref{eq:Tensor}. By quasi-projectivity, the Grothendieck group of $\opn{Nef}(\Sigma)$ is $\Div(\Sigma)$ so that $\gH^0_\R|_{\opn{Nef}}$ extends to a group isomorphism 
\[
\opn{P}\colon\Div(\Sigma)\to\Pol(\Sigma)
\]
distinct from $\gH^0_\R$ outside $\opn{Nef}(\Sigma)$. Writing a Cartier divisor $D=D_+-D_-$ as the difference of two nef Cartier divisors, we obtain the corresponding virtual polytope $\opn{P}(D)$ as the formal difference $\nabla_+-\nabla_-$ of the compatible polyhedra $\nabla_\pm=\opn{P}(D_\pm)=\gH^0_\R(D_\pm)$. Moreover, we recover the section polytope of $D$ by
\[
\gH^0_\R(D)=(\nabla_+:\nabla_-):=\{u\in M_\R\mid u+\nabla_-\subseteq\nabla_+\},
\]
which is simply a special case of~\eqref{eq:AP} for $\ell=0$.

\begin{remark}
Given two virtual polytopes $\nabla=\nabla_+-\nabla_-$ and $\nabla'=\nabla'_+-\nabla'_-$, we can always pass to a ``common denominator'' by writing $\nabla=(\nabla_++\nabla'_-)-(\nabla_-+\nabla'_-)$ and $\nabla'=(\nabla'_++\nabla_-)-(\nabla'_-+\nabla_-)$. 
\end{remark}

\begin{remark}
Virtual polyhedra are best visualised locally by fixing a $\sigma\in\Sigma(\kdimX)$. Namely, the function $\Pol^+(\Sigma)\to M$ given by assigning $\nabla$ to its vertex $\nabla(\sigma)$ is additive and thus extends to $\Pol(\Sigma)$. Writing $\nabla=\nabla_+-\nabla_-$, we can express the local polyhedra of $\nabla$ inside $M_\R$ by $\nabla_\sigma=\nabla(\sigma)+\sigma^\vee=\big(\nabla_+(\sigma)-\nabla_-(\sigma)\big)+\sigma^\vee$. 
\end{remark}

\subsubsection{The lattice structure}
\label{sec:lattice-structure}
The usual relation 
\[
D\leq D'\quad\text{if and only if}\quad D'-D\text{ is effective}
\]
turns $\Div(\Sigma)$ into a {\em meet semilattice} with meet
\[
\big(\sum a_\rho D_\rho\big)\wedge\big(\sum a'_\rho D_\rho\big):=\sum\min(a_\rho,a'_\rho)D_\rho.
\]
If we let $\Pol^+(M_\R)$ be the set of all polyhedra (which comprises also the empty set), intersection defines a meet for the partial order given by inclusion, and
\[
\gH^0_\R\colon\big(\Div(\Sigma),\leq,\wedge\big)\to
\big(\Pol^+(M_\R),\subseteq,\cap\big),
\]
is actually a {\em meet semilattice morphism}, that is,
\[
\gH^0_\R(D'\wedge D'')=\gH^0_\R(D')\cap\gH^0_\R(D''). 
\]

The identity
\[
(D'\wedge D'')+D=(D'+D)\wedge(D''+D)
\]
implies that upon adding a suitable ample divisor $D$, the divisors $D'_+=D'+D$ and $D''_+=D''+D$, as well as their meet $D'_+\wedge D''_+$, are nef. Therefore,
\[
\opn{P}(D'_+\wedge D''_+)=\gH^0_\R(D'_+\wedge D''_+)=\gH^0_\R(D'_+)\cap\gH^0_\R(D''_+)=\opn{P}(D'_+)\cap\opn{P}(D''_+).
\]
Since
\[
\opn{P}(D'\wedge D'')=\opn{P}\big((D'_+\wedge D''_+)-D\big)=\opn{P}(D'_+\wedge D''_+)-\opn{P}(D),
\]
a sufficiently ample $\Delta$ yields
\[
\opn{P}(D'\wedge D'')=(\nabla'_+\cap\nabla''_+)-\Delta=\big[(\nabla'+\Delta)\cap(\nabla''+\Delta)\big]-\Delta
\]
and defines thus the virtual intersection $\nabla'\cap\nabla''$. This was written $\vcap$ in the introduction but we shall slightly abuse notation and denote the virtual intersection also by $\cap$ in an effort to keep notation light. In particular, we always have
\[
\opn{P}(D'\wedge D'')=\opn{P}(D')\cap\opn{P}(D''). 
\]

\begin{remark}
\label{rem:Positivity}
For a given projective variety $X=\PP(\Delta)$ we only use the normal fan $\Sigma=\mc N(\Delta)$ and the existence of ample polyhedra, but not the polarisation. In particular, we are free to replace $\Delta$ by any other ample polyhedron and to assume that $\Delta$ is sufficiently ample for any finitely many given polyhedra at a time. 
\end{remark}

Table~\ref{table:DivPol} summarises the relationship between divisors, invertible sheaves and virtual polyhedra (with $\subseteq$ and $\cap$ on $\Pol(\Sigma)$ understood virtually).
\begin{table}[h]
\begin{tabular}{c|c|c}
$\opn{Div}(\Sigma)$ & invertible sheaves $\subseteq(j_\TX)_*\CO_\TX=\kk[M]$ & $\Pol(\Sigma)$\\\hline\vspace{-5pt}&&\\
$D\leq D'$&$\CO_X(D)\subseteq\CO_X(D')$&$\opn{P}(D)\subseteq\opn{P}(D')$\\[3pt]
$D+D'$&$\CO_X(D)\cdot\CO_X(D')=\CO_X(D)\otimes\CO_X(D')$&$\opn{P}(D)+\opn{P}(D')$\\[3pt]
$-D$&$\CO_X(D)^{-1}=\mc{H}om_{\CO_X}(\CO_X(D),\CO_X)$&$-\opn{P}(D)$\\[3pt]
$D\wedge D'$&$\CO_X(D)\cap\CO_X(D')$& $\opn{P}(D)\cap\opn{P}(D')$\\
\end{tabular}
\caption{}
\label{table:DivPol}
\end{table}
It follows that $(\Pol(\Sigma),\cap,\subseteq)$ is again a meet semilattice, and $\opn{P}\colon\Div(\Sigma)\to\Pol(\Sigma)$ preserves the semilattice structure. As a result, we tacitly identify divisors and their sheaves with virtual polyhedra.

\section{Weil decorations and toric sheaves}
\label{sec:WDTS}
We let $X=\toric(\Sigma)$ be an $n$-dimensional toric variety.
Vector spaces over $\kk$ will be endowed with the Zariski topology.

\subsection{Weil decorations}
We extend the meet semilattices $\Div(\Sigma)$ and $\Pol(\Sigma)$ by formal upper bounds $(\infty)$ and $M_\R$, respectively, and let
\[
\widehat\Div(\Sigma)=\Div(\Sigma)\cup\{(\infty)\},\quad\widehat\Pol(\Sigma)=\Pol(\Sigma)\cup\{M_\R\}.
\]
Extending the isomorphism $\opn{P}\colon\Div(\Sigma)\to\Pol(\Sigma)$ accordingly, we set $\opn{P}(\infty):=M_\R$.

\begin{definition}
\label{def:WeilDec}
(i) A {\em Weil decoration} on a $\kk$-vector space $E$ is a map
\[
\CD\colon E\to\widehat\Div(\Sigma) 
\]
such that
\begin{itemize}
\item $\CD(e)=(\infty)$ if and only if $e=0$, and $\CD|_{E\setminus\{0\}}$ factorises over the projectivisation $\P(E)$ of $E$;

\smallskip

\item for all $e$, $e'\in E$, the inequality
\begin{equation}
\label{eq:Basic} 
\CD(e+e')\geq\CD(e)\wedge\CD(e')
\end{equation}
holds true.
\end{itemize}

(ii) A {\em morphism $\CD\to\CD'$} of Weil decorations on vector spaces $E$ and $E'$ is a linear map $\lambda\colon E\to E'$ such that
\[
\CD(e)\leq\CD'(\lambda(e))
\]
for all $e\in E$.
\end{definition}

\subsubsection{Strata}
Recall from Section~\ref{sec:Intro} that 
\[
\S(D)=\CD^{-1}(D)=\{e\in E\mid\CD(e)=D\}
\]
for $D\in\widehat\Div(\Sigma)$. In particular, $\S(D)\not=\varnothing$ if and only if $D\in\CD(E)$.
Further, we set
\[
E_{\geq D}:=\{e\in E\mid\CD(e)\geq D\}\quad\text{and}\quad E_{>D}:=\{e\in E\mid\CD(e)>D\}.
\]
By Inequality~\ref{eq:Basic}, $E_{\geq D}$ is a vector subspace of $E$, while 
\[
E_{>D}=\bigcup_{\rho\in\Sigma(1)}E_{\geq D+D_\rho}
\]
is a finite union of vector subspaces.

\begin{proposition}
\label{prop:DE}
Let $\CD\colon E\to\widehat\Div(\Sigma)$ be a Weil decoration. Then the image $\CD(E)\subseteq\widehat\Div(\Sigma)$ 
is a finite sublattice.
\end{proposition}

\begin{proof}
First, we show that the image is finite. Let $D\in\Div(\Sigma)$. If $E_{\geq D}\subsetneq E$, we can pick $e'\in E\setminus E_{\geq D}$ and consider $E_{\geq D\wedge\CD(e')}$ which contains both $E_{\geq D}$ and $e'$. Since $E$ is finite dimensional, this process must eventually stop. Differently put, there exists a divisor $D$ with $E_{\geq D}=E$, that is, $\CD(e)\geq D$ for all $e\in E$.

\medskip

Furthermore, increasing in $D$ the coefficients of the prime divisors $D_\rho$ if necessary, we may assume that $D$ is maximal with the property that $E_{\geq D}=E$. In particular, $E_{>D}$ is a finite union of proper vector subspaces of $E$. The disjoint decomposition
\[
E=E_{\geq D}=\S(D)\sqcup E_{>D}
\]
implies $\S(D)\not=\varnothing$ whence $D\in\CD(E)$. An induction argument yields the claim. 

\medskip

Second, we show that $\CD(E)$ is closed under $\wedge$. Let $e$, $e'\in E\setminus\{0\}$. By the first part of the proof adapted to $E'=\mr{span}_\kk(e,e')$, we deduce that there is an open set $U$ in $E'$ for which $\CD(\cdot)$ is minimal. Let $v$ and $v'$ be two linearly independent vectors in $U$ so that 
\[
\CD(e),\,\CD(e')\geq\CD(v)=\CD(v').
\]
In particular, $\CD(e)\wedge\CD(e')\geq\CD(v)=\CD(v')$. On the other hand, $\CD(v)$, $\CD(v')\geq\CD(e)\wedge\CD(e')$ by Inequality~\eqref{eq:Basic}, whence the result.
\end{proof}

Actually, the proof yields more. Since the sets $\S(D)$ are locally closed,
\[
\Strat_\CD:=\{\S(D)\mid D\in\CD(E)\}
\]
defines a stratification, subsequently referred to as the {\em canonical stratification}. Further, $\Strat_\CD$ carries the natural partial order
\begin{equation}
\label{eq:StratPOS}
\S\leq\S'\quad\text{if and only if}\quad\S\subseteq\overline{\S'}.
\end{equation} 
By design, $\S(D)\leq\S(D')$ if and only if $D\geq D'$. Additionally, the {\em generic stratum}
\[
\eta_\CD:=\S(D)
\]
given by the maximal $D$ appearing in the proof of Proposition~\ref{prop:DE}, defines an upper bound for $\Strat_\CD$. The meet of $\widehat\Div(\CD)$ carries over to a join 
\[
\S(D)\vee\S(D'):=\S(D\wedge D') 
\]
on $(\Strat_\CD,\leq)$ as it is straightforward to check. Summarising, we obtain the

\begin{proposition}
\label{prop:Strat2Div}
The canonical stratification $(\Strat_\CD,\leq,\vee)$ defines a join semilattice with the generic stratum $\eta_\CD$ as upper bound. Furthermore, the Weil decoration $\CD$ descends to an anti-semilattice isomorphism onto the image of 
\[
\CD\colon(\Strat_\CD,\leq,\vee)\to(\widehat\Div(\Sigma),\leq,\wedge),\quad\S(D)\mapsto D, 
\]
that is,
\[
\CD(\S\vee\S')=\CD(\S)\wedge\CD(\S')
\]
for all strata $\S$, $\S'\in\Strat$. 
\end{proposition}

\subsubsection{The Homomorphiesatz}
A converse implication also holds, which we state as follows. Call a finite stratification $\Strat$ of a given $\kk$-vector space $E$ {\em admissible}, if
\begin{itemize}
\item $\{0\}\in\Strat$; it is referred to as the {\em trivial stratum} and simply denoted $0$.

\smallskip

\item $\overline\S\subseteq E$ is a vector subspace for all strata $\S\in\Strat$.

\smallskip

\item $\Strat$ is a join semilattice with respect to the natural partial order~\eqref{eq:StratPOS}.
\end{itemize}

The proof of the following proposition is straightforward and left to the reader.

\begin{proposition}
\label{prop:HomSatz}
Let $(\Strat,\leq,\vee)$ be an admissible stratification. For $e\in E$, let $\S(e)$ be the unique stratum containing $e$. Then every anti-semilattice morphism $\CD_\Strat\colon\Strat\to\widehat{\Div}(\Sigma)$ with the property that $\CD_\Strat(\S)=(\infty)$ if and only if $\S=0$ induces a Weil decoration
\[
\CD\colon E\to\widehat\Div(\Sigma),\quad\CD(e)=\CD_\Strat(\S(e)).
\]
In particular, $\Strat$ is a refinement of the
canonical stratification $\Strat_\CD$ and
\[
\CD_\Strat=\CD\circ[\Strat\to\Strat_\CD]. 
\]
\end{proposition}

\subsection{Toric sheaves}
\label{subsec:TS}
The category of toric sheaves consists of 
$\TX$-equivariant, reflexive $\CO_X$-sheaves with $\TX$-equivariant sheaf morphisms. 
By torsionfreeness of reflexive sheaves, the natural morphism $\CE \to (j_\TX)_*j_\TX^* \CE = (j_\TX)_*\CE|_\TX$ is actually injective. Let 
\[
E:=\text{ the torus invariant subspace of }\CE(\TX)=\Gamma(\TX,\CE)
\]
subsequently referred to as the $\kk$-vector space {\em associated with} $\CE$; in particular, $\CE|_\TX$ is the sheafification of $E\otimes\kk[M]$. Slightly abusing notation we therefore consider as for rank $1$ sheaves 
\[
\CE\subseteq(j_\TX)_*\CE|_\TX=E\otimes(j_\TX)_*\CO_\TX=E\otimes\kk[M]
\]
as a subsheaf of $E\otimes\kk[M]$, cf.\ Equation~\eqref{eq:Subsheaf} on page~\pageref{eq:Subsheaf}.

\begin{remark}
\label{rem:Satur}
A $\TX$-equivariant subsheaf $\CE'$ of a toric sheaf $\CE$ is again reflexive if $\CE'$ is {\em saturated}, that is, $\CE/\CE'$ is torsionfree, see~\cite[II.1.1.16]{oss}.  
\end{remark}

\medskip

For $e\in E\setminus\{0\}$ we define the subsheaf
\begin{equation}
\label{eq:CEe}
\CE(e):=\CE\cap\big(\kk\cdot e \otimes(j_\TX)_*\CO_\TX\big)
\end{equation}
which means that $\CE(e)(U_\sigma)=\CE(U_\sigma)\cap(\kk\cdot e\otimes\kk[M])$. It is reflexive (being saturated in $\CE$) and of rank $1$. The isomorphic subsheaf $\CL(e)$ of $(j_\TX)_*\CO_\TX=\kk[M]$ resulting via
\[
\begin{tikzcd}
\CE(e)\ar[r,hookrightarrow]\ar[d,"\cdot 1/e"',"\cong"]& e\cdot\kk[M]\ar[d,"\cdot 1/e","\cong"']\\
\CL(e)\ar[r,hookrightarrow]&\kk[M]
\end{tikzcd}
\]
induces a well-defined Weil divisor $D(e)$ with $\CO_X(D(e))=\CL(e)$.

\begin{lemma}
\label{lem:Basic}
For $e$, $e'\in E\setminus\{0\}$ with $e+e'\not=0$ we have $\CL(e)\cap\CL(e')\subseteq\CL(e+e')$ as subsheaves of $\kk[M]$. 
\end{lemma}

\begin{proof}
Let $f\in\CL(e)\cap\CL(e')$. Then $ef\in\CE(e)\subseteq\CE$ and $e'f\in\CE(e')\subseteq\CE$ whence $(e+e')f\in\CE$, or equivalently, $f\in\CL(e+e')$.
\end{proof}

In other words, $D(e)\wedge D(e')\leq D(e+e')$ so that
\[
\CD_\CE(e):=D(e),\,e\in E\setminus\{0\},\quad\text{and}\quad\CD_\CE(0):=(\infty)
\]
defines the Weil decoration {\em associated with $\CE$}. In fact, we can recover $\CE$ from $\CD_\CE$, for instance via the Klyachko filtrations~\eqref{eq:KlyachkoFilt} on page~\pageref{eq:KlyachkoFilt}. Since the filtrations are preserved by morphisms of Weil decorations, the latter correspond to morphisms of toric sheaves.

\subsection{Examples}
\label{subsec:Exam}
Further examples of Weil decorations are in order, cf.\ also Examples~\ref{exam:LB} and~\ref{exam:SumLB} from Section~\ref{sec:Intro}.

\begin{example}
The trivial toric sheaf $0$ is associated with $E=0$ and Weil decoration $\CD_0(0)=(\infty)$.
\end{example}

\begin{example}
\label{exam:TPLB}
If $\CE$ is reflexive and $\CO_X(D)$ invertible, then 
\[
\CE(D):=\CE\otimes\CO_X(D)
\]
is again reflexive \cite[Proposition 1.1]{hartreflexiv}. The associated vector space is the one of $\CE$, and its associated Weil decoration is $\CD_{\CE(D)}=\CD_\CE+D$.
\end{example}

\begin{example}
\label{exam:SumTRS}
Consider the direct sum $\CE=\CE'\oplus\CE''$ of two toric sheaves $\CE'$ and $\CE''$ associated with the vector spaces $E'$ and $E''$, respectively. Then $E=E'\oplus E''$ is associated with $\CE$. We obtain nontrivial maps $\CL(e'\oplus e'')\to\CL(e')$ and $\CL(e'\oplus e'')\to\CL(e'')$ whence $D(e'\oplus e'')\leq D(e')$, $D(e'')$ whenever $e'\in E'$ and $e''\in E''$ are nontrivial. From Lemma~\ref{lem:Basic}, we conclude that 
\[
\CD_\CE(e'\oplus e'')=\CD_{\CE'}(e')\wedge\CD_{\CE''}(e'')
\]
defines the Weil decoration associated with $\CE$.
\end{example}

\begin{remark}
\label{rem:Factor}
For $\CE=\CO_X(D_1)\oplus\CO_X(D_2)$, the stratification given by $0$, $\S(D_1)=\kk^\times\cdot e_1$, $\S(D_2)=\kk^\times\cdot e_2$ and $\eta_\CE=\S(D_1\wedge D_2)$ yields the canonical stratification if $D_1\not=D_2$. In case of equality, this stratification is not the canonical one. The latter merely consists of the trivial and the generic stratum. 
\end{remark}

\begin{example}
\label{exam:HomEF} 
Let $\CE$, $\CF$ be toric sheaves with Weil decorations $\CD_\CE$ and $\CD_\CF$. The sheaf hom 
\[
\mc H:={\mc{H}om}_{\mc O_X}(\CE,\CF) 
\]
is reflexive, for so is $\CF$. Namely, $\CF=\mc G^\vee$, e.g., $\mc G=\mc F^\vee$, so that
\[
\mc H={\mc{H}om}_{\mc O_X}(\CE,\mc G^\vee)={\mc{H}om}_{\CO_X}(\CE\otimes_{\CO_X}\mc G,\CO_X)
\]
which is reflexive by definition. Let $E$ and $F$ be the vector spaces associated with $\CE$ and $\CF$, respectively. Then
\[
{\mc{H}om}_{\CO_X}(\CE,\CF)|_\TX={\mc{H}om}_{\CO_\TX}(\CE|_\TX,\CF|_\TX)
\]
which corresponds to the free $\kk[M]$-module 
\[
\gHom_{\kk[M]}(E\otimes\kk[M],F\otimes\kk[M])=\gHom_\kk(E,F)\otimes_\kk\kk[M].
\]
Therefore, $H:=\gHom_\kk(E,F)$ is the vector space associated with $\mc H$. To compute the associated Weil decoration $\CD_\CH$, we let $\varphi\colon E\to F$ in $H$ and consider 
$\mc H(\varphi)=\mc H\cap(\kk\varphi\otimes\kk[M])$. 
Then $\mc H(\varphi)$ contains $\varphi\cdot\CO_X(D)$ for a divisor $D\in\Div(\Sigma)$, that is,
$\varphi\cdot\mc O_X(D)\subseteq\mc H(\varphi)$, if and only if
\[
\varphi\cdot\CO_X(D)\big(\mc E(e)\big)=\varphi(e)\cdot\CO_X(D)\cdot\CO_X(\CD_\CE(e))\subseteq\CF(\varphi(e))=\varphi(e)\cdot\CO_X\big(\CD_\CF(\varphi(e))\big).
\]
This, in turn, is equivalent to $D+\CD_\CE(e)\leq\CD_\CF(\varphi(e))$ or $D\leq\CD_\CF(\varphi(e))-\CD_\CE(e)$ for all $e\in E$. It follows that 
\[
\CD_\mc H(\varphi)=\bigwedge_{e\in E}\big(\CD_\mc F(\varphi(e))-\CD_\CE(e)\big).
\]
In particular, the decoration of the dual sheaf $\CE^\vee$ is given by
\[
\CD_{\CE^\vee}(\varphi)=\bigwedge_{\varphi(e)\not=0}-\CD_\CE(e).
\]
We shall recover this result in Remark~\ref{rem:WDDual} by different methods.
\end{example}

\begin{example}
\label{exam:LocFree}
For any toric sheaf $\CE$ we can consider the {\em local Weil decorations} 
\[
\CD_\sigma:=\CD_{\CE,\sigma}:=\CD_{\CE|_{U_\sigma}},\quad\sigma\in\Sigma
\]
obtained from the divisors $\{D(e)|_{U_\sigma}\mid D(e)\in\CD_\CE(E)\}$. Since globally distinct divisors $D(e')\not=D(e'')$ may coincide locally the resulting canonical $\sigma$-stratifications might be coarser than the one of $\CE$. By~\cite{gub86}, $\CE$ is locally free if and only if the local Weil decorations of $\CE$ induce coordinate space stratifications, cf.\ Example~\ref{exam:SumTRS}. 

\medskip

As illustration, consider the tangent sheaf of $\P^2$ whose fan is given in Figure~\ref{fig:1stFanP2}. 
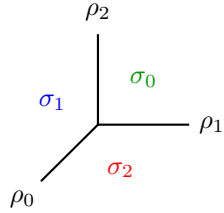
\begin{figure}[ht]
\newcommand{\scaleA}{0.3}
\begin{tikzpicture}[scale=\scaleA]
\draw[thick, color=black]
  (0,0) -- (4,0) (0,0) -- (0,4) (0,0) -- (-2.5,-2.5);
\draw[thick, color=green]
  (2,2) node{$\sigma_0$};
\draw[thick, color=blue]
  (-2,1) node{$\sigma_1$};
\draw[thick, color=red]
  (1,-2) node{$\sigma_2$};
\draw[thick, color=black]
  (5,0) node{$\rho_1$} (0,5) node{$\rho_2$} (-3.3,-3.3) node{$\rho_0$};
\end{tikzpicture}
\caption{The fan of $\PP^2$.}
\label{fig:1stFanP2}
\end{figure} 

As we show in Example~\ref{exam:TB} below, its global canonical stratification is
\[
\begin{tikzcd}
& \eta=\S(0)& \\
\S(D_{\rho_0})=\kk^\times\cdot\rho_0\ar[ru, no head] & \S(D_{\rho_1})=\kk^\times\cdot\rho_1 \ar[u, no head] & \S(D_{\rho_2})=\kk^\times\cdot\rho_2\ar[lu, no head].
\end{tikzcd}
\]
Locally, say over $U_{\sigma_0}$, we find the canonical stratification
\[
\begin{tikzcd}
& \eta=\S(0)& \\
\S(\div(x^{(1,0)}))=\kk^\times\cdot\rho_1\ar[ru, no head] & & \S(\div(x^{(0,1)}))=\kk^\times\cdot\rho_2\ar[lu, no head]
\end{tikzcd}
\]
for $D_{\rho_0}|_{U_{\sigma_0}}=0$; similarly for $U_{\sigma_1}$ and $U_{\sigma_2}$. 
\end{example}

\section{Morphisms of toric sheaves}
\label{sec:TorMor}
From now on we assume that 
\[
X=\PP(\Delta)\text{ \bf is a smooth and semi-projective toric variety}.  
\]
We recall from Remark~\ref{rem:Positivity} that we are free to replace $\Delta$ by any sufficiently ample polyhedron if necessary. 
Moreover, we tacitly identify $\Div(\Sigma)$ with the virtual polyhedra in $\Pol(\Sigma)$ and consider Weil decorations as $\Pol(\Sigma)$-valued. We continue to use the letter $\CD$; in particular $\CD(e)$ can be a divisor or a (virtual) polyhedron.

\medskip

We refer to a Weil decoration taking values in $\Pol^+(\Sigma)$ as {\em positive} and to the corresponding toric sheaf as {\em positively decorated}. Differently put, the image of the Weil decoration consists of nef divisors and $(\infty)$. Passing to a common denominator as in Section~\ref{subsec:VP} allows us to write
\[
\CD_\CE=\CD^+_\CE-\Delta
\]
for a positive Weil decoration
\[
\CD^+_\CE\colon E\to\widehat{\Pol^+}(\Sigma)
\]
associated with the toric sheaf
\[
\CE^+=\CE(\Delta)=\CE\otimes\CO_X(\Delta).
\]
We call $\CD^+_\CE:=\CD_{\CE^+}$ a {\em materialisation} of $\CD_\CE$. We think of a materialisation as passing to numerators of virtual polyhedra; we choose a common denominator $\Delta$ with corresponding materialisations when working with finitely many Weil decorations at a time. Observe that $\cap$ and $\subseteq$ acquire their usual meaning for sufficiently positive Weil decorations.

\medskip

For $\sigma\in\Sigma(n)$, the local Weil decorations $\CD_\sigma=\CD_{\CE|_{U_\sigma}}$ of $\CE$ are always positive as
\begin{equation}
\label{eq:LocParl}
\CPvar_\sigma(e)=\CPvar(e)+\sigma^\vee=r_\sigma(e)+\sigma^\vee\in\widehat{\Pol^+}(\sigma)
\end{equation}
where $\div(x^{-r_\sigma(e)})=\CD_\CE(e)|_{U_\sigma}$ and $\sigma$ is regarded as the fan given by its faces. The polyhedra of the local Weil decorations $\CPvar^+_\sigma(e)=r^+_\sigma(e)+\sigma^\vee$ coming from a materialisation are simply obtained by shifting $\CPvar_\sigma(e)$ by $r^+_\sigma(e)-r_\sigma(e)$. Note, however, that while $\CPvar(e)=\bigcap_{\sigma\in\Sigma}\CPvar_\sigma(e)$ only holds virtually,  
\begin{equation}
\label{eq:CP+}
\CPvar^+(e)=\bigcap_{\sigma\in\Sigma}\CPvar^+_\sigma(e)=\bigcap_{\sigma\in\Sigma(\kdimX)}\CPvar^+_\sigma(e)
\end{equation}
is an actual intersection for any materialisation, cf.~\eqref{eq:NablaInter}. In passing we remark that $r^+_\sigma(e)=\CD^+(e)(\sigma)$ and $r^+_\sigma(e)-r_\sigma(e)=\Delta(\sigma)$, following the notation of~\eqref{eq:sigmaCorner}. As a consequence, the local Weil decorations are positive for any $\sigma\in\Sigma$.

\medskip

Let $\CE$ be a toric sheaf and $\CE^+=\CE(\Delta)$ be associated with a materialisation. For $u\in M_\R$ and $\sigma\in\Sigma$ we define the {\em local evaluation spaces} by
\[
\eval_\sigma(u):=\{e\in E\mid u\in\CPvar_\sigma(e)\}\quad\text{and}\quad\eval^+_\sigma(u):=\{e\in E\mid u\in\CPvar^+_\sigma(e)\}.
\]
In particular, $\eval^+_\sigma(u)=\eval_\sigma(u+r_\sigma(e)-r^+_\sigma(e))$. Moreover, we define the {\em global evaluation spaces} 
\begin{equation}
\label{eq:GlobEval}
\eval(u):=\bigcap_{\sigma\in\Sigma}\eval_\sigma(u)\quad\text{and}\quad\eval^+(u):=\bigcap_{\sigma\in\Sigma}\eval^+_\sigma(u)=\{e\in E\mid u\in\CPvar^+(e)\}.
\end{equation}
Note that the latter set makes only sense for a positively decorated sheaf. If $\CE$ was already positively decorated, we may simply put $\CE=\CE^+$ so that $\eval(u)=\eval^+(u)$.

\medskip

For instance, the local evaluation spaces of $\CE=\CO_X(\nabla)$, $\nabla\in\Pol(\Sigma)$, are
\begin{equation}
\label{eq:EvalLB}
\eval_\sigma(u)=\left\{\begin{array}{cc}\kk&u\in\nabla_\sigma\\0&u\not\in\nabla_\sigma\end{array}\right..
\end{equation}
If, in addition, $\nabla\in\Pol^+(\Sigma)$, then $\eval(u)=\bigcap_{\sigma\in\Sigma}\eval_\sigma(u)=\kk$ if and only if $u\in\nabla$. 

\medskip

In fact, for integral $m\in M$ we have $m\in\CPvar_\sigma(e)$ if and only if $e\otimes x^m$ is a section of $\CE(e)|_{U_\sigma}$, cf.~\eqref{eq:CEe} on page~\pageref{eq:CEe}. Therefore, 
\begin{equation}
\label{eq:EvalSec}
\eval_\sigma(m)=\gH^0(U_\sigma,\CE)_m
\end{equation}
(the subscript $m$ indicating the graded piece of degree $m$); if, in addition, $\CE$ is positively decorated, then
\begin{equation}
\label{eq:GlobSec}
\eval(m)=\gH^0(X,\CE)_m.
\end{equation}

\begin{remark}
\label{rem:VS}
The evaluation spaces are vector spaces for all $u\in M_\R$: Any local Weil decoration yields $\eval_\sigma(u)=\overline\S$ where $\S$ is the maximal stratum for $u\in\CPvar_\sigma(\S)$. Similarly for the global evaluation spaces, see also~\ssect{subsec:EulerGlobSec}.
\end{remark}

As we have already noticed in \ssect{subsec:TS}, a toric morphism $\lambda\colon\CE'\to\CE$ corresponds to a morphism of Weil decorations. As a result we get an induced linear map $E'\to E$ we keep on denoting $\lambda$, such that $\CPvar'(e')\subseteq\CPvar(\lambda(e'))$. In particular, we obtain a map between the corresponding local evaluation spaces $\lambda\colon\big(\eval'_\sigma(m)\subseteq E'\big)\to\big(\eval_\sigma(m)\subseteq E\big)$.

\begin{proposition}
\label{prop:Iso}
The following are equivalent:
\begin{enumerate}
\item[{\rm(i)}] The sequence of toric sheaves
\begin{equation}
\label{eq:ExSeq}
\begin{tikzcd}
0\ar[r]&\CE'\ar[r,"\lambda"]&\CE\ar[r,"\mu"]&\CE''\ar[r]&0
\end{tikzcd}
\end{equation}
is exact.

\smallskip

\item[{\rm(ii)}] For all $m\in M$ and $\sigma\in\Sigma$, the induced sequences of $\kk$-vector spaces
\[
0\to\eval'_\sigma(m)\stackrel{\lambda}{\longrightarrow}\eval_\sigma(m)\stackrel{\mu}{\longrightarrow}\eval''_\sigma(m)\to0 
\]
are exact.

\smallskip

\item[{\rm(iii)}] For every simultaneous materialisation of $\CE'$, $\CE$ and $\CE''$, and all $m\in M$, the induced sequences of $\kk$-vector spaces 
\[
\begin{tikzcd}
0\ar[r]&\eval'^+(m)\ar[r,"\lambda"]&\eval^+(m)\ar[r,"\mu"]&\eval''^+(m)\ar[r]&0  
\end{tikzcd}
\]
are exact.
\end{enumerate}
\end{proposition}

\begin{proof}
(i) $\Leftrightarrow$ (ii) This is just a consequence of~\eqref{eq:EvalSec}.

\smallskip

(ii) $\Rightarrow$ (iii) follows from~\eqref{eq:GlobEval} and the injectivity of $\lambda$.

\smallskip

(iii) $\Rightarrow$ (i): Since $\CE'^+$, $\CE^+$ and $\CE''^+$ are globally generated, exactness of the global evaluation spaces in conjunction with~\eqref{eq:GlobSec} yields stalkwise exactness of~\eqref{eq:ExSeq}.
\end{proof}

\begin{remark}
\label{rem:EntireSpace}
In particular, we can speak about short exact sequences of Weil decorations though the category is not abelian. Also note that (ii) applied to $\sigma=0$ means exactness of the underlying vector space sequence, that is, $0\to E'\to E\to E''\to0$ is exact. 
\end{remark}

\begin{remark}
By (i) $\Rightarrow$ (iii), exactness of $0\to\CE'\to\CE\to\CE''\to0$ implies exactness of the global sections {\em whenever the toric sheaves $\CE'$, $\CE$ and $\CE''$ are positively decorated}, but the implication might fail for exact sequences without zeroes to the left and the right, that is, $\CE'\to\CE\to\CE''$. 

\smallskip

However, taking a {\em sufficiently positive} twist of the exact sequence $\CE'\to\CE\to\CE''$, the standard short exact sequences associated with the morphisms $\lambda$ and $\mu$ immediately imply exactness of the sequences of global sections $\eval'^+(m)\to\eval^+(m)\to\eval''^+(m)$. The equivalence (i) $\Leftrightarrow$ (ii) and the implication (iii) $\Rightarrow$ (i) hold at any rate.
\end{remark}

\begin{proposition}
\label{prop:Sur}
A toric morphism $\mu\colon\CE\to\CE''$ between positively decorated toric sheaves is surjective if and only if for all $e''\in E''$, we have 
\[
\CPvar''(e'')=\bigcup_{\mu(e)=e''}\CPvar(e). 
\]
\end{proposition}

\begin{proof}
We first observe that functoriality implies $\CPvar(e)\subseteq\CPvar''(\mu(e))$ whence
\[
\bigcup_{\mu(e)=e''}\CPvar(e)\subseteq\CPvar''(e''). 
\]
Further, surjectivity of $\mu$ is equivalent to surjectivity of $\mu\colon\eval^+(m)\to\eval''^+(m)$ for all $m\in M$, cf.\ Proposition~\ref{prop:Iso}.

\smallskip

For the implication $(\Rightarrow)$, take $m\in\CPvar''(e'')$ so that $e''\in\eval''^+(m)$. Then there exists $e\in\mu^{-1}(e'')\cap\eval^+(m)$ which implies $m\in\CPvar(e)$. Consequently,
\[
\CPvar''(e'')\cap M=\bigcup_{\mu(e)=e''}\CPvar(e)\cap M\subseteq\bigcup_{\mu(e)=e''}\CPvar(e)\subseteq\CPvar''(e'').
\]
Every vertex $\CPvar''(e'')(\sigma)$ of the polyhedron $\CPvar''(e'')$ is contained in $M$ and therefore must be a vertex of some $\CPvar(e)$, that is,
$\CPvar(e)_\sigma=\CPvar''(e'')(\sigma)+\sigma^\vee=\CPvar''(e'')_\sigma$ with $\mu(e)=e''$. Taking the convex hull on both sides yields equality. 

\medskip

For the converse $(\Leftarrow)$, assume that $\CPvar''(e'')=\bigcup_{\mu(e)=e''}\CPvar(e)$. If $e''\in\eval''^+(m)$, then $m\in\CPvar''(e'')\cap M$. Hence there exists $e\in E$ with $\mu(e)=e''$ and $m\in\CPvar(e)$, that is, $\eval^+(m)\to\eval''^+(m)$ is surjective.
\end{proof}

\begin{example}
\label{exam:TB}
Consider the dual toric Euler sequence
\begin{equation}
\label{eq:DES}
\begin{tikzcd}
0\ar[r]&\opn{Cl}(X)^\vee\otimes\CO_X\ar[r]&\fbox{$\CE:=\bigoplus_{\rho\in\Sigma(1)}\CO_X(D_\rho)$}\ar[r,"\mu"]&\mc T_X\ar[r]&0.
\end{tikzcd}
\end{equation}
At the level of associated vector spaces, we get
\[
\begin{tikzcd}
0\ar[r]&\opn{Cl}(X)^\vee\otimes_\Z\kk\ar[r]&\bigoplus_{\rho\in\Sigma(1)}\kk\cdot e_\rho\ar[r,"\mu"]&N_\kk\ar[r]&0,
\end{tikzcd}
\]
where $\mu$ is given by linear extension of $e_\rho\mapsto\rho$; both sequences are induced by~\eqref{eq:MotherSeq}. Since $D_\rho\wedge D_{\rho'}=0$ for all $\rho\not=\rho'\in\Sigma(1)$, the Weil decoration of $\CE$ is 
\[
\renewcommand{\arraystretch}{1.3}
\CD_\CE(e)=\begin{cases}D_\rho&\text{if }e\in\opn{span}(\rho)\\0&\text{if otherwise}\end{cases}
\]
for $e\not=0$. By Proposition~\ref{prop:Sur}, we find for $v\in N_\kk\setminus\{0\}$
\[
\renewcommand{\arraystretch}{1.3}
\CD_{\mc T_X}(v)=\begin{cases}D_\rho&\text{if }v\in\opn{span}(\rho)\text{ and }-\rho\not\in\Sigma(1),\\D_\rho+D_{-\rho}&\text{if }v\in\opn{span}(\rho)\text{ and }-\rho\in\Sigma(1),\\0&\text{if otherwise}.\end{cases}
\]
Provided $\Delta$ has no parallel facets, this yields the Hasse diagram
\[
\begin{tikzcd}
& \mathllap{\eta_{\mc T_X}=\ } \S(D_{\rho_i})\vee\S(D_{\rho_j})\mathrlap{,\,i\not=j}& \\
\S(D_{\rho_1})=\kk^\times\cdot\rho_1\ar[ru, no head] & \ldots & \S(D_{\rho_m})=\kk^\times\cdot\rho_m\ar[lu, no head]
\end{tikzcd}
\]
for the canonical stratification. At any rate, via~\eqref{eq:KlyachkoFilt} we recover the filtration 
\[
\renewcommand{\arraystretch}{1.3}
(N_\kk)_\rho^\ell=\begin{cases}N_{\kk}&\mbox{if }\ell\leq0\\
\opn{span}(\rho)&\mbox{if }\ell=1\\
0&\mbox{if }\ell\geq2
\end{cases}
\]
arising in Klyachko's formalism.
\end{example}

\begin{remark}
\label{rem:CanSeq}
For a toric sheaf $\CE$ and strata $\S$, $\T\in\Strat_\CE\setminus\{0\}$, we let
\begin{equation}
\label{eq:rem:CanSeq}
\CO_X(\S):=\CO_X\big(\CD_\CE(\S)\big)\quad\text{and}\quad\CE_{\S,\T}=\overline\S\otimes_\kk\CO_X(\T).
\end{equation}
Materialising if necessary, we may assume that $\CE$ and thus
\[
\widehat{\CE}:=\bigoplus\limits_{0\not=\S\in\Strat_\CE}\CE_{\S,\S}=\bigoplus\limits_{0\not=\S\in\Strat_\CE}\overline\S\otimes\CO_X(\S)  
\]
are positively decorated. Proposition~\ref{prop:Sur} entails the canonical surjection
\begin{equation}
\label{eq:CanSur}
\begin{tikzcd}
\widehat\CE\ar[r,"\pi"]&\CE\ar[r]&0\,, 
\end{tikzcd}
\end{equation}
where the toric morphism $\pi$ is induced by the natural linear map $\pi\colon\bigoplus_{\S\in\Strat_\CE}\overline\S\to E$. Indeed, if $e\in E$, then there exists a unique stratum $\S$ with $e\in\S$ and
\[
\CD_{\CE_{\S,\S}}(e)=\CD_{\CE}(e)=\CD_{\CE}(\S). 
\]
In principle, the Weil decoration of $\CE$ can be thus computed from the Weil decoration of $\widehat\CE$ which was determined in Example~\ref{exam:SumTRS}. The subtle part is to determine the kernel of $\pi$ -- this is precisely where the matroid introduced in~\cite{parliaments} enters the stage.
\end{remark}

On the other hand, injectivity of the toric morphism $\lambda\colon\CE'\to\CE$ is tantamount to the injectivity of the associated linear map $\lambda\colon E'\to E$, cf.\ Remark~\ref{rem:EntireSpace}. In particular, $\CD'(e')\subseteq\CD(\lambda(e'))$. A priori, we can say hardly more. However, if $\lambda$ is saturated (cf.\ Remark~\ref{rem:Satur}), we have the

\begin{proposition}
\label{prop:Inj}
Let $\CE'=\ker\mu$ be the kernel of a morphism $\mu\colon\CE\to\CE''$ of positively decorated toric sheaves. If $E'$ is the associated vector space of $\CE'$
, then
\[
\CPvar_{\CE'}(e')=\CPvar_\CE(\lambda(e'))
\]
for all $e'\in E'$, that is, $\CD_{\CE'}=\CD_\CE|_{E'}$.
\end{proposition}

\begin{proof}
We have an exact sequence $0\to\CE'\stackrel{\lambda}{\to}\CE\stackrel{\mu}{\to}\CE''$, which becomes a short exact sequence by replacing $\mu$ with $\CE \twoheadrightarrow \im(\mu)$. It follows from Proposition~\ref{prop:Iso} that $\eval_\CE(m)\cap\ker\mu=\lambda(\eval_{\CE'}(m))$ which means that for $e'\in E'$, we have $m\in\CPvar_\CE(\lambda(e'))$ if and only if $m\in\CPvar_{\CE'}(e')$.
\end{proof}

\begin{example}
Consider the original toric Euler sequence
\[
\begin{tikzcd}
0\ar[r]&\Omega^1_X\ar[r,"\lambda"]&\fbox{$\CE^\vee=\bigoplus_{\rho\in\Sigma(1)}\CO_X(-D_\rho)$}\ar[r]&\opn{Cl}(X)\otimes\CO_X\ar[r]&0
\end{tikzcd}
\]
and regard Weil decorations as $\Div(\Sigma)$-valued. The exact sequence of associated vector spaces is
\[
\begin{tikzcd}
0\ar[r]&M_\kk\ar[r,"\lambda"]&\bigoplus_{\rho\in\Sigma(1)}\kk\cdot e^\rho\ar[r]&\opn{Cl}(X)\otimes_\Z\kk\ar[r]&0,
\end{tikzcd}
\]
where $\lambda(u)=\sum\langle u,\rho\rangle e^\rho$. By Proposition~\ref{prop:Inj}, the Weil decoration of $\Omega^1_X$ is determined by
\[
\CD_{\Omega^1_X}(u)=\CD_{\CE^\vee}\big(\sum\langle u,\rho\rangle e_\rho\big)=\bigwedge_{\langle u,\rho\rangle\not=0}-D_\rho=-\sum_{\langle u,\rho\rangle\not=0}D_\rho
\]
where $u\in M_\kk\setminus\{0\}$, which induces the filtration
\[
\renewcommand{\arraystretch}{1.3}
(M_\kk)_\rho^\ell:=\begin{cases}
M_{\kk}&\mbox{if }\ell\leq-1\\
\rho^\perp&\mbox{if }\ell=0\\
0&\mbox{if }\ell\geq1
\end{cases}
\]
via~\eqref{eq:KlyachkoFilt}. 
\end{example}

\begin{remark}
\label{rem:WDDual}
From the previous example we get a second way to compute the Weil decoration of the dual $\CE^\vee$ of $\CE$, cf.\ also Example~\ref{exam:HomEF}. Dualising the sequence~\eqref{eq:CanSur} we obtain
\[
\begin{tikzcd}
0\ar[r]&\CE^\vee\ar[r]&\bigoplus_{0\not=\S\in\Strat_\CE}\CO_X(-D(\S))\otimes\overline\S^\vee\to(\ker\pi)^\vee.
\end{tikzcd}
\]
The dual $(\ker\pi)^\vee$ is reflexive by definition. Upon twisting by some $\mc O(\Delta)$ we may assume without loss of generality that all the sheaves are positively decorated so that Proposition~\ref{prop:Inj} applies. The beginning of the induced sequence of associated vector spaces reads
\[
\begin{tikzcd}
0\ar[r]&E^\vee\ar[r]&\bigoplus_{0\not=\S\in\Strat_\CE}\overline\S^\vee, 
\end{tikzcd}
\]
where $\varphi\in E^\vee$ is mapped to the linear functional on $\bigoplus_{0\not=\S\in\Strat_\CE}\overline\S^\vee$ given by $\bigoplus e_\S\mapsto\sum\varphi(e_\S)$, where $e_\S\in\overline\S$. Hence
\[
\CD_{\CE^\vee}(\varphi)=\bigwedge_{\varphi(e_\S)\not=0}\hspace{-5pt}-D(\S)=\bigwedge_{\varphi(e_\S)\not=0}\hspace{-5pt}-\CD_\CE(\S).
\]
\end{remark}

\begin{example}
\label{exam:F1PolRes}
As a further illustration we revisit the exact sequence~\eqref{eq:ExacSeq} from the introduction, namely
\[
0\to\CO_X(\nablap)\to\fbox{$\CE:=\CO_X(\NablaA)\oplus\CO_X(\NablaB)$}\to \CO_X(\nablam)\to0.
\]
This corresponds to the exact sequence of associated vector spaces
\[
\begin{tikzcd}[column sep=huge]
0\ar[r]&\kk e_+\ar[r,"{\lambda=(1,-1)^\top}"]&\kk e_0\oplus\kk e_1\ar[r,"{\mu=(1,1)}"]&\kk e_-\ar[r]&0
\end{tikzcd}
\]
and positive Weil decorations
\[
\CD_+(e_+)=\nablap,\quad\CD_\CE(e_i)=\nabla_i,\,i=0,\,1\quad\text{and}\quad\CD_-(e_-)=\nablam,  
\]
cf.\ Examples~\ref{exam:LB} and~\ref{exam:SumTRS}. We can check exactness by invoking Propositions~\ref{prop:Sur} and Proposition~\ref{prop:Inj}, for we have
\[
\CD_-(e_-)=\nablam=\nabla_0\cup\nabla_1=\CD_\CE(e_0)\cup\CD_\CE(e_1)
\]
and
\[
\CD_+(e_+)=\nablap=\nabla_0\cap\nabla_1=\CD_\CE(e_0)\cap\CD_\CE(e_1)=\CD_\CE(\lambda(e_+))
\]
respectively. See also Figure~\ref{fig:IncExc} which translates Figure~\ref{fig:F1PolRes} to our context.

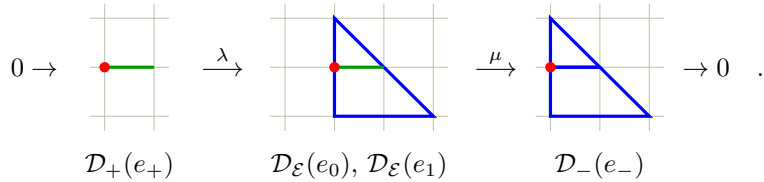
\begin{figure}[ht]
\newcommand{\scaleA}{0.65}
\newcommand{\raiseC}{42pt}
\raisebox{\raiseC}{$0\rightarrow$ }
\begin{tikzpicture}[scale=\scaleA]
\draw[color=oliwkowy!40] (-0.3,-0.3) grid (1.3,2.3);
\draw[very thick, color=green]
  (0,1) -- (1,1);
\fill[thick, color=red]
  (0,1) circle (3pt);
\draw[thick, color=black]
  (0.5,-1) node{$\CD_+(e_+)$};
\end{tikzpicture}
\raisebox{\raiseC}{ $\stackrel{\lambda}{\longrightarrow}$ }
\begin{tikzpicture}[scale=\scaleA]
\draw[color=oliwkowy!40] (-1.3,-0.3) grid (2.3,2.3);
\draw[very thick, color=blue]
  (0,1) -- (0,0) -- (2,0) -- (1,1);
\draw[very thick, color=blue]
  (0,1) -- (0,2) -- (1,1);
\draw[very thick, color=green]
  (0,1) -- (1,1);
\fill[very thick, color=red]
  (0,1) circle (3pt);
\draw[thick, color=black]
  (0.5,-1) node{$\CD_\CE(e_0),\,\CD_\CE(e_1)$};
\end{tikzpicture}
\raisebox{\raiseC}{ $\stackrel{\mu}{\longrightarrow}$ }
\begin{tikzpicture}[scale=\scaleA]
\draw[color=oliwkowy!40] (-0.3,-0.3) grid (2.3,2.3);
\draw[very thick, color=blue]
  (0,0) -- (0,1) -- (1,1) -- (2,0) -- cycle;
\draw[very thick, color=blue]
  (0,1) -- (0,2) -- (1,1) -- cycle;
\fill[thick, color=red]
  (0,1) circle (3pt);
\draw[thick, color=black]
  (1,-1) node{$\CD_-(e_-)$};
\end{tikzpicture}
\raisebox{\raiseC}{ $\rightarrow 0\quad$.}
\caption{The sequence~\eqref{eq:ExacSeq} from a Weil decoration point of view.}
\label{fig:IncExc}
\end{figure}
\end{example}

\section{Extensions of nef line bundles}
\label{sec:ExtNLB}
We keep on assuming that $X$ is smooth and semi-projective. Further, consider two $\CO_X$-modules $\mc F$ and $\mc G$ and their space of extensions 
\[
\gExt(\mc F,\mc G):=\gExt^1_{\CO_X}(\mc F,\mc G).
\]
A {\em universal extension} is a short exact sequence of the form 
\begin{equation}
\label{eq:ExSeqExt}
\begin{tikzcd}
0\ar[r]&\mc G\otimes_\kk\gExt(\mc F,\mc G)^\vee\ar[r]&\CE\ar[r]&\CF\ar[r]&0
\end{tikzcd}
\end{equation}
such that the extension $\CE_t$ induced by pushout along $t\in\gExt(\mc F,\mc G)$, namely the lower sequence in
\[
\begin{tikzcd}
0\ar[r]\ar[d,equal]&\mc G\otimes_\kk\gExt(\mc F,\mc G)^\vee\ar[r]\ar[d,"t"] &
\CE\ar[r]\ar[d]&\CF\ar[r]\ar[d,equal]&0\ar[d,equal]\\
0\ar[r]&\CG\ar[r]&\CE_t\ar[r]&\CF\ar[r]&0
\end{tikzcd}
\quad,
\]
is actually the one specified by $t$. If $\mc F$ and $\mc G$ are toric sheaves, then there is a natural $M$-grading on $\gExt(\mc F,\mc G)$ and we can consider the {\em $\TX$-equivariant} universal extensions obtained by replacing $\gExt(\mc F,\mc G)$ in~\eqref{eq:ExSeqExt} with the degree zero piece $\gExt(\mc F,\mc G)_0$. 

\medskip

For two nef line bundles given by lattice polytopes $\NablaM$, $\NablaP\in\Pol^+(\Sigma)$, the universal toric extension
\[
\Ext(\NablaM,\NablaP)_0:=\Ext\big(\CO_X(\NablaM),\CO_X(\NablaP)\big)_0 
\]
was determined in~\cite{ka48-dispExt}. The basic idea is to use the isomorphism
\begin{equation}
\label{eq:DimExt}
\Ext(\NablaM,\NablaP)_0\cong\widetilde{\opn{H}}^{\raisebox{-3pt}{\scriptsize$0$}}(\NablaM\setminus\NablaP)\cong\kk^s
\end{equation}
from~\cite{immaculate} and~\cite{dop}, where $\widetilde{\opn{H}}^{\raisebox{-3pt}{\scriptsize$0$}}$ denotes reduced cohomology and $s+1$ the number of connected components of $\NablaM\setminus\NablaP$. In particular, they handle the case $\NablaP\subseteq\NablaM$ using inclusion/exclusion sequences, cf.\ Figure \ref{fig:F1PolRes}.

\medskip

We are going to show that this approach can be generalised to any pair of nef polytopes $(\NablaP,\NablaM)$ by using Weil decorations. First, a technical result is in order.

\subsection{Inclusion of polyhedra}
\label{subsec:IncPol}
Let $\Delta$ be an ample lattice polyhedron in $M_\R$, that is, $\mc N(\Delta)=\Sigma$. Further, let $\CQ\in\Pol^+(\Sigma)$ be inside $\Delta$. Their corresponding set of vertices will be written $\opn{vert}(\Delta)$ and $\opn{vert}(\CQ)$, respectively. For $0\not=b\in N$ denote
\[
F(\Delta,b)=\{u\in\Delta\mid\langle u,b\rangle=\min\langle\Delta,b\rangle\}
\]
the face of $\Delta$ associated with $b$. For $v\in M_\R$, $0\not=b\in N_\R$ and $\nabla\in\Pol^+(\Sigma)$ let
\begin{equation}
\label{eq:Mad}
[b<v]:=\{u\in\Delta\mid\langle u,b\rangle<\langle v,b\rangle\}\text{ and }[b<\nabla]:=\{u\in\Delta\mid\langle u,b\rangle<\min\langle\nabla,b\rangle\};
\end{equation}
see also Figure~\ref{fig:IncPol} for illustration. 
\begin{figure}[ht]
\begin{tikzpicture}[scale=0.7]
\draw[color=oliwkowy!40] (-0.2,-2.2) grid (4.4,2.2);
\fill[pattern color=green!30, pattern=north west lines]
  (0,2.0) -- (4.4,-0.2) -- (4.4,-2.2) -- (0,0) -- cycle;
\draw[thick, dashed, green]
  (0,2.0) -- (4.4,-0.2);
\draw
   (0,2.2) -- (0,0) -- (4.4,-2.2);
\draw 
   (0.4,-0.6) node {$\Delta$};
\fill
   (2,1) circle (3pt);
\draw[->] 
   (2,1) -- (2.5,2);
\draw 
   (1.7,0.7) node {$v$}
   (2.8,1.8) node {$b$};
\draw[green]
   (2,0) node {$[b<v]$};
\end{tikzpicture}
\hspace{30pt}
\begin{tikzpicture}[scale=0.7]
\draw[color=oliwkowy!40] (-0.2,-2.2) grid (4.4,2.2);
\fill[pattern color=green!30, pattern=north west lines]
  (0,2.2) -- (0.8,2.2) -- (4.4,-1.4) -- (4.4,-2.2) -- (0,0) -- cycle;
\draw[thick, dashed, green]
  (0.8,2.2) -- (4.4,-1.4);
\draw
   (0,2.2) -- (0,0) -- (4.4,-2.2);
\draw 
   (0.4,-0.6) node {$\Delta$};
\fill
   (2,1) circle (3pt);
\draw[->] 
   (2,1) -- (3,2);
\draw 
   (3.2,1.8) node {$b$};
\draw[very thick]
   (2,2.2) -- (2,1) -- (4.4,1);
\draw
   (4.0,1.5) node {$\nabla$};
\draw[green]
   (2,0) node {$[b<\nabla]$};
\end{tikzpicture}
\caption{The polyhedra sit inside $M_\R$ while $b\in N_\R$.}
\label{fig:IncPol}
\end{figure}
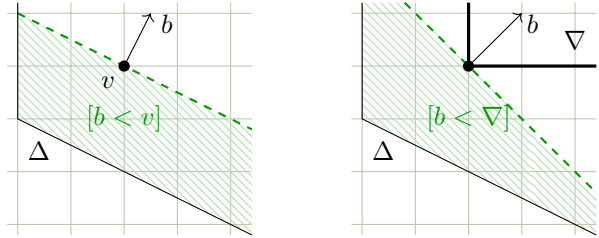

\begin{lemma}
\label{lem:PolLem}
Let $C$ be a connected component of $\Delta\setminus\CQ$ and $\nabla:=\CQ\cup C$. Then
\begin{enumerate}
\item[{\rm (i)}] $\nabla$ is a lattice polyhedron with $\opn{vert}(\nabla)\subseteq\opn{vert}(\CQ)\cup\opn{vert}(\Delta)$.

\smallskip

\item[{\rm (ii)}] $\Sigma\leq\mc N(\nabla)$, that is, $\nabla\in\Pol^+(\Sigma)$.

\smallskip

\item[{\rm (iii)}] for each $\rho\in\Sigma(1)$ we have
\[
\min\langle\Delta,\rho\rangle\leq \min\langle\nabla,\rho\rangle\leq\min\langle\CQ,\rho\rangle,  
\]
and at least one equality holds.
\end{enumerate}
\end{lemma}

\begin{proof}
First, we note that $\nabla$ is a closed and convex subset of $M_\R$, cf.~\cite[Proposition 13]{ka48-dispExt}; Figure~\ref{fig:PolLem} displays an example of our setup.
\begin{figure}[ht]
\begin{tikzpicture}[scale=0.5]
\draw[color=oliwkowy!40] (-1.3,-2.3) grid (5.3,6.3);
\fill[color=blue]
  (0,0) circle (7pt) (0,4) circle (5pt) (2,4) circle (7pt) (4,0) circle (7pt);
\fill[color=black]
  (0,-1) circle (5pt) (0,5) circle (5pt) (1,5) circle (5pt) (2,4) circle (5pt) (4,0) circle (5pt) (4,-1) circle (5pt);
\draw[thick, color=black]
  (0,-1) -- (0,5) -- (1,5) -- (2,4) -- (4,0) -- (4,-1) -- cycle;
\draw[very thick, dotted, color=blue]
  (0,0) -- (0,4) -- (2,4) -- (4,0) -- cycle;
\fill[color=red]
  (0,0) circle (5pt);
\draw[thick, color=blue]
  (1.5,2.5) node{$\CQ$};
\draw[thick]
  (0.5,-1.5) node{$\Delta$};
\fill[pattern color=green!30, pattern=north west lines]
  (0,4) -- (0,5) -- (1,5) -- (2,4) -- cycle;
\draw[thick,green]
  (0.5,4.5) node{$C$};
\end{tikzpicture}
\caption{An example of a configuration $\CQ\subseteq\Delta$. The area shaded in green represents a connected component $C$. The origin of $M_\R$ is marked in red.}
\label{fig:PolLem}
\end{figure}
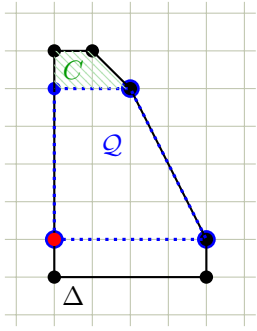

\medskip

Let $\sigma$ be the tail cone of $\Delta$ and $Q$. Next, let $0\not=b\in\sigma^\vee \subseteq N_\R$. Then $\min\langle\Delta,b\rangle\leq\min\langle\nabla,b\rangle\leq\min\langle\CQ,b\rangle$, for $\Delta\supseteq\nabla\supseteq\CQ$. Furthermore, $\min\langle\nabla,b\rangle<\min\langle\CQ,b\rangle$ if and only if $[b<\CQ]\cap C\not=\varnothing$. Since $[b<\CQ]$ is connected and therefore lies in a connected component of $\Delta\setminus\CQ$, the assumption $\min\langle\nabla,b\rangle<\min\langle\CQ,b\rangle$ implies $[b<\CQ]\subseteq C$ whence $\min\langle\nabla,b\rangle=\min\langle\Delta,b\rangle$. This proves (iii).

\medskip

Next consider $v\in\opn{vert}(\nabla)$, which gives the full dimensional normal cone 
\[
\mc N(v,\nabla):=\{b\in\sigma^\vee\mid\langle v,b\rangle=\min\langle\nabla,b\rangle\}\in\mc N(\nabla).
\]
We are going to show (ii), namely that $\mc N(v,\nabla)$ is a union of cones in $\Sigma$. 
Indeed, (iii) implies either $\langle v,\rho\rangle=\min\langle\nabla,\rho\rangle=\min\langle\CQ,\rho\rangle$, or $v\in F(\Delta,\rho)\subseteq C$. On the other hand, we can pick $p\in\partial\Delta$ -- not necessarily a vertex -- such that the interiors of $\mc N(p,\Delta)$ and $\mc N(v,\nabla)$ intersect nontrivially. Let $h\in\innt\mc N(p,\Delta)\cap\innt\mc N(v,\nabla)\subseteq \sigma^\vee$. We distinguish two cases.

\medskip

{\em Case 1:} Assume that $v\not\in\opn{vert}(\CQ)$. Then $\langle p,h\rangle\leq\langle v,h\rangle<\min\langle\CQ,h\rangle$ so that $p\in[h<\CQ]$ and $[h<\CQ]\cap C\not=\varnothing$ as it contains $v$. Hence $p\in[h<\CQ]\subseteq C\subseteq\nabla$ and $v=p$ by unicity of the minimum. Furthermore, $\nabla$ and $\Delta$ agree in $[h<\CQ]$. Consequently, $\mc N(v,\nabla)=\mc N(v,\Delta)\in\Sigma$ and $v\in\opn{vert}(\Delta)$, which entails also (i). 

\medskip

{\em Case 2:} If $v\in\opn{vert}(\CQ)$, then $\mc N(p,\Delta)\subseteq\mc N(v,\nabla)$. Indeed, let $y\in\mc N(p,\Delta)$. As $v$ is a vertex of $\nabla$ by our assumption above, $\mc N(v,\nabla)\subseteq\mc N(v,\CQ)$. In particular, $h\in\innt\mc N(p,\Delta)\cap\innt\mc N(v,\CQ)$ whence $\mc N(p,\Delta)\subseteq\mc N(v,\CQ)$ as $\Sigma\leq\mc N(\CQ)$. Therefore, $y\in\mc N(p,\Delta)$ also satisfies
\[
y\in\mc N(v,\CQ). 
\]
Assume that $y\not\in\mc N(v,\nabla)$, and let $w$ be a vertex of $\nabla$ with $y\in\mc N(w,\nabla)$. Then
\[
\langle p,y\rangle\leq\langle w,y\rangle<\langle v,y\rangle,  
\]
that is, $w\in C$, since $v\in\opn{vert}(\CQ)$ and thus $p\in[y<v]\subseteq C$. But $h\in\innt\mc N(p,\Delta)\cap\innt\mc N(v,\nabla)$, which gives $\langle p,h\rangle<\langle v,h\rangle$ for $p\in C$ and $v\in\CQ$ so that $p\not=v$. On the other hand, $\langle v,h\rangle=\min\langle\nabla, h\rangle\leq\langle p,h\rangle$ as $p\in C\subseteq\nabla$ -- contradiction!
\end{proof}

\subsection{The Weil decoration of the universal extension}
\label{subsec:ExtParl}
Let $\NablaP$ and $\NablaM$ be in $\Pol^+(\Sigma)$ with corresponding Weil divisors $D_+$ and $D_-$. Consider the intersection
\[
\CQ:=\NablaM\cap\NablaP\in\Pol(\Sigma)
\]
as a virtual polyhedron, cf. Section \ref{sec:lattice-structure}, which might or might not be in $\Pol^+(\Sigma)$. 

\medskip

Since $\Ext(\mc F\otimes\mc L,\mc G\otimes\mc L)_0=\Ext(\mc F,\mc G)_0$ 
we are free to twist the sequence~\eqref{eq:ExSeqExt} by a suitable 
ample line bundle. In particular, we may assume straightaway 
that $\NablaM$ is ample and $\CQ\in\Pol^+(\Sigma)$. From Lemma~\ref{lem:PolLem} we obtain polytopes
\[
\nabla_\nu:=\CQ\cup C_\nu\in\Pol^+(\Sigma),\quad\nu=0,\ldots,s, 
\]
where $s=\dim_\kk\gExt(\NablaM,\NablaP)_0$ by~\eqref{eq:DimExt} and $C_\nu$ are the connected components of $\NablaM\setminus\NablaP$. We consider the polytopes $\nabla_\nu$ as free generators $[\nabla_\nu]$ of the vector space
\[
V=\bigoplus_{\nu=0}^s\kk[\nabla_\nu]
\]
with dual basis $[\nabla^0],\ldots,[\nabla^s]$. Then
\[
\textstyle\gExt:=\gExt(\NablaM,\NablaP)_0=\gH^1\!\!\big(X,\CO_X(\NablaP-\NablaM)\big)_0=\widetilde{\gH}^{\raisebox{-3pt}{\scriptsize$0$}}(\NablaM\setminus\NablaP)=V{\Big/}\sum_\nu[\nabla_\nu].
\]
This yields the exact sequence
\begin{equation}
\label{eq:ExtSeqVS}
\begin{tikzcd}[column sep=40pt]
0\ar[r]&\gExt^\vee\ar[r,"\iota"]&\fbox{$\displaystyle V^\vee=\bigoplus_{\nu=0}^s\kk[\nabla^k]$}\ar[r,"\underline1:=(1\ldots1)"]&\kk\cdot e_-\ar[r]&0
\end{tikzcd}  
\end{equation}
from which we build a sequence of toric sheaves
\begin{equation}
\label{eq:ExtSeq}
\begin{tikzcd}
0\ar[r]&\gExt^\vee\otimes\ \CO_X(\NablaP)\ar[r,"\iota"]&\CE(\NablaM,\NablaP)\ar[r,"\underline1"]&\CO_X(\NablaM)\ar[r]&0
\end{tikzcd}  
\end{equation}
as follows.

\medskip

The Weil decoration $\CD_+$ of $\gExt^\vee\!\otimes\,\CO_X(\NablaP)$ assigns the generic stratum $\eta_+=\gExt^\vee\setminus\{0\}$ to $\NablaP$; the Weil decoration $\CD_-$ of $\CO_X(\NablaM)$ sends $e_-$ to $\NablaM$. Finally, $\nabla_\mu\cap\nabla_\nu=\CQ$ for $\mu\not=\nu$ and $\nabla_\nu\cap\NablaP=\CQ$ (as actual intersections!), so we can define $\CE:=\CE(\NablaM,\NablaP)$ via the Weil decoration sending $0\not=\varphi\in V^\vee$ to
\[
\CD_\CE(\varphi)=\left\{\begin{array}{ll}\nabla_\nu&\text{if }\varphi\in[\nabla^\nu]\cdot\kk\\\NablaP&\text{if }\varphi\in\iota(\eta_+)=\ker\underline1\\
\CQ&\text{if otherwise}\end{array}\right.
\]
with Hasse diagram
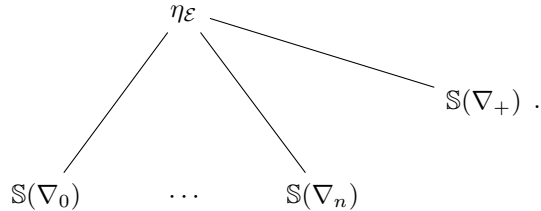
\[
\begin{tikzcd}
& \eta_\CE & & \\
& & & \S(\NablaP)\ar[llu, no head]\\
\S(\nabla_0) \ar[ruu, no head] & \cdots & \S(\nabla_n)\ar[luu, no head] & 
\end{tikzcd}
.
\]
In particular, $\CE(\NablaM,\NablaP)$ splits into a direct sum of line bundles if $\NablaP\subseteq\NablaM$. By design, $\iota$ and $\underline1$ induce morphisms of Weil decorations and therefore toric sheaf morphisms.

\begin{theorem}
\label{thm:UnivExt}
If $\NablaM$ is ample and $\CQ=\NablaM\cap\NablaP\in\Pol^+(\Sigma)$, 
then $\CE(\NablaM,\NablaP)$ is the universal extension of $\CO_X(\NablaM)$ by $\CO_X(\NablaP)$.
\end{theorem}

\begin{proof}
First of all, the sequence~\eqref{eq:ExtSeq} is an exact sequence of toric sheaves. Indeed, it follows by design that $\iota$ is a morphism of Weil decorations. Moreover, $\underline1$ sends the strata $\S(\nabla_\nu)$ to $\kk^\times\cdot e_-$ and $\S(\Delta_+)$ to $0$. Their associated polytopes are contained in $\NablaM$ and $M_\R$, respectively. Exactness follows then directly from Proposition~\ref{prop:Inj} and Proposition~\ref{prop:Sur}, since $\NablaM=\bigcup_{\nu=0}^n\nabla_\nu$. 

\medskip

If $\NablaP\subseteq\NablaM$, then $\CQ=\NablaP$ and the split bundle $\CE(\NablaM,\NablaP)=\CE(\NablaM,\CQ)$ is the universal extension by \cite[Theorem 19]{ka48-dispExt}. 

\medskip

It remains to treat the case $\CQ\subsetneq\NablaP$. The Weil decorations
\[
\begin{tikzcd}
\fbox{$\gExt^\vee\setminus\{0\}\mapsto\CQ$}\ar[r,"\iota_\CQ=\iota"]\ar[d,hookrightarrow] &\fbox{$\begin{array}{l}\S_\nu\mapsto\nabla_\nu\\\S_+,\,\eta\mapsto\CQ\end{array}$}\ar[r,"\underline1"]\ar[d,hookrightarrow]&\fbox{$1\mapsto\NablaM$}\ar[d,equal]\\[3pt]
\fbox{$\gExt^\vee\setminus\{0\}\mapsto\NablaP$}\ar[r,"\iota"]&\fbox{$\begin{array}{l}\S_\nu\mapsto\nabla_\nu\\\S_+\mapsto\NablaP,\,\eta\mapsto\CQ\end{array}$}\ar[r,"\underline1"]&\fbox{$1\mapsto\NablaM$}
\end{tikzcd}
\]
induce the exact sequences sequences 
\[
\begin{tikzcd}
0\ar[r]&\gExt^\vee\otimes\;\CO_X(\CQ)\ar[r,"\iota_\CQ"]\ar[d,hookrightarrow] &
\CE(\NablaM,\CQ)\ar[r,"\underline1"]\ar[d,hookrightarrow]&
\CO_X(\NablaM)\ar[r]\ar[d,equal]&0\\[3pt]
0\ar[r]&\gExt^\vee\otimes\;\CO_X(\NablaP)\ar[r,"\iota"]&\CE(\NablaM,\NablaP)\ar[r,"\underline1"]&\CO_X(\NablaM)\ar[r]&0
\end{tikzcd}
.
\]
Since both $\iota_\CQ$ and $\iota$ have the same cokernel, $\CE(\NablaM,\NablaP)$ is the pushout of $\iota_\CQ$ and the inclusion $\gExt^\vee\otimes\,\CO_X(\CQ)\into\gExt^\vee\otimes\,\CO_X(\NablaP)$. By \cite[Section 4.2.2]{ka48-dispExt}, this pushout is just the torus invariant universal extension.
\end{proof}

Undoing the twist by tensoring with $\CO_X(\Delta)^{-1}$ yields the toric universal extension of the original line bundles $\CO_X(\NablaM)$ by $\CO_X(\NablaP)$, cf.\ also Example~\ref{exam:TPLB}.

\begin{example}
We revisit~\cite[Example 34]{ka48-dispExt} and consider the smooth projective toric variety $X=X(\Delta)$ given in Figure~\ref{fig:NonIntegralInter} together with the polytopes $\NablaM$ and $\NablaP$. 
\begin{figure}[ht]
\begin{tikzpicture}[scale=0.5]
\draw[color=oliwkowy!40] (-1.3,-1.3) grid (4.3,5.3);
\draw[very thick, color=blue]
  (0,0) -- (3,0) -- (3,1) -- (2,3) -- (1,4) -- (0,4) -- cycle;
\draw[thick, color=black]
  (1.5,-0.5) node{$2$} (-0.5,2) node{$1$} (3.5,0.5) node{$3$}
  (2.7,2.5) node{$4$} (1.8,3.8) node{$5$} (0.5,4.5) node{$6$};
\fill[thick, color=red]
  (0,0) circle (4pt);
\draw[thick, color=black]
  (3.5,4.5) node{$M_\R$};
\draw[thick, color=blue]
  (1.5,1.5) node{$\Delta$};
\end{tikzpicture}
\qquad
\begin{tikzpicture}[scale=0.5]
\draw[color=oliwkowy!40] (-1.3,-1.3) grid (3.3,3.3);
\draw[thick, color=black]
  (-1,1) -- (3,1) (1,-1) -- (1,3) (1,1) -- (-1,0) (1,1) -- (-1,-1);
\draw[thick, color=black]
  (2,0.5) node{$\rho_1$} (0.5,2) node{$\rho_2$} (-1.5,0) node{$\rho_4$}
  (-1.5,1) node{$\rho_3$} (-1.5,-1.25) node{$\rho_5$} (1,-1.5) node{$\rho_6$};
\fill[thick, color=red]
  (1,1) circle (4pt);
\draw[thick, color=black]
  (2.5,2.5) node{$N_\R$};
\end{tikzpicture}
\qquad
\begin{tikzpicture}[scale=0.7]
\draw[color=oliwkowy!40] (-0.3,1.7) grid (2.3,5.3);
\draw[thick, color=black]
  (0,2) -- (1,2) -- (0,4) -- cycle;
\draw[very thick, color=green]
  (0,3) -- (1,3);
\draw[thick, color=green]
  (1.6,3.3) node{$\NablaP$};
\draw[thick, color=black]
  (1.6,1.7) node{$\NablaM$};
\fill[thick, color=red]
  (0,3) circle (3pt);
\draw[thick, color=black]
  (1.5,4.5) node{$M_\R$};
\end{tikzpicture}
\caption{The polytope $\Delta$, the normal fan $\Sigma=\mc N(\Delta)$, and the polytopes $\NablaM$ and $\NablaP$. The red dots mark the origins.}
\label{fig:NonIntegralInter}
\end{figure}
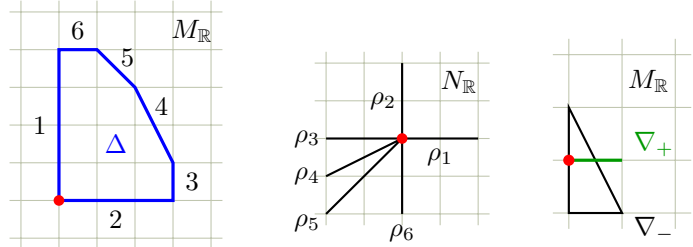

\medskip

To ease notation we let $D_i=D_{\rho_i}$ and $D_\pm=D_{\nabla_\pm}$. The divisor of $\Delta$ is given by
\[
D_\Delta=3D_3+7D_4+5D_5+4D_6,
\]
while $\NablaM$ and $\NablaP$ have divisors
\[
D_-=D_2+D_3+D_4+D_5+D_6\quad\text{and}\quad D_+=D_3+2D_4+D_5.
\]
Finally, the divisor of $\CQ=\NablaP\cap\NablaM$ is
\[
D_\CQ=D_3+D_4+D_5. 
\]
We need to make $\NablaM$ ample and $\CQ$ a compatible lattice polytope. This can be achieved by twisting with $\Delta$ which yields polytopes denoted $\nabla^+_-$, $\nabla^+_+$ and $\CQ^+$ with respective divisors
\begin{align*}
D_{\nabla^+_-}&=D_2+4D_3+8D_4+6D_5+5D_6\\[3pt]
D_{\nabla^+_+}&=4D_3+9D_4+6D_5+4D_6\\[3pt]
D_{\CQ^+}&=4D_3+8D_4+6D_5+4D_6.
\end{align*}
Further, $\nabla^+_-\setminus\CQ^+$ has two connected components $\nabla_0$ and $\nabla_1$, cf.\ Figure~\ref{fig:LattPol}, in accordance with the two connected components of $\DeltaM\setminus\DeltaP$, cf.\ Figure~\ref{fig:NonIntegralInter}.
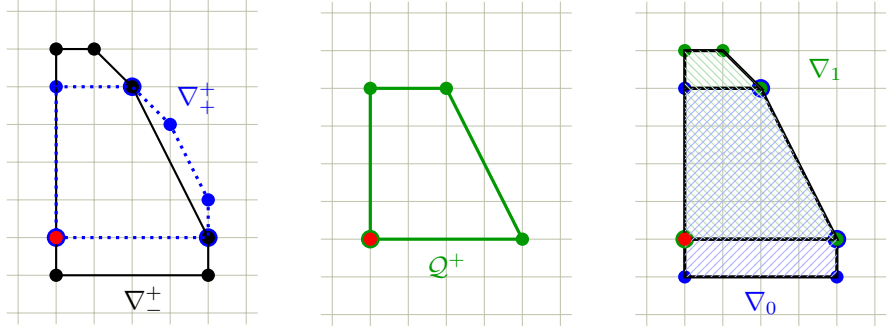
\begin{figure}[ht]
\begin{tikzpicture}[scale=0.5]
\draw[color=oliwkowy!40] (-1.3,-2.3) grid (5.3,6.3);
\fill[color=blue]
  (0,0) circle (7pt) (0,4) circle (5pt) (2,4) circle (7pt) (3,3) circle (5pt) (4,1) circle (5pt) (4,0) circle (7pt);
\fill[color=black]
  (0,-1) circle (5pt) (0,5) circle (5pt) (1,5) circle (5pt) (2,4) circle (5pt) (4,0) circle (5pt) (4,-1) circle (5pt);
\draw[thick, color=black]
  (0,-1) -- (0,5) -- (1,5) -- (2,4) -- (4,0) -- (4,-1) -- cycle;
\draw[very thick, dotted, color=blue]
  (0,0) -- (4,0) -- (4,1) -- (3,3) -- (2,4) -- (0,4) -- cycle;
\fill[color=red]
  (0,0) circle (5pt);
\draw[thick, color=blue]
  (3.7,3.7) node{$\nabla^+_+$};
\draw[thick]
  (2.3,-1.7) node{$\nabla^+_-$};
\end{tikzpicture}
\qquad
\begin{tikzpicture}[scale=0.5]
\draw[color=oliwkowy!40] (-1.3,-2.3) grid (5.3,6.3);
\fill[color=green]
  (0,0) circle (7pt) (0,4) circle (5pt) (2,4) circle (5pt) (4,0) circle (5pt);
\draw[very thick, color=green]
  (0,0) -- (0,4) -- (2,4) -- (4,0) -- cycle;
\fill[color=red]
  (0,0) circle (5pt);
\draw[thick, color=green]
  (2.0,-0.7) node{$\CQ^+$};
\end{tikzpicture}
\qquad
\begin{tikzpicture}[scale=0.5]
\draw[color=oliwkowy!40] (-1.3,-2.3) grid (5.3,6.3);
\fill[color=blue]
  (4,0) circle (7pt) (4,-1) circle (5pt) (0,-1) circle (5pt) (0,4) circle (5pt) (2,4) circle (7pt);
\draw[very thick, color=black]
  (0,0) -- (4,0) -- (4,-1) -- (0,-1) -- cycle;
\fill[color=green]
  (0,0) circle (7pt) (0,5) circle (5pt) (1,5) circle (5pt) (2,4) circle (5pt) (4,0) circle (5pt);
\draw[very thick, color=black]
  (0,4) -- (0,5) -- (1,5) -- (2,4) -- cycle;
\draw[very thick, color=black]
  (0,0) -- (0,4) -- (2,4) -- (4,0) -- cycle;
\fill[pattern color=green!30, pattern=north west lines]
 (0,0) -- (4,0) -- (2,4) -- (1,5) -- (0,5) -- cycle;
\fill[pattern color=blue!30, pattern=north east lines]
 (0,-1) -- (4,-1) -- (4,0) -- (2,4) -- (0,4) -- cycle;
\fill[color=red]
  (0,0) circle (5pt);
\draw[thick, color=blue]
  (2,-1.7) node{$\nabla_0$};
\draw[thick, color=green]
  (3.7,4.5) node{$\nabla_1$};
\end{tikzpicture}
\caption{The lattice polytopes of $\nabla^+_+$ and $\nabla^+_-$, the intersection $\CQ^+$ and the polytopes $\nabla_0$ and $\nabla_1$.}
\label{fig:LattPol}
\end{figure}
The Weil decoration of $\CE^+:=\CE(\nabla^+_-,\nabla^+_+)$ is therefore given on $0\not=v\in V$ by
\[
\renewcommand{\arraystretch}{1.3}
\CD_{\CE^+}(v)=\begin{cases}
\nabla_0&\text{if }v\in[\nabla^0]\cdot\kk\\[3pt]
\nabla_1&\text{if }v\in[\nabla^1]\cdot\kk\\[3pt]
\nabla^+_+&\text{if }v\in([\nabla^1]-[\nabla^0])\cdot\kk\\[3pt]
\CQ^+&\text{if otherwise}
\end{cases}
.
\]
In terms of divisors, we have $D_{\nabla_0}=4D_3+8D_4+6D_5+5D_6$ and $D_{\nabla_1}=D_2+4D_3+8D_4+5D_5+4D_6$. Untwisting with $\Delta$ yields for $\CE=\CE(\NablaM,\NablaP)$ the Weil decoration 
\[
\renewcommand{\arraystretch}{1.3}
\CD_\CE(v)=\begin{cases}
D_3+D_4+D_5+D_6&\text{if }v\in[\nabla^0]\cdot\kk\\[3pt]
D_2+D_3+D_4&\text{if }v\in[\nabla^1]\cdot\kk\\[3pt]
D_3+2D_4+D_5=\nabla^+&\text{if }v\in([\nabla^1]-[\nabla^0])\cdot\kk\\[3pt]
D_3+D_4+D_5&\text{if otherwise}
\end{cases}
\]
in accordance with the Klyachko filtrations of~\cite[Example 34]{ka48-dispExt}. 
\end{example}

\section{The cohomology of toric sheaves}
\label{sec:CohomTS}
From now on and for the rest of the paper we assume that 
\[
X\text{ \bf is a smooth and projective toric variety}
\]
and turn to the study of the cohomology of a toric sheaf $\CE$ on $X$ with Weil decoration $\CD=\CD_\CE$. Further, we assume that
\[
\CE^+:=\CE(\Delta)=\CE\otimes\CO_X(\Delta) 
\]
is {\em amply decorated} for some $\Delta\in\Pol^+(\Sigma)$, that is, its Weil decoration $\CD^+$ yields an ample line bundle $\CO_X(\CPvar^+(e))$ for every $e\in E\setminus\{0\}$. If $\Delta$ itself is ample, we write $X=\PP(\Delta)$, cf.\ also Remark~\ref{rem:Positivity}.

\subsection{The associated constructible $\kk$-sheaf on $\Delta$}
\label{subsec:ConstSheav} 
In~\eqref{eq:DefCFCE} we defined a $\kk$-sheaf (= $\kk$-vector space valued sheaf) over $\Delta$ by
\[
\CF(\CE)(U)=\{e\in E\mid U\subseteq\CPvar^+(e)\}=\bigcap_{u\in U}\eval^+(u)
\]
for any open subset $U$ of $\Delta$. It is a subsheaf of the constant sheaf $\underline E$ on $\Delta$. The restriction maps for $V\subseteq U$ are therefore 
just inclusions, and the sheaf is completely determined by its stalks. 
At $u\in\Delta$, the stalk is given by
\[
\CF(\CE)_u=\{e\in E\mid u\in\innt_\Delta\CPvar^+(e)\}\subseteq \eval^+(u).
\]
For $A$, $B\subseteq M_\R$, we denote by $\innt_A(B)$ the interior of $A\cap B$ with respect to induced topology on $A$, i.e.,  
\[
\innt_A(B)=\{u\in A\cap B\mid\exists\, U\text{ open in $M_\R$ with $u\in U$ and }U\cap A\subseteq B\}. 
\]
The relevance of $\CF(\CE)$ stems from the following theorem whose proof will occupy us in Section \ref{sec:Proof}.

\begin{theorem}
\label{thm:CohomSheafE}
Let $X$ be a smooth projective toric variety. If $\CE^+=\CE(\Delta)$ is amply decorated with $\Delta \in \Pol^+(\Sigma)$, then
\[
\gH^\ell(X,\CE)_0\cong\gH^\ell\hspace{-0.5ex}\big(\Delta,\CF(\CE)\big),
\]
where the left-hand side denotes the sheaf cohomology of $\CE$ in degree $0$.
\end{theorem}

\begin{remark}
\label{rem:GenDegree}
To compute the degree $m\in M$ piece of cohomology, we define the sheaf
\[
\CF_m(\CE)(U)=\{e\in E\mid U\subseteq\CPvar^+(e)\}
\]
on $\Delta+m$, i.e.\ $U\subseteq\Delta+m$. Then $\gH^\ell(X,\CE)_m\cong\gH^\ell\hspace{-0.5ex}\big(\Delta+m,\CF_m(\CE)\big)$, see also Examples~\ref{subsubsec:cohomSheafL} and~\ref{subsubsec:cohomSheafTang} below.
\end{remark}

If $\CE$ is already amply decorated, Theorem~\ref{thm:CohomSheafE} also caters for the case $\CO_X(\Delta)=\CO_X$, i.e., $\CE^+=\CE$, even though $\Delta=\{0\}$ is not ample but merely nef.

\begin{corollary}
\label{coro:Acyclic}
An amply decorated toric sheaf $\CE$ is acyclic, that is, $\gH^\ell(X,\CE)=0$ for $\ell\geq1$.
\end{corollary}

\begin{remark}
\label{rem:AmpleNec}
Ampleness of the line bundles $\CO_X(\CD^+(e))=\CO_X(\CD(e))\otimes\CO_X(\Delta)$ is actually necessary as the example in \ssect{subsubsec:cohomSheafL} shows. Note that this is actually stronger than ampleness of $\CE(\Delta)$. For example, consider the Kaneyama bundle~\cite{kaneyama} on $X=\P^2$ which arises from the short exact sequence
\[
0\to\CO_X(-D_0)\to\bigoplus_{i=0}^2\CO_X(D_i)\to\CE\to0.
\]
It is not amply decorated since its generic stratum carries the divisor $-D_0$. Still, it is an ample toric vector bundle for the middle term is obviously ample, and quotients of ample vector bundles are again ample. Actually, making the first term of the exact sequence even more negative by considering
\[
0\to\CO_X(-D_0-D_1-D_2)\to\bigoplus_{i=0}^2\CO_X(D_i)\to\CE_2\to0
\]
we get an ample toric vector bundle that has non-vanishing higher cohomology.
\end{remark}

\begin{remark}
The condition of ample decoration also appears implicitly in \cite[Thm.\ 6.10]{tropical3} where the authors show a special case of Corollary~\ref{coro:Acyclic}, namely that the equivariant Euler characteristic of an amply decorated toric vector bundle equals the dimension of its global sections. 
\end{remark}

\subsection{Examples}
\label{subsec:Examples}
We illustrate Theorem~\ref{thm:CohomSheafE} by several examples.

\subsubsection{The cohomology of line bundles}
\label{subsubsec:cohomSheafL}
We revisit Example~\ref{exam:LB} and consider a line bundle $\CL$. Then Theorem~\ref{thm:CohomSheafE} is a consequence of~~\cite{immaculate} and~\cite{dop}.

\medskip

Indeed, let $\CL^+=\CL(\Delta)$ be ample. Its associated Weil decoration $\CD^+$ is defined on $E=\kk$ and thus specified by $\CD^+(1)$. Write $\CF$ for $\CF(\CL)$ and let 
\[
U:=\innt_\Delta\CPvar^+(1).
\]
The stalk of $\mc F$ at $u\in\Delta$ is given by
\[
\renewcommand{\arraystretch}{1.3}
\CF_u=\begin{cases}
\kk & \text{if }u\in U,\\
0 & \text{if } u\in Z:=\Delta\setminus U.
\end{cases}
\]
We denote $j\colon U\hookrightarrow\Delta$ and $i\colon Z\hookrightarrow\Delta$ the resulting open and closed embedding. The associated  exact Gysin sequence
\[
0\to j_!\,j^*\uk\to\uk\to i_*\,i^*\uk\to0
\]
yields the long exact cohomology sequence
\[
0\to\gH^0(\Delta,j_!\,j^*\uk)\to\fbox{$\gH^0(\Delta,\uk)=\kk$}\to\gH^0(\Delta,i_*\,i^*\uk)\to\gH^1(\Delta,j_!\,j^*\uk)\to\ldots.
\]
Now $j_!\,j^*\uk=j_!\,\uk=\CF$ while $\gH^\ell(\Delta,\uk)=\gH^\ell(\Delta,\kk)$ is trivial for $\ell>0$ as $\Delta$ is contractible. Since
\[
\gH^\ell(\Delta,i_*\,i^*\uk)=\gH^\ell(Z,i^*\uk)=\gH^\ell(Z,\kk)
\]
we deduce for the reduced cohomology $\widetilde{\opn H}$ with coefficients in $\kk$ that
\begin{equation}
\label{eq:CohomLB}
\gH^\ell(\Delta,\CF)=\gH^\ell(\Delta,j_!\,j^*\uk)=\widetilde{\gH}^{\raisebox{-3pt}{\scriptsize$\ell\!-\!1$}}(Z)\quad\text{for }\ell\geq0.
\end{equation}
Since $\CD^+(1)$ is ample, $Z=\Delta\setminus U$ is homotopy equivalent to $\Delta\setminus\overline U$ so that $\widetilde{\opn H}^{\raisebox{-2pt}{\scriptsize$\ell-1$}}(Z)\cong\widetilde{\opn H}^{\raisebox{-2pt}{\scriptsize$\ell-1$}}(\Delta\setminus\CD^+(1))$ which is just $\gH^{\ell}(X,\mc L)_0$.

\subsubsection{The cohomology of $\CO_{\P^2}(-3)$}
\label{ssec:cohom:O-three}
It is instructive to consider an explicit example, say $\CL=\CO_{\P^2}(-3)$. For this we let $\P^2$ be the projective plane $\PP(\Delta_1)$ given by the standard simplex
\begin{equation}
\label{eq:StandDelta}
\Delta_1:=\conv\{[0,0],\; [1,0],\; [0,1]\},
\end{equation}
whose fan is labelled as in Figure~\ref{fig:FanP2}. 
\begin{figure}[ht]
\newcommand{\scaleA}{0.3}
\begin{tikzpicture}[scale=\scaleA]
\draw[thick, color=black]
  (0,0) -- (4,0) (0,0) -- (0,4) (0,0) -- (-2.5,-2.5);
\draw[thick, color=green]
  (2,2) node{$\sigma_0$};
\draw[thick, color=blue]
  (-2,1) node{$\sigma_1$};
\draw[thick, color=red]
  (1,-2) node{$\sigma_2$};
\draw[thick, color=black]
  (5,0) node{$\rho_1$} (0,5) node{$\rho_2$} (-3.3,-3.3) node{$\rho_0$};
\end{tikzpicture}
\caption{The fan $\Sigma(\PP^2)$.}
\label{fig:FanP2}
\end{figure}
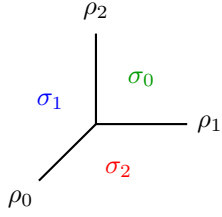
It is well-known that $\gH^\ell(\CO_{\P^2}(-3))=0$ unless $\ell=2$. If $\CO_{\P^2}(-3)$ is realised as $\CO_{\P^2}(-3D_{\rho_0})\subseteq(j_{\kk^\times})_*\mc O_{\kk^\times}$, then $\gH^2$ is located in $[-1,-1]\in M$. 

\medskip

In our formalism, we start with the Weil decoration $\CPvar$ of $\CL$ on $E=\kk$ determined by $\CPvar(1)=-3D_{\rho_0}$. Rescaling the polytope $\Delta_1$ from~\ref{eq:StandDelta} we now take
\[ 
\Delta=\Delta_4=4\Delta_1=\conv\{[0,0],\; [4,0],\; [0,4]\}
\]
and consider the materialised Weil decoration $\CPvar^+(1)=\Delta$ of $\CL^+=\CL(\Delta)$, cf.\ Figure~\ref{fig:O-3}.
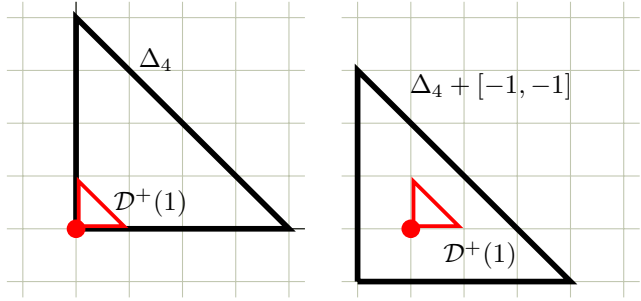
\begin{figure}[ht]
\newcommand{\scaleB}{0.7}
\begin{tikzpicture}[scale=\scaleB]
\draw[color=oliwkowy!40] (-1.3,-1.3) grid (4.3,4.3);
\draw[thin, black] 
  (0,0) -- (4.3,0) (0,0) -- (0,4.3);
\draw[line width=0.75mm, black]
  (0,0) -- (4,0) -- (0,4) -- (0,0);
\draw[thick, color=black] 
  (1.5,3.2) node{$\Delta_4$};
\draw[line width=0.5mm, red]
  (0.05,0.05) -- (0.90,0.05) -- (0.05,0.90) -- (0.05,0.05);
\draw[thick, color=black] 
  (1.4,0.5) node{$\CPvar^+(1)$};
\fill[thick, color=red]
  (0,0) circle (5pt);
\end{tikzpicture}
\quad
\begin{tikzpicture}[scale=\scaleB]
\draw[color=oliwkowy!40] (-1.3,-1.3) grid (4.3,4.3);
\draw[line width=0.75mm, black]
  (-1,-1) -- (3,-1) -- (-1,3) -- (-1,-1);
\draw[thick, color=black] 
  (1.5,2.7) node{$\Delta_4+[-1,-1]$};
\draw[line width=0.5mm, red]
  (0.05,0.05) -- (0.90,0.05) -- (0.05,0.90) -- (0.05,0.05);
\draw[thick, color=black] 
  (1.3,-0.5) node{$\CPvar^+(1)$};
\fill[thick, color=red]
  (0,0) circle (5pt); 
\end{tikzpicture}
\caption{The materialised Weil decoration of $\CL=\CO_{\P^2}(\Delta_1-\Delta_4)$ with $\Delta=\Delta_4$ (left-hand side) and $\Delta=\Delta_4+[-1,-1]$ (right-hand side). We indicate the origin by a red dot.}
\label{fig:O-3}
\end{figure}
Using Equation~\eqref{eq:CohomLB} we compute $\widetilde{\opn H}^{\raisebox{-2pt}{\scriptsize$\ell$}}((\Delta+m)\setminus U)$ for $m\in M$ and $U=\innt_\Delta\CPvar^+(1)$. 
The space $(\Delta+m)\setminus U$ is contractible unless $m=[-1,-1]$. In the latter case we obtain a space homotopy equivalent to a $1$-sphere which yields indeed $\gH^2(\CO_{\P^2}(-3))_{[-1,-1]}\cong\kk$.

\subsubsection{The cohomology of the twisted tangent sheaf $\CT_{\PP^2}(\ell)$}
\label{subsubsec:cohomSheafTang}
Next we discuss the twisted tangent sheaf $\CE=\CT_{\PP^2}(\ell)$, $\ell\in\Z$, with Weil decoration $\CD=\CD_\CE$. Recall that its Euler characteristic is
\begin{equation}
\label{eq:EulerChar}
\chi\big(\PP^2,\,\CT_{\PP^2}(\ell)\big) \;=\; (\ell+3)^2-1=(\ell+2)(\ell+4).
\end{equation}
Further, the tangent sheaf has {\em natural} cohomology, i.e., it is 
concentrated in one $\gH^i$. 
Namely, the cohomology of $\CT_{\PP^2}(\ell)$ lies in 
$\gH^0$ for $\ell\geq -1$, 
in $\gH^1$ for $\ell=-3$, 
and in $\gH^2$ for $\ell\leq -5$. For $\ell+3=\pm 1$ 
we obtain the immaculate bundles $\CT_{\PP^2}(-2)$ and $\CT_{\PP^2}(-4)$, 
that is, there is no cohomology at all.

\medskip

We first consider the untwisted case $\ell=0$. 
As we have just remarked, $\gH^\ell(\CT_{\PP^2})$ vanishes unless $\ell=0$ where we find $\kk^8$. It is well-known that we have nontrivial $M$-homogeneous pieces for $m=[0,0]$ contributing $\kk^2$, and for $\pm[1,0]$, $\pm[0,1]$ and $\pm[1,-1]$, each piece contributing $\kk$. 

\medskip

To recover this result in our formalism we recall that the associated vector space is $E=N_\R$. The Weil decoration $\CD$ of $\CT_{\PP^2}$ is given by
\begin{equation}
\label{eq:WDzero}
\CPvar(\rho_0)=D_{\rho_0}=\Delta_1,\;
\CPvar(\rho_1)=D_{\rho_1}=\Delta_1-[1,0],\;
\CPvar(\rho_2)=D_{\rho_2}=\Delta_1-[0,1],
\end{equation}
cf.\ also Figure~\ref{fig:TangP2}. To get an amply decorated bundle we twist by $\Delta=\Delta_1$, cf.~\eqref{eq:StandDelta}, whence 
\[
\CPvar^+(\rho_0)=2\Delta_1,\quad
\CPvar^+(\rho_1)=2\Delta_1-[1,0],\quad
\CPvar^+(\rho_2)=2\Delta_1-[0,1]
\]
is the materialised Weil decoration of $\CT_{\PP^2}^+=\CT_{\PP^2}(\Delta_1)$. 
Note that 
\[
\CPvar^+(\eta)=\CPvar^+(\rho_0)\cap\CPvar^+(\rho_1)\cap\CPvar^+(\rho_2)=\Delta_1;
\]
cf.\ also the left hand side of Figure~\ref{fig:MatTangP2}. 
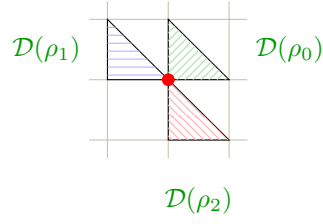
\begin{figure}[ht]
\newcommand{\scaleB}{0.8}
\begin{tikzpicture}[scale=\scaleB]
\draw[color=oliwkowy!40] (-1.3,-1.3) grid (1.3,1.3);
\draw[thin, black]
  (0,0) -- (1.0,0) -- (0,1.0) -- (0,0);
\draw[thin, black]
  (0,0) -- (-1.0,0) -- (-1.0,1.0) -- (0,0);
\draw[thin, black]
  (0,0) -- (1.0,-1.0) -- (0,-1.0) -- (0,0);
\fill[pattern color=green!30, pattern=north east lines]
  (0,0) -- (1.0,0) -- (0,1.0) -- (0,0);
\fill[pattern color=blue!30, pattern=horizontal lines]
  (0,0) -- (-1.0,0) -- (-1.0,1.0) -- (0,0);
\fill[pattern color=red!30, pattern=north west lines]
  (0,0) -- (1.0,-1.0) -- (0,-1.0) -- (0,0);
\fill[thick, red] 
  (0,0) circle (3pt);
\draw[thick, color=cConeA]
  (2.0,0.5) node{$\CPvar(\rho_0)$} (-2.0,0.5) node{$\CPvar(\rho_1)$} (0.5,-2.0) node{$\CPvar(\rho_2)$};
\end{tikzpicture}
\caption{The positive Weil decoration of $\CT_{\PP^2}$.}
\label{fig:TangP2}
\end{figure}
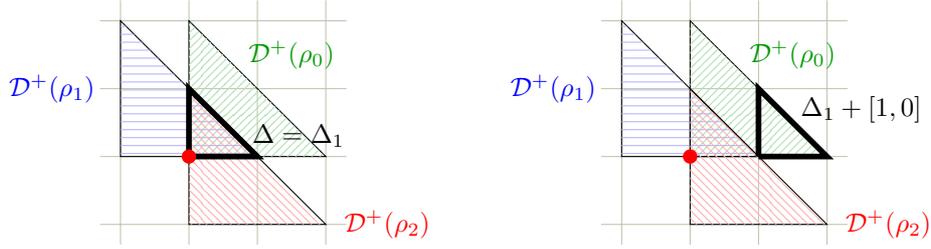
\begin{figure}[ht]
\newcommand{\scaleA}{0.3} 
\newcommand{\scaleB}{0.9}
\newcommand{\spaceA}{\hspace*{2em}}
\begin{tikzpicture}[scale=\scaleB]
\draw[color=oliwkowy!40] (-1.3,-1.3) grid (2.3,2.3);
\draw[thin, black] 
  (0,0) -- (2.0,0) -- (0,2.0) -- (0,0);
\fill[pattern color=green!30, pattern=north east lines]
  (0,0) -- (2.0,0) -- (0,2.0) -- (0,0);
\draw[thin, black]
  (1,0) -- (-1.0,0) -- (-1.0,2.0) -- (1,0);
\fill[pattern color=blue!30, pattern=horizontal lines]
   (1,0) -- (-1.0,0) -- (-1.0,2.0) -- (1,0);
\draw[thin, black]
  (0,1) -- (2.0,-1.0) -- (0,-1.0) -- (0,1);
\fill[pattern color=red!30, pattern=north west lines]
  (0,1) -- (2.0,-1.0) -- (0,-1.0) -- (0,1);
\draw[thick, color=green]
  (1.5,1.5) node{$\CPvar^+(\rho_0)$};
\draw[thick, color=blue]
  (-2.0,1.0) node{$\CPvar^+(\rho_1)$};
\draw[thick, color=red]
  (2.9,-1.0) node{$\CPvar^+(\rho_2)$};
\draw[line width=0.75mm, black]
  (0,0) -- (1,0) -- (0,1) -- (0,0);
\draw[thick, color=black]
  (1.6,0.3) node{$\Delta=\Delta_1$};
\fill[thick, red] 
  (0,0) circle (3pt);
\end{tikzpicture}
\spaceA
\begin{tikzpicture}[scale=\scaleB]
\draw[color=oliwkowy!40] (-1.3,-1.3) grid (2.3,2.3);
\draw[thin, black] 
  (0,0) -- (2.0,0) -- (0,2.0) -- (0,0);
\fill[pattern color=green!30, pattern=north east lines]
  (0,0) -- (2.0,0) -- (0,2.0) -- (0,0);
\draw[thin, black]
  (1,0) -- (-1.0,0) -- (-1.0,2.0) -- (1,0);
\fill[pattern color=blue!30, pattern=horizontal lines]
  (1,0) -- (-1.0,0) -- (-1.0,2.0) -- (1,0);
\draw[thin, black]
  (0,1) -- (2.0,-1.0) -- (0,-1.0) -- (0,1);
\fill[pattern color=red!30, pattern=north west lines]
  (0,1) -- (2.0,-1.0) -- (0,-1.0) -- (0,1);
\draw[thick, color=green]
  (1.5,1.5) node{$\CPvar^+(\rho_0)$};
\draw[thick, color=blue]
  (-2.0,1.0) node{$\CPvar^+(\rho_1)$};
\draw[thick, color=red]
  (2.9,-1.0) node{$\CPvar^+(\rho_2)$};
\draw[line width=0.75mm, black]
  (1,0) -- (2,0) -- (1,1) -- (1,0);
\draw[thick, color=black]
  (2.5,0.7) node{$\Delta_1+[1,0]$};
\fill[thick, red] 
  (0,0) circle (3pt);
\end{tikzpicture}
\caption{The ample materialised Weil decoration 
of $\CT_{\PP^2}$. In addition, the black triangle displays $\Delta=\Delta_1$ 
on the left and $\Delta=\Delta_1+[1,0]$ on the right hand side.}
\label{fig:MatTangP2}
\end{figure}

\medskip

To capture the degree $m\in M$ we consider $\CF_m=\CF_m(\CT_{\PP^2})$. For $m=0$ we have $\CF=\underline N_\kk$ whence $\gH^0(\Delta,\CF)=N_\kk$ is the only non-trivial cohomology of $\CF$.

\medskip

To compute the cohomology of $\CE=\CT_{\PP^2}$ in degree $m\in M\setminus\{0\}$, we shift $\Delta$ by $m$ while keeping the ample materialised Weil decoration $\CPvar^+$ unchanged. See the right hand side of Figure~\ref{fig:MatTangP2} for the case $m=[1,0]$. Here, $\CF=\rho_0\cdot\underline{\kk}\subsetneq\constE$ whence $\gH^0(\Delta,\CF)=\rho_0\cdot\underline{\kk}$ etc. Altogether, our Theorem~\ref{thm:CohomSheafE} obtains the right answer.

\medskip

Next we turn to the case of the twisted tangent sheaves $\CT_{\PP^2}(\ell)=\CT_{\PP^2}(\ell D_{\rho_0})$ with $\ell\geq 1$. The Weil decoration $\CD_\ell$ of $\CT_{\PP^2}(\ell)$ on $E=N_\R$ is given by
\[
\CPvar(\rho_0)=(\ell+1)\cdot\Delta_1,\quad
\CPvar(\rho_1)=(\ell+1)\cdot\Delta_1-[1,0],\quad
\CPvar(\rho_2)=(\ell+1)\cdot\Delta_1-[0,1].
\]
This is amply decorated as
\[
\CPvar(\eta)=\CPvar(\rho_0)\cap\CPvar(\rho_1)\cap\CPvar(\rho_2)=\ell\cdot\Delta_1=\Delta_\ell\,;
\]
consequently, $\CT_{\PP^2}(\ell)$ is acyclic, cf.\ Corollary~\ref{coro:Acyclic}. 

\medskip

It remains to understand $\gH^0(X,\CT_{\PP^2}(\ell))$. Since we can take $\Delta=\Delta_0=\{0\}$, the relative interior of $\Delta$ just records membership. Moreover, the sheaf $\CF_m$ is supported on this point. 
It equals the entire space $N_\kk$, if $m$ belongs to the triple 
intersection $\CPvar(\eta)$; further, 
$\rho_\indStrat\cdot\kk$ for $m\in\CPvar(\rho_\indStrat)\setminus\CPvar(\eta)$, 
and $0$ otherwise. 

\medskip

The cases $\ell=1$ and $\ell=2$ are displayed in Figure~\ref{fig:TBGlobSec}. For $\ell=1$ we find $3$ points each contributing $\kk^2$ and nine points each contributing $\kk$. Hence $
\chi(\CT_{\PP^2}(1))=15$ in accordance with~\eqref{eq:EulerChar}. Similarly, we find $
\chi(\CT_{\PP^2}(2))=3\cdot 4+6\cdot2=24$. 

\medskip

In general, we find $(\ell+1)(\ell+2)/2$ points in $\CPvar(\eta)$ each contributing $\kk^2$, and $3(\ell+2)$ points from the sets $\CPvar(\rho_i)\setminus\CPvar(\eta)$ each contributing $\kk$. Altogether, this yields
\[
(\ell+1)(\ell+2)+3(\ell+2)=(\ell+3)^2-1=\chi\big(\CT_{\PP^2}(\ell)\big).
\]
\begin{figure}[ht]
\newcommand{\scaleA}{0.3} 
\newcommand{\scaleB}{0.9}
\newcommand{\spaceA}{\hspace*{2em}}
\begin{tikzpicture}[scale=\scaleB]
\draw[color=oliwkowy!40] (-2.3,-1.3) grid (3.3,3.3);
\draw[thin, black] 
  (0,0) -- (2.0,0) -- (0,2.0) -- (0,0);
\fill[pattern color=green!30, pattern=north east lines]
  (0,0) -- (2.0,0) -- (0,2.0) -- (0,0);
\draw[thin, black]
  (1,0) -- (-1.0,0) -- (-1.0,2.0) -- (1,0);
\fill[pattern color=blue!30, pattern=horizontal lines]
   (1,0) -- (-1.0,0) -- (-1.0,2.0) -- (1,0);
\draw[thin, black]
  (0,1) -- (2.0,-1.0) -- (0,-1.0) -- (0,1);
\fill[pattern color=red!30, pattern=north west lines]
  (0,1) -- (2.0,-1.0) -- (0,-1.0) -- (0,1);
\draw[thick, color=green]
  (1.5,1.5) node{$\CPvar(\rho_0)$};
\draw[thick, color=blue]
  (-1.8,1.5) node{$\CPvar(\rho_1)$};
\draw[thick, color=red]
  (1.5,-1.5) node{$\CPvar(\rho_2)$};
\fill[thick, red] 
  (0,0) circle (3pt);
\end{tikzpicture}
\spaceA
\begin{tikzpicture}[scale=\scaleB]
\draw[color=oliwkowy!40] (-2.3,-1.3) grid (3.3,3.3);
\draw[thin, black] 
  (0,0) -- (3.0,0) -- (0,3.0) -- (0,0);
\fill[pattern color=green!30, pattern=north east lines]
  (0,0) -- (3.0,0) -- (0,3.0) -- (0,0);
\draw[thin, black]
  (2,0) -- (-1.0,0) -- (-1.0,3.0) -- (2,0);
\fill[pattern color=blue!30, pattern=horizontal lines]
   (2,0) -- (-1.0,0) -- (-1.0,3.0) -- (2,0);
\draw[thin, black]
  (0,2) -- (3.0,-1.0) -- (0,-1.0) -- (0,2);
\fill[pattern color=red!30, pattern=north west lines]
  (0,2) -- (3.0,-1.0) -- (0,-1.0) -- (0,2);
\draw[thick, color=green]
  (2.5,1.5) node{$\CPvar(\rho_0)$};
\draw[thick, color=blue]
  (-1.8,1.5) node{$\CPvar(\rho_1)$};
\draw[thick, color=red]
  (2.5,-1.5) node{$\CPvar(\rho_2)$};
\fill[thick, red] 
  (0,0) circle (3pt);
\end{tikzpicture}
\caption{$\ell=1$ (left hand side) and $\ell=2$ (right hand side).}
\label{fig:TBGlobSec}
\end{figure}
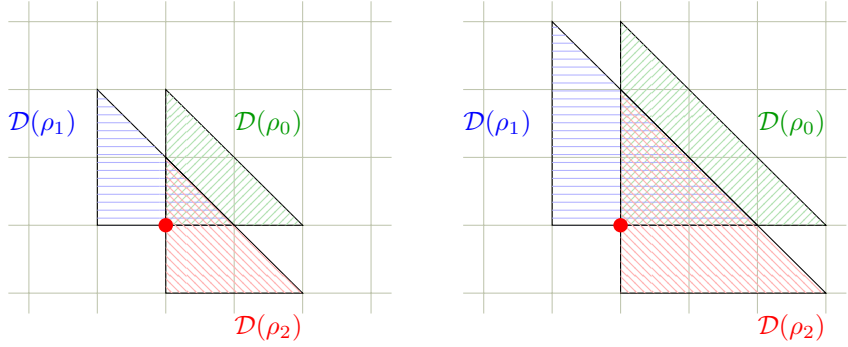
We deal with negative $\ell$ in \ssect{subsec:TangBundle}.

\subsection{The local sheaves $\CF_\sigma$}
\label{subsec:LocalSheaf}
Mutatis mutandis the definition for the $\kk$-sheaf $\CG$ in \ssect{subsec:ConstSheav} also works for the affine toric varieties $U_\sigma$, $\sigma\in\Sigma$. We define the {\em local $\kk$-sheaves of $\CE$} by
\[
\CE_\sigma:=\CE|_{U_\sigma}\quad\text{and}\quad\CF_\sigma(\CE)=\CF(\CE_\sigma).
\]

\begin{lemma}
\label{lem:ExacFunc}
$\mc F_\sigma$ is additive and exact in the sense that if $0\to\CE'\to\CE\to\CE''\to0$ is exact and $\CE'(\Delta)$, $\CE(\Delta)$, $\CE''(\Delta)$ are positively decorated, then 
\[
\begin{tikzcd}
0\ar[r]&\fbox{$\CF'_\sigma=\CF_\sigma(\CE')$}\ar[r,"\lambda"]&\fbox{$\CF_\sigma=\CF_\sigma(\CE)$}\ar[r,"\mu"]&\fbox{$\CF''_\sigma=\CF_\sigma(\CE'')$}\ar[r]&0 
\end{tikzcd}
\]
is exact for all $\sigma\in\Sigma$.
\end{lemma}

\begin{proof}
For additivity, we note that $U\subseteq\CPvar_\sigma(e\oplus e')$ if and only if $U\subseteq\CPvar_\sigma(e)\cap\CPvar_\sigma(e')$. This shows $\CF_\sigma(\CE\oplus\CE')=\CF_\sigma(\CE)\oplus\CF_\sigma(\CE')$. 

\medskip

For exactness, we proceed stalkwise and let $u\in\Delta$. The stalks are subspaces of the associated vector spaces for which we have the exact sequence
\[
\begin{tikzcd}
0\ar[r]&E'\ar[r,"\lambda"]&E\ar[r,"\mu"]&E''\ar[r]&0. 
\end{tikzcd} 
\]
Hence, regarding $E'$ as a subspace of $E$, we have $(\CF'_\sigma)_u=(\CF_\sigma)_u\cap E'$. For surjectivity of $(\CF_\sigma)_u\to(\CF''_\sigma)_u$, let $e''\in(\CF''_\sigma)_u$, i.e., $u\in\innt_\Delta\CD''_\sigma(e'')$. Surjectivity of $\CE\to\CE''$ implies that
\[
\CPvar''_\sigma(e'')=\bigcup_{\mu(e)=e''}\CPvar_\sigma(e). 
\]
Since the divisors $\CPvar''(e'')$ and $\CPvar(e)$ are Cartier, both $\CPvar''_\sigma(e'')$, $\CPvar_\sigma(e)\subseteq M_\R$ are translated copies of $\sigma^\vee$, cf.~\eqref{eq:LocParl}. In particular, there exists a single $e\in E$ with $\mu(e)=e''$ such that $\CPvar''_\sigma(e'')=\CPvar_\sigma(e)$. Thus $\innt_\Delta\CD''_\sigma(e'')=\innt_\Delta\CD_\sigma(e)$ and finally
\[
\innt_\Delta\CD''_\sigma(e'')=\bigcup_{\mu(e)=e''}\innt_\Delta\CD_\sigma(e) 
\]
which proves the surjectivity.
\end{proof}

By definition, $e\in\Gamma(\Delta,\CF_\sigma(\CE))$ if and only if $\Delta\subseteq\CPvar^+_\sigma(e)$. Since $\CD_\sigma(e)=r_\sigma(e)+\sigma^\vee$ and $\CD^+_\sigma(e)=r^+_\sigma(e)-r_\sigma(e)+\CD_\sigma(e)$, \cite[Claim 3.3.2 and Section 3.4]{dop} imply
\[
\Delta\setminus\CPvar^+_\sigma(e)=\varnothing\stackrel{\text{\cite{dop}}}{\Longleftrightarrow}-r_\sigma\in\sigma^\vee\Longleftrightarrow0\in\CPvar_\sigma(e)\Longleftrightarrow e\in\eval_\sigma(0).
\]
Hence~\eqref{eq:EvalSec} on page~\pageref{eq:EvalSec} yields
\begin{equation}
\label{eq:H00}
\Gamma(\Delta,\CF_\sigma)=\eval_\sigma(0)=\gH^0(U_\sigma,\CE)_0.
\end{equation}

\medskip

Furthermore, $\CPvar^+(e)=\bigcap_{\sigma\in\Sigma(\kdimX)}\CPvar^+_\sigma(e)$, cf.\ \eqref{eq:CP+} on page~\pageref{eq:CP+}, entails that $\CF(\CE)=\bigcap_{\sigma\in\Sigma}\CF_\sigma(\CE)$, that is,
\begin{equation}
\label{eq:InvLim}
\CF(\CE)=\varprojlim_{\sigma\in\Sigma}\CF_\sigma(\CE).
\end{equation}
Here, the inverse limit is taken with respect to the directed poset $\Sigma$ with $\tau\leq\sigma$ if and only if $\tau$ is a face of $\sigma$.

\medskip

The following proposition will be used in \ssect{subsubsec:Functors}.

\begin{proposition}
\label{prop:Acyclic}
We have 
\[
\gH^\ell\hspace{-0.5ex}\big(\Delta,\CF_\sigma(\CE)\big)=0\quad\text{for}\quad\ell\geq1.
\]
\end{proposition}

\begin{remark}
By quasi-coherence, $\CE_\sigma=\CE|_{U_\sigma}$ is acyclic over the open affines $U_\sigma$, $\sigma\in\Sigma$. Hence Proposition~\ref{prop:Acyclic} is a special case of Theorem~\ref{thm:CohomSheafE}.
\end{remark}

We break up the proof of Proposition~\ref{prop:Acyclic} into several intermediate results; it will be completed with Lemma~\ref{lem:LocalVan} below.

\medskip

By results of Klyachko~\cite{klyachko} a toric sheaf over a smooth, projective variety admits a finite resolution by locally free toric sheaves of finite rank which split into a direct sum of line bundles. Using Weil decorations we obtain a more refined result. Recall 
\[
\CE_{\S,\T}=\overline\S\otimes_\kk\CO_X(\T)
\]
from~\eqref{eq:rem:CanSeq} of Remark~\ref{rem:CanSeq}. If $\S=\T$ we simply write $\CE_\S$ for $\CE_{\S,\S}$. These sheaves are the building blocks for a complex associated with $\CE$. To define it we introduce the following notation. For $\S$, $\T\in\Strat$ with $\S\leq\T$ we let
\begin{equation}
\label{eq:DefCH}
\opn{ch}_\ell(\S,\T):=
\{\S=\S_0<\S_1<\ldots<\S_{\ell-1}<\S_\ell=\T\mid \S_1,\ldots,\S_{\ell-1}
\in\Strat\}
\end{equation}
be the set of strict chains of length $\ell$ in $\Strat$ starting at $\S_0=\S$ and terminating at $\S_\ell=\T$. For example, we find
\begin{equation}
\label{eq:ExamCh}
\opn{ch}_0(\S,\T)=\begin{cases}
\{\S\}&\text{if }\S=\T\\
\varnothing&\text{if }\S<\T
\end{cases}
\quad\text{and}\quad
\opn{ch}_1(\S,\T)=\begin{cases}
\varnothing&\text{if }\S=\T\\
\{\S<\T\}&\text{if }\S<\T
\end{cases}
.
\end{equation}
If $\CE_{\S_0,\S_\ell}\subseteq\bigoplus_{\S<\T}\bigoplus_{\opn{ch}_\ell(\S,\T)}\CE_{\S,\T}$ is the summand corresponding to the chain $\S_0<\ldots<\S_\ell$, we also write $\CE_{\S_0,\ldots,\S_\ell}$ for $\CE_{\S_0,\S_\ell}$.

\begin{lemma}
\label{lem:CanRes}
Let $\CE$ be a toric reflexive sheaf of rank $r$. Then for all $\sigma\in\Sigma$ we have the long exact sequence
\begin{equation}
\label{eq:CanRes}
0\to\bigoplus_{\S\leq\T}\bigoplus_{\opn{ch}_r(\S,\T)}\!\CE_{\S,\T}(U_\sigma)\to\cdots\to\bigoplus_{\S\leq\T}\bigoplus_{\opn{ch}_0(\S,\T)}\!\CE_{\S,\T}(U_\sigma)\to\CE(U_\sigma)\to0,
\end{equation}
where for $\ell\geq1$ the differentials are induced by the direct sum of sheaf maps
\[
\begin{tikzcd}[ampersand replacement=\&]
\CE_{\S_0,\ldots,\S_\ell} = \CE_{\S_0,\S_\ell}\ar[rr, "{\left(\begin{smallmatrix}1 & -1 & \cdots & \pm 1\end{smallmatrix}\right)}"] \&\&\displaystyle\bigoplus_{i=0}^\ell\CE_{\S_0,\ldots \widehat\S_i\ldots,\S_{\ell-1}}= \CE_{\S_1,\S_\ell}\oplus\CE_{\S_0,\S_\ell}^{\oplus(\ell-1)}\oplus \CE_{\S_0,\S_{\ell-1}},
\end{tikzcd}
\]
and $\bigoplus_{\S\leq\T}\bigoplus_{\opn{ch}_0(\S,\T)}\!\CE_{\S,\T}=\bigoplus_\S\CE_\S\stackrel{(1\ldots1)}{\longrightarrow}\CE$.
\end{lemma}

\begin{proof}
For $0\leq i\leq\ell$, the sum of the signed identities 
\[
\begin{tikzcd}
\CE_{\S_0,\ldots,\S_\ell}\ar[r,"(-1)^i"]&\CE_{\S_0,\ldots \widehat\S_i\ldots, \S_\ell}
\end{tikzcd}
\]
defines a usual \v Cech-type differential so that we have indeed a complex. To check exactness we first note that 
\[
\begin{split}
\CE_{\S_0,\ldots,\S_\ell}(U_\sigma) = \CE_{\S_0,\S_\ell}(U_\sigma) &=
\overline\S_0\otimes\CO_X(\S_\ell)(U_\sigma)\\
&=\overline\S_0\otimes\kk[(r_\sigma(\S_\ell)+\sigma^\vee)\cap M] \subseteq E\otimes\kk[M]
\end{split}
\]
where $x^{r_\sigma(\S_\ell)}$ is the local generator of $\CO(\S_\ell)|_{U_\sigma}$.

\medskip

As everything is $M$-graded, we can check exactness degreewise. We therefore fix $m\in M$ and note that
\[
\renewcommand{\arraystretch}{1.3}
\CE_{\S_0,\ldots,\S_\ell}(U_\sigma)_m = 
\begin{cases}
\overline\S_0\cdot x^m & \text{if $m\in r_\sigma(\S_\ell)+\sigma^\vee$}, \\
0 & \text{if otherwise}.
\end{cases}
\]
Taking the joins over the strata $\S$ with $m\in r_\sigma(\S)+\sigma^\vee$ yields a maximal stratum $\S(m)$. The resulting complex is given by 
\begin{equation}
\label{eq:CanSeq}
\begin{tikzcd}[column sep=7pt]
0\ar[r]&\bigoplus\limits_{\S_0<\cdots < \S_{\ell(m)} \leq \S(m)}\ \CE_{\S_0,\ldots,\S_{\ell(m)}}(U_\sigma)_m\ar[r]\ar[d,equal]&\cdots\ar[r]&\bigoplus\limits_{\S \leq \S(m)}\CE_\S(U_\sigma)_m\ar[r]\ar[d,equal]&\CE(U_\sigma)_m\ar[r]\ar[d,equal]&0\\
0\ar[r]&\bigoplus\limits_{\S_0<\cdots < \S_{\ell(m)} \leq \S(m)}\ \overline\S_0\ar[r]&\cdots\ar[r]&\bigoplus\limits_{\S \leq \S(m)}\overline\S\ar[r]&\CE(U_\sigma)_m\ar[r]&0
\end{tikzcd}
\end{equation}
where $\ell(m)$ is the maximal length among all chains $\S_0 < \cdots < \S_{\ell}$ such that $\CE_{\S_0,\ldots,\S_\ell}(U_\sigma)_m \neq 0$. Since $\CE(U_\sigma)=\sum_\S\overline\S\otimes\CO(D(\S))(U_\sigma)$, the sum on the right hand side being taken in $E \otimes \kk[M]$, we get
\[
\CE(U_\sigma)_m=\overline{\S(m)}\cdot x^m.
\]
We call this complex $\CC_\bullet$ and decompose its terms $C_j$ into
\[ 
\begin{split}
\CC_j&=\quad\bigoplus_{\mathclap{\S_0<\cdots<\S_j\leq\S(m)}}\,\CE_{\S_0,\ldots,\S_j}(U_\sigma)_m=:\CC^=_j \oplus\CC^<_j,\quad 0\leq j\leq\ell(m),
\end{split}
\]
where $\CC^=_j$ means taking the sum with $\S_j=\S(m)$ and $\CC^<_j$ means taking the sum with $\S_j<\S(m)$. The differential restricts to an isomorphism $\CC^=_{j+1}\xrightarrow{\sim}\CC^<_j$ being the direct sum of
\[
\CE_{\S_0,\ldots,\S_j,\S_{j+1}=\S(m)}(U_\sigma)_m\xrightarrow{\pm1}
\CE_{\S_0,\ldots,\S_j<\S(m)}(U_\sigma)_m.
\]
So starting with the chains of maximal length $\ell(m)$, we may split off the direct summand $\CC^=_{\ell(m)}=\CC^<_{\ell(m)-1}$ and consider the homotopic complex
\[
\begin{tikzcd}[row sep=0ex]
&\CC^=_{\ell(m)-1}&\CC^=_{\ell(m)-2}\\
0\to\cancel{\CC^=_{\ell(m)}}\ar[rd,"\sim"']\ar[ru]&\oplus\ar[r]&\oplus\ar[r]&\cdots\\
&\cancel{\CC^<_{\ell(m)-1}}&\CC^<_{\ell(m)-2}
\end{tikzcd}
.
\]
By repeating this procedure $(\ell(m)-1)$-times, we are left with the terms
\[
\begin{tikzcd}
0\to\fbox{$\CC^=_1=\CE_{\S(m)}(U_\sigma)_m=\overline{\S(m)}$}\ar[r,"\sim"]&\CE(U_\sigma)_m\to0. 
\end{tikzcd}
\]
Consequently, the original complex was exact, too.
\end{proof}

\begin{corollary}
\label{coro:CanSeq}
For any toric sheaf $\CE$ of rank $r$, the complex
\begin{equation}
\label{eq:ResLB}
0\to\bigoplus_{\S\leq\T}\bigoplus_{\opn{ch}_r(\S,\T)}\!\CE_{\S,\T}\to\cdots\to\bigoplus_{\S\leq\T}\bigoplus_{\opn{ch}_0(\S,\T)}\!\CE_{\S,\T}\to\CE\to0
\end{equation}
is exact. Furthermore, all line bundles occurring in this complex are nef if $\CE$ is positively decorated, and are ample, if $\CE$ is amply decorated.
\end{corollary}

We break up the sequence~\eqref{eq:ResLB} into short exact sequences and apply the exact functor $\CF_\sigma$, cf.\ Lemma~\ref{lem:ExacFunc}. The associated long exact sequences in cohomology show that Proposition~\ref{prop:Acyclic} boils down to the

\begin{lemma}
\label{lem:LocalVan}
Let $\CE=\CO_X(\nablap-\nablam)$ be a line bundle so that $\CE^+=\CO_X(\nablap)$ is ample and $\Delta=\nablam$ is nef. Further, put $\CF_\sigma:=\CF_\sigma(\CE)$ for $\sigma\in\Sigma$. Then $\gH^\ell(\Delta,\CF_\sigma)=0$ for all $\ell\geq1$. 
\end{lemma}

\begin{proof}
The proof can be carried out along the lines of Example~\ref{subsubsec:cohomSheafL} 
and is a direct consequence of~\cite[3.3.2]{dop}. We sketch a proof for the convenience of the reader.  

\medskip

The local Weil decorations of $\CO_X(\nablam)$ and $\CO_X(\nablap)$ are given by $r^-_\sigma+\sigma^\vee=\nablam+\sigma^\vee$ and $r^+_\sigma+\sigma^\vee=\nablap+\sigma^\vee$. Then
\[
\renewcommand{\arraystretch}{1.3}
\CF(U)=\begin{cases}\kk&\text{if }U\subseteq\nablap+\sigma\dual,
\\0&\text{if otherwise}.
\end{cases}
\]
To compute the stalks we set
\[
S(\sigma):=\nablam\setminus(\nablap+\sigma\dual) 
\]
and obtain
\[
\renewcommand{\arraystretch}{1.3}
(\CF)_u=\begin{cases}
\kk&\text{if }u\notin\overline{S(\sigma)},\\
0&\text{if }u\in\overline{S(\sigma)},
\end{cases}
\]
cf.\ \ssect{subsec:ConstSheav}. 
Our claim boils down to the observation that $S(\sigma)$ is either empty or retractible to the point $r_\sigma$, and so is therefore $\overline{S(\sigma)}$.
\end{proof}

\section{Proof of Theorem~\ref{thm:CohomSheafE}}
\label{sec:Proof}
Let $\CE$ be a toric sheaf such that $\CE^+=\CE(\Delta)$ is amply decorated, and $\CF=\CF(\CE)$ its associated constructible sheaf over $\Delta$. In this section we set out to prove the isomorphism 
\[
\gH^\bullet(X,\CE)_m\cong\gH^\bullet\big(\Delta+m,\CF_m(\CE)\big).  
\]
To simplify notation we shall restrict without loss of generality to the case $m=0$. 

\medskip

The crucial link between $\CE$ and the so-called local sheaves $\CF_\sigma=\CF_\sigma(\CE)$ from Section~\ref{sec:CohomTS} is given by
\[
\gH^0(U_\sigma,\CE)_0=\Gamma(\Delta,\CF_\sigma),  
\]
cf.\ \eqref{eq:H00} on page~\pageref{eq:H00}. The proof proceeds essentially in two steps. As in \ssect{subsec:LocalSheaf} we consider the inverse limit $\varprojlim=\varprojlim_{\sigma\in\Sigma}$. Abbreviating $\R\Gamma$ to $\H$, the first step then consists in showing that
\[
\H(X,\CE)_0=\R\varprojlim\big(\gH^0(U_\sigma,\CE)_0\big)=\R\varprojlim\big(\Gamma(\Delta,\CF_\sigma)\big)=\H(\Delta,\R\varprojlim(\CF_\sigma)\big),
\]
cf.\ \ssect{subsec:SheavesFanSpaces} and in particular~\eqref{eq:Step1}. Since we are dealing with vector spaces, $\H$ is quasi-isomorphic to its cohomology $\gH^\bullet$; we shall occasionally use $\H$ to stress that we are speaking about a derived functor. The second step consists in proving
\[
\R\varprojlim(\CF_\sigma)=\varprojlim(\CF_\sigma)\stackrel{\eqref{eq:InvLim}}{=}\CF, 
\]
cf.\ \ssect{subsec:CCT}.

\subsection{Sheaves on fan spaces}
\label{subsec:SheavesFanSpaces}
A good reference for this subsection is~\cite{ladkani}.

\subsubsection{Sheaves}
\label{subsubsec:Sheaves}
Any fan $\Sigma$ can be regarded as a finite category via its poset structure, namely
\[
\renewcommand{\arraystretch}{1.3}
\opn{Ob}(\Sigma)=\{\sigma\mid\sigma\in\Sigma\},\quad\Hom_\Sigma(\tau,\sigma)=\begin{cases}\text{inclusion }\tau\to\sigma&\text{if }\tau\leq\sigma,\\\varnothing&\text{if otherwise}.\end{cases}
\]
Further, we endow $\Sigma$ with the Alexandrov topology which declares a subset $\Sigma'$ of $\Sigma$ to be open if and only if $\Sigma'$ defines a subfan of $\Sigma$. The subfans $[\sigma]$ generated by $\sigma\in\Sigma$ give rise to irreducible open subsets, and every irreducible open set is of this form.

\begin{remark}
This topology makes $\Sigma$ a finite $T_0$-space. Conversely, any finite $T_0$ space arises from a finite poset, see~\cite{ladkani}. Be aware that the partial order used there is inverse to ours, that is, $\sigma\leq\tau$ holds in [ibid.] if and only if $\tau$ is a face of $\sigma$.  
\end{remark}

Let $\mb A$ be an abelian category and
\[
\Sh(\Sigma,\mb A):=\,\text{the category of $\mb A$-valued sheaves over $\Sigma$}.
\]
Since $[\sigma]$ is the minimal open set containing $\sigma$, any sheaf $\mc G$ in $\Sh(\Sigma,\mb A)$ is determined by its stalks $\mc G_\sigma=\Gamma([\sigma],\mc G)$ at $\sigma\in\Sigma$. Consequently, we may regard $\mc G$ as a covariant functor $\Sigma^\opp\to\mb A$, that is, we can identify $\Sh(\Sigma,\mb A)$ with the functor category $\Func(\Sigma^\opp,\mb A)$. In particular, the global section functor
\[
\Gamma_\Sigma\colon\Sh(\Sigma,\mb A)\to\mb A,\quad\Gamma_\Sigma(\mc G):=\Gamma(\Sigma,\mc G)=\varprojlim_{\sigma\in\Sigma}\mc G_\sigma 
\]
is just the natural inverse limit $\varprojlim\colon\Func(\Sigma^\opp,\mb A)\to\mb A$. 

\medskip

In the sequel, we shall restrict to the abelian categories $\mb A=\VecDelta$, the $\kk$-sheaves over $\Delta$, and $\mb A=\VecK$, the category of $\kk$-vector spaces. We write
\[
\Sh(\Sigma,\Delta):=\Sh(\Sigma,\VecDelta)\quad\text{and}\quad\Sh(\Sigma):=\Sh(\Sigma,\VecK)
\]
and view elements of $\Sh(\Sigma,\Delta)$ as ``double sheaves'' on $\Sigma$ and $\Delta$. Since both $\VecDelta$ and $\VecK$ are abelian with enough injectives, so are $\Sh(\Sigma,\Delta)$ and $\Sh(\Sigma)$, cf.~\cite[Corollary 2.2]{ladkani}. We regard the system of $\kk$-linear sheaves $\CF_\sigma(\CE)$ associated with $\CE$ as an object in $\Sh(\Sigma,\Delta)$ via
\begin{equation}
\label{eq:FFunctor}
\CF_\kbb\colon\Sigma^\opp\to\VecDelta,\quad\sigma\mapsto\CF_\sigma. 
\end{equation}
Indeed, we have an inclusion $\CF_\sigma(\CE)\into\CF_\tau(\CE)$ whenever $\tau\leq\sigma$ so that 
\[
\CF_\kbb\in\Func(\Sigma^\opp,\VecDelta)=\Sh(\Sigma,\Delta).  
\]
In particular, 
\[
\Gamma(\Sigma,\CF_\kbb)=\CF_\kbb(\Sigma)=\varprojlim\CF_\sigma(\CE)=\CF(\CE).
\]

\subsubsection{Functors}
\label{subsubsec:Functors}
A map $f\colon\Sigma\to\Sigma'$ between fans is continuous if and only if it is order preserving, that is, $\tau\leq\sigma$ implies $f(\tau)\leq f(\sigma)$. A continuous map gives rise to the usual functors $f_*$, $f_!\colon\Sh(\Sigma,\mb A)\to\Sh(\Sigma',\mb A)$ and $f^{-1}\colon\Sh(\Sigma',\mb A)\to\Sh(\Sigma,\mb A)$. Namely, we have 
\begin{equation}
\label{eq:NatMaps}
(f^{-1}\mc G')_\sigma=\mc G'_{f(\sigma)},\quad(f_*\mc G)_{\sigma'}=\varprojlim_{f(\sigma)\leq\sigma'}\mc G_\sigma,\quad
(f_!\mc G)_{\sigma'}=\varprojlim_{f(\sigma)\geq\sigma'}\mc G_\sigma,
\end{equation}
for which the adjunctions $f_!\adjF f^{-1}\adjF f_*$ hold true, that is
\begin{align}
\Hom_{\Sh(\Sigma,\mb A)}(f^{-1}\mc G',\mc G)&\cong\Hom_{\Sh(\Sigma',\mb A)}(\mc G',f_*\mc G)\nonumber\\[-1.8ex]
&\label{eq:Adjunc}\\[-1.8ex]
\Hom_{\Sh(\Sigma,\mb A)}(\mc G,f^{-1}\mc G')&\cong\Hom_{\Sh(\Sigma',\mb A)}(f_!\mc G,\mc G'),\nonumber
\end{align}
see~\cite[(2.1)]{ladkani}. In particular, $f_*$ is left- and $f_!$ is right-exact. Furthermore, $f^{-1}$ is exact.

\medskip

Next, let $\bullet$ denote an abstract singleton which we consider as a 
trivial poset, and let $f\colon\Sigma\to\bullet$ be the constant map. 
For $\mc G\in\Sh(\Sigma,\mb A)$ and $A\in\mb A$ we have
\[
f_*(\mc G)=\Gamma(\Sigma,\mc G)\quad\text{and}\quad f^{-1}(A)=\text{ constant sheaf on $\Sigma$ with value $A$.}
\]
On the other hand, the locally closed embedding $\iota_\sigma\colon\bullet\to\Sigma$ whose image is $\{\sigma\}=[\sigma]\setminus\partial\sigma$, where $\partial\sigma$ is the fan defined by the faces strictly contained in $\sigma$, yields
\[
\renewcommand{\arraystretch}{1.3}
\iota_\sigma^{-1}(\mc G)=\mc G_\sigma,\quad(\iota_{\sigma*}A)_{\sigma'}=\begin{cases}A&\text{if }\sigma'\geq\sigma\\0&\text{if otherwise}\end{cases},\quad(\iota_{\sigma!}A)_{\sigma'}=\begin{cases}A&\text{if }\sigma'\leq\sigma\\0&\text{if otherwise}\end{cases}.
\]

\begin{lemma}
\label{lem:BasicsFanCat}
The functors $\iota^{-1}_\sigma$ and $\Gamma(\Sigma,\cdot)=\varprojlim$ map injectives to injectives.
\end{lemma}

\begin{proof}
Both $\iota^{-1}_\sigma$ and $f_*=\Gamma(\Sigma,\cdot)$ are right-adjoint to exact functors. For the functor $f_*=\Gamma(\Sigma,\cdot)=\varprojlim$ this follows from the first adjunction formula~\eqref{eq:Adjunc}. For $\iota^{-1}$, this follows from the second adjunction formula~\eqref{eq:Adjunc} and the fact that $\iota_{\sigma!}$ is exact as $\iota_\sigma$ is a locally closed embedding~\cite[II.6.3]{iversen}. 
\end{proof}

\begin{lemma}
\label{lem:DeltaTop}
The functor
\begin{equation}
\label{eq:DeltaTop}
\Gamma^\Delta\colon\Sh(\Sigma,\Delta)\to\Sh(\Sigma),\quad\Gamma^\Delta(\mc G)_\sigma:=\Gamma(\Delta,\mc G_\sigma)
\end{equation}
is left-exact and maps injectives to injectives. Furthermore, its right-derived functor satisfies
\[
\R\Gamma^\Delta(\mc F_\kbb)=\Gamma^\Delta(\mc F_\kbb) 
\]
for the sheaf $\mc F_\kbb$ defined in~\eqref{eq:FFunctor}. 
\end{lemma}

\begin{proof}
Since $\Gamma(\Delta,\cdot) \colon \VecDelta\to\VecK$ has a left adjoint which is exact, so does $\Gamma^\Delta\colon\Sh(\Sigma,\Delta)\to\Sh(\Sigma)$. Namely, $d^{-1}\colon\Sh(\Sigma)\to\Sh(\Sigma,\Delta)$ sends $\CG \in \Sh(\Sigma)$ to the functor $\sigma \mapsto \underline{\CG_\sigma}$ where $\underline{\CG_\sigma}$ denotes the constant sheaf on $\Delta$ associated to the vector space $\CG_\sigma$. Consequently, $\Gamma^\Delta$ maps injectives to injectives.

\medskip

Next, we prove $\R\Gamma^\Delta(\mc F_\kbb)=\Gamma^\Delta(\mc F_\kbb)$ by proceeding again stalkwise. Since $\iota^{-1}_\sigma$ is exact its right-derived functor exists and satisfies $\R\iota^{-1}_\sigma=\iota^{-1}_\sigma$. Moreover, Lemma~\ref{lem:BasicsFanCat} asserts that $\iota^{-1}_\sigma$ maps injectives to injectives. As $\iota^{-1}_\sigma\circ\Gamma^\Delta(\mc G)=\Gamma(\Delta,\mc G_\sigma)$ for all $\mc G \in \Sh(\Sigma,\Delta)$, we deduce that
\begin{align*}
\big(\R\Gamma^\Delta(\CF_\kbb)\big)_\sigma=\R\iota^{-1}_\sigma\circ\R\Gamma^\Delta(\CF_\kbb)&=\R(\iota^{-1}_\sigma\circ\Gamma^\Delta)(\CF_\kbb)\\
&=\R(\Gamma^\Delta\circ\iota^{-1}_\sigma)(\CF_\kbb)=\gH^\bullet(\Delta,\CF_\sigma).
\end{align*}
But Proposition~\ref{prop:Acyclic} implies that $\gH^\bullet(\Delta,\CF_\sigma)=\gH^0(\Delta,\mc F_\sigma)$ 
for every $\sigma\in\Sigma$ whence the claim.
\end{proof}

The next result is a simple consequence of~\cite[Proposition II.9.2]{Hartshorne}.

\begin{lemma}
\label{lem:LocalGlobalCommutes}
The following diagram commutes
\begin{equation}
\begin{tikzcd}
\Sh(\Sigma,\Delta) \ar[r, "\Gamma^\Delta"] \ar[d, "\Gamma_\Sigma"] & \Sh(\Sigma) \ar[d, "\Gamma_\Sigma"] \\
\VecDelta \ar[r, "\Gamma^\Delta"] & \VecK
\end{tikzcd}
\end{equation}
where $\Gamma^\Delta$ denotes in both cases the functor $\Gamma(\Delta,\cdot)$.
\end{lemma}

\subsubsection{The Klyachko-\v Cech complex}
\label{subsec:KCSC}
With $\Sigma$ we associate the complex
\begin{equation}
\label{eq:KCComplex}\CC(\Sigma)^\bullet\colon\quad0\to\bigoplus_{\sigma\in\Sigma(\kdimX)}\kk_\sigma
\to\bigoplus_{\tau\in\Sigma(\kdimX-1)}\kk_\tau\to\ldots\to\bigoplus_{\rho\in\Sigma(1)}\kk_\rho\to\kk_0\to 0,
\end{equation}
where in each direct sum we index the base field $\kk_\sigma:=\kk$ over the set of $d$-dimensional cones $\Sigma(d)$. The nontrivial arrows $\kk_\sigma\to\kk_\tau$ are $\pm\Id_\kk$ for facets $\tau$ of $\sigma$, the sign being determined by the natural orientation of the pair $(\tau,\sigma)$.

\medskip

For a toric sheaf $\CE$ over $X$, Klyachko~\cite{klyachko} built a subcomplex 
\begin{equation}
\label{eq:KCComplex:two}
\KK(\CE)^\bullet_m\colon\quad0\to\bigoplus_{\sigma\in\Sigma(\kdimX)}\!\gH^0(U_\sigma,\CE)_m\to\ldots\to\bigoplus_{\rho\in\Sigma(1)}\!\gH^0(U_\rho,\CE)_m\to E\to 0
\end{equation}
of $E\otimes_\kk\kk[M]\otimes\CC(\Sigma)^\bullet$ by replacing the subspace $E\otimes_\kk\kk_\sigma\cdot x^m$ with $\gH^0(U_\sigma,\CE)_m=\eval_\sigma(m)\subseteq E$, cf.~\eqref{eq:EvalSec} on page~\pageref{eq:EvalSec}. The cohomology of this complex computes the piece of degree $m$ of the sheaf cohomology of $\CE$ on $X$, that is,
\[
\gH^\bullet(\KK(\CE)^\bullet_m)\cong\gH^\bullet(X,\CE)_m. 
\]
As we focus on the case $m=0$ we set
\[
\KK(\CE)^\bullet:=\KK(\CE)^\bullet_0.
\]

\medskip

This links into the sheaf $\mc F=\varprojlim\mc F_\kbb$ as follows. Since $H^0(U_\sigma,\CE)_0 = \Gamma^\Delta(\CF_\sigma)$ by \eqref{eq:H00}, we can write the complex of~\eqref{eq:KCComplex:two} for $m=0$ as
\begin{equation}
\label{eq:BL}
\begin{tikzcd}[column sep=2ex, row sep=2ex]
0\ar[r] &
\displaystyle \bigoplus_{\mathclap{\sigma\in\Sigma(\kdimX)}} \Gamma^\Delta(\mc F_\kbb)_\sigma \ar[d, equal]
\ar[r] &
\displaystyle \bigoplus_{\mathclap{\tau\in\Sigma(\kdimX-1)}} \Gamma^\Delta(\mc F_\kbb)_\tau \ar[d, equal]
\ar[r] & \ldots \ar[r] &
\displaystyle \bigoplus_{\mathclap{\rho\in\Sigma(1)}} \Gamma^\Delta(\mc F_\kbb)_\rho \ar[d, equal]
\ar[r] &
\Gamma^\Delta(\mc F_\kbb)_0 \ar[d, equal]
\ar[r] & 
0\\
& \Gamma^\Delta(\mc F_\sigma) & \Gamma^\Delta(\mc F_\tau) & & \Gamma^\Delta(\mc F_\rho) & \Gamma^\Delta(\mc F_0) \mathrlap{=E}
\end{tikzcd}
\end{equation}
where 
we used $\mc F_0 = \underline E$ on $\Delta$. By~\cite[Section 3.5]{bl03}, this complex represents $\gH^\bullet\!\big(\Sigma,\Gamma^\Delta(\mc F_\kbb)\big)$ in the bounded derived category $\CD^b(\VecK)$. Since $\R\Gamma_\Sigma=\R\varprojlim$, 
we deduce
\begin{equation}
\label{eq:Step1}
\gH^\bullet(X,\CE)_0=(\R\varprojlim\circ\Gamma^\Delta)(\CF_\kbb)
=\H(\Delta,\R\varprojlim\mc F_\kbb).
\end{equation}
The latter equality is just the derived version of Lemma~\ref{lem:LocalGlobalCommutes}.
To conclude, it remains to prove the

\begin{proposition}
\label{prop:DerLimLoc}
The functor $\mc F$ from~\eqref{eq:FFunctor} satisfies
\[
\R\varprojlim\mc F_\kbb=\varprojlim\mc F_\kbb \left(=\mc F(\CE) \right).
\] 
\end{proposition}

This will occupy us next.

\subsection{Complexes of \v Cech type}
\label{subsec:CCT}
To compute $\R\varprojlim\mc F_\kbb$ we consider several complexes of \v Cech type. First, we prove that we can argue stalkwise to set up the initial complex.

\subsubsection{The initial complex}
For $u\in\Delta$ we consider the functor
\[
(\cdot)_u\colon\VecDelta\to\VecK 
\]
obtained by taking stalks at $u$, as well as the induced functor
\[
(\cdot)_u\colon\Sh(\Sigma,\Delta)\to\Sh(\Sigma).
\]
Note that $(\cdot)_u$ is exact on both $\VecDelta$ and $\Sh(\Sigma,\Delta)$.

\begin{lemma}
\label{lem:reductionStalks}
For all $u\in\Delta$, we have
\[
(\cdot)_u\circ\varprojlim=\varprojlim\circ(\cdot)_u
\]
as functors from $\Sh(\Sigma,\Delta)\to\VecK$. Furthermore, $(\cdot)_u$ maps injectives to $\varprojlim$-acyclics.
\end{lemma}

\begin{proof}
The first statement follows from
\[
0\to\varprojlim_{\sigma\in\Sigma}\mc G_\sigma\to\bigoplus_{\sigma\in\Sigma(n)}\mc G_\sigma\to\bigoplus_{\tau\in\Sigma(n-1)}\mc G_\tau
\]
and talking stalks. Next assume that $\mc G\in\Sh(\Sigma,\Delta)$ is injective. In particular, $\mc G$ is flasque so that $\varprojlim\mc G_\sigma=\mc G(\Sigma)\to\mc G(\Sigma')=\varprojlim_{\Sigma'}\mc G$ is surjective for all subfans $\Sigma'$. Since $(\cdot)_u$ commutes with $\varprojlim$, this implies that $\mc G_u(\Sigma)\to\mc G_u(\Sigma')$ is surjective, that is, $\mc G_u$ is flasque. Hence it is acyclic for $\varprojlim=\Gamma(\Sigma,\cdot)$.
\end{proof}

We now turn to the proof of Proposition~\ref{prop:DerLimLoc}. By the previous lemma, 
\[
(\R\varprojlim\mc F_\kbb)_u=\R\varprojlim\big(\mc F_{\kbb u}\big)
\] 
and we need to show that the right hand side equals $\varprojlim\big(\mc F_{\kbb u}\big)=\mc F(\CE)_u$ for all $u\in\Delta$. For this we put
\begin{align}
F(\sigma):=\CF_{\sigma,u}&=\{e\in E\mid u\in\innt_{\Delta}\!\CPvar^+_\sigma(e)\}\nonumber\\
&=\bigcup\{\S\in\Strat_\sigma\mid u\in\innt_{\Delta}\!\CPvar^+_\sigma(\S)\}.\label{eq:Fsig}
\end{align}
Consider the complex
\begin{equation}
\label{eq:BLComplex}
\bigoplus_{\sigma\in\Sigma(\kdimX)}F(\sigma)\to\bigoplus_{\tau\in\Sigma(\kdimX-1)}F(\tau)\to\ldots\to
\bigoplus_{\rho\in\Sigma(1)}F(\rho)\to F(0)=E\to0,
\end{equation}
As for~\eqref{eq:BL}, this complex is isomorphic to $\R\varprojlim(\mc F_{\kbb u})$ in the bounded derived category $\CD^b(\VecK)$. The kernel of $\bigoplus_{\sigma\in\Sigma(\kdimX)}F(\sigma)\to\bigoplus_{\tau\in\Sigma(\kdimX-1)}F(\tau)$ is just 
\[
\varprojlim(\CF_{\kbb u})=\CF(\CE)_u=\bigcap_{\rho\in\Sigma(1)}F(\rho); 
\]
the latter equality follows from $\CD^+(e)=\bigcap_{\rho\in\Sigma(1)}\CD^+_\rho(e)$. Thus, showing exactness of the augmented complex
\begin{equation}
\label{eq:FComp}
\opn{K}(F)^\bullet\colon\;0 \to\hspace{-0.2em}\bigcap_{\rho\in\Sigma(1)}\hspace{-0.4em} F(\rho)\to\bigoplus_{\sigma\in\Sigma(\kdimX)}\hspace{-0.2em}F(\sigma)\to\ldots\to\bigoplus_{\rho\in\Sigma(1)}\hspace{-0.2em}F(\rho)\to E\to0
\end{equation}
will immediately imply Proposition~\ref{prop:DerLimLoc}. For this it will be convenient to pass from~\eqref{eq:KCComplex} to the augmented complex
\begin{equation}
\label{eq:KCNew}
\CC(\Sigma)^\bullet\colon\quad0\to\kk\to\bigoplus_{\sigma\in\Sigma(\kdimX)}\hspace{-0.2em}\kk_\sigma\to\hspace{-0.4em}\bigoplus_{\tau\in\Sigma(\kdimX-1)}\hspace{-0.6em}\kk_\tau\to\ldots\to\bigoplus_{\rho\in\Sigma(1)}\hspace{-0.2em}\kk_\rho\to\kk_0\to0
\end{equation}
which we continue to denote $\CC(\Sigma)^\bullet$ to keep notation tight. In particular, we can regard $\opn{K}(F)^\bullet$ as a subcomplex of $E\otimes\CC(\Sigma)^\bullet$. Since $\tau\leq\sigma$ implies $\CPvar^+_\tau(e)=\CPvar^+_\sigma(e)+\tau^\vee$ we have $F(\sigma)=\bigcap_{\rho\in\sigma(1)}F(\rho)$. It remains to understand $F(\rho)$ for $\rho\in\Sigma(1)$. 

\medskip

Recall that $\innt_{\Delta}\CPvar^+_\rho(e)$ in Definition~\eqref{eq:Fsig} refers to the interior points of $\CPvar^+_\rho(e)\cap\Delta$ relative to $\Delta$. This can lead to a different behaviour depending on whether $\CPvar^+_\rho(e)$ intersects $\Delta$ in a face or not.

\begin{example}
\label{exam:SpecRel}
We revisit~\eqref{ssec:cohom:O-three} and consider the line bundle $\CO_{\PP^2}(-3D_{\rho_0})=\CO_{\PP^2}(D_{\rho_0}-4D_{\rho_0})$. Here, $\CPvar^+(e)=\Delta_1$ is the standard simplex spanned by $[0,0]$, $[1,0]$ and $[0,1]$. For $\Delta=4\Delta_1$ we obtain
\[
\innt_\Delta\CPvar^+_{\rho_0}=\Delta\cap\{\langle u,\rho_0\rangle<1\}\quad\text{and}\quad\innt_\Delta\CPvar^+_{\rho_i}=\Delta\cap\{\langle u,\rho_i\rangle\geq0\},\quad i=1,2.
\]
On the other hand, $\Delta=4\Delta_1-[1,0]$ gives $\innt_\Delta\CPvar^+_{\rho_1}=\Delta\cap\{\langle u,\rho_1\rangle>0\}$, see also Figure~\ref{fig:LeqRho}.
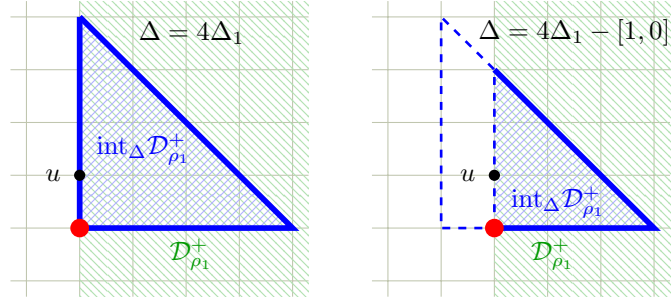
\begin{figure}[ht]
\newcommand{\scaleB}{0.7}
\begin{tikzpicture}[scale=\scaleB]
\draw[color=oliwkowy!40] (-1.3,-1.3) grid (4.3,4.3);
\fill[pattern color=green!30, pattern=north west lines]
  (0,-1.3) -- (4.3,-1.3) -- (4.3,4.3) -- (0,4.3);
\fill[pattern color=blue!30, pattern=north east lines]
  (0,0) -- (4,0) -- (0,4) -- (0,0);
\draw[thick, color=blue]
  (1.2,1.5) node{$\innt_\Delta\!\CPvar^+_{\rho_1}$};
\draw[thick, color=green] 
  (2.1,-0.5) node{$\CPvar^+_{\rho_1}$};
\draw[line width=0.75mm, blue]
  (0,0) -- (4,0) -- (0,4);
\draw[line width=0.75mm,blue]
  (0,0) -- (0,4);
\draw[thick, color=black] 
  (2.1,3.7) node{$\Delta=4\Delta_1$};
\fill[thick, color=red]
  (0,0) circle (5pt);
\fill
  (0,1) circle (3pt);
\draw[thick, color=black] 
  (-0.5,1) node{$u$};
\end{tikzpicture}
\qquad
\begin{tikzpicture}[scale=\scaleB]
\draw[color=oliwkowy!40] (-1.3,-1.3) grid (4.3,4.3);
\fill[pattern color=green!30, pattern=north west lines]
  (1,-1.3) -- (4.3,-1.3) -- (4.3,4.3) -- (1,4.3);
\fill[pattern color=blue!30, pattern=north east lines]
  (1,0) -- (4,0) -- (1,3) -- (1,0);
\draw[thick, color=blue]
  (2.2,0.5) node{$\innt_\Delta\!\CPvar^+_{\rho_1}$};
\draw[thick, color=green] 
  (2.1,-0.5) node{$\CPvar^+_{\rho_1}$};
\draw[line width=0.75mm, blue]
  (1,0) -- (4,0) -- (1,3);
\draw[line width=0.35mm, dashed, blue]
  (1,0) -- (1,3) -- (0,4) -- (0,0) -- (1,0);
\draw[thick, color=black] 
  (2.5,3.7) node{$\Delta=4\Delta_1-[1,0]$};
\fill[thick, color=red]
  (1,0) circle (5pt);
\fill
  (1,1) circle (3pt);
\draw[thick, color=black] 
  (0.5,1) node{$u$};
\end{tikzpicture}
\caption{$\innt_\Delta\CPvar^+_{\rho_1}$ for $4\Delta_1$ (left hand side) and $4\Delta_1-[1,0]$ (right hand side). In the latter case, it comprises the blue triangle except the dashed border which also includes the origin and $[0,3]$.}
\label{fig:LeqRho}
\end{figure}
\end{example}

In order to cover these different cases effectively, we pose the following relations on the reals depending on the $u\in\Delta$ fixed above:
\begin{equation}
\label{eq:EqRa}
\renewcommand{\arraystretch}{1.3}
a\gneqR b\quad\stackrel{\text{Def}}{\iff}\quad\begin{cases}a > b &\text{if }\langle u,\rho\rangle>\min\langle\Delta,\rho\rangle\,(\text{the {\em standard case}, e.g., }u\in\innt\Delta)\\
a\geq b &\text{if }\langle u,\rho\rangle=\min\langle\Delta,\rho\rangle
\end{cases}
\end{equation}
and
\begin{equation}
\label{eq:EqRb}
\renewcommand{\arraystretch}{1.3}
a\geqR b\quad\stackrel{\text{Def}}{\iff}\quad\begin{cases}
a\geq b & \text{if }\langle u,\rho\rangle>\min\langle\Delta,\rho\rangle\,(\text{the {\em standard case}})\\
a>b&\text{if }\langle u,\rho\rangle=\min\langle\Delta,\rho\rangle
\end{cases}
\end{equation}

\medskip

{\bf Example~\ref{exam:SpecRel} cont.} 
Using the notation above we have
\[
\innt_\Delta\CPvar^+_{\rho_1}=\{u\in\Delta\mid\langle u,\rho_1\rangle\gneqRa0\}. 
\]
Concretely, we pick $u=[0,1]$. If $\Delta=4\Delta_1$, we have $\min\langle\Delta,\rho_1\rangle=0$, that is, we are not in the standard case. Hence $\langle u,\rho_1\rangle\gneqRa0$ now means $\langle u,\rho_1\rangle\geq0$. If, on the other hand, $\Delta=4\Delta_1-[1,0]$, then $\min\langle\Delta,\rho_1\rangle=-1$ so that $\langle u,\rho_1\rangle\gneqRa0$ actually means $\langle u,\rho_1\rangle>0$. See also Figure~\ref{fig:LeqRho}.

\begin{remark}
The notation is inspired by the standard case. This comes at the price that if $\langle u,\rho\rangle=\min\langle\Delta,\rho\rangle$, $a\gneqR b$ can happen \emph{without} $a\geqR b$. Be that as it may, our notation avoids treating the nonstandard case separately; if confused the reader may just wish to assume the standard case, e.g., $u\in\Delta\setminus\partial\Delta$.
\end{remark}

Coming back to the proof of Proposition~\ref{prop:DerLimLoc}. we can rewrite~\eqref{eq:Fsig} as
\[
F(\rho)=\{e\in E\mid \langle u,\rho\rangle\gneqR\min\langle \CPvar^+_\rho(e),\rho\rangle\}= \{e\in E\mid \CPvar^+_\rho(e)\not\subseteq[\rho\geqR u]\},
\]
where, according to~\eqref{eq:Mad},
\[
[\rho\geqR u]=\{v\in M_\R\mid\langle v,\rho\rangle\geqR\langle u,\rho\rangle\}. 
\]

\begin{example}
\label{exam:LBCase}
Considering the case of a line bundle of~\eqref{subsubsec:cohomSheafL}, that is, $E=\kk$ and $\DeltaP=\CPvar^+(1)$, we have
\[
F(\rho)=\kk\iff\DeltaP\not\subseteq[\rho\geqR u]\iff\DeltaP\cap([\rho\lneqR u])\neq\emptyset.
\]
In particular, $u\in\innt_{\Delta}(\DeltaP)$ implies $\DeltaP\not\subseteq[\rho\geqR u]$ for all $\rho\in\Sigma(1)$, whence $F(\sigma)=\kk$ for all $\sigma\in\Sigma$. As a result, $\opn{K}(F)^\bullet$ is in this case just the (augmented) complex $\CC(\Sigma)^\bullet$, that is,
\begin{equation}
\label{eq:Case1}
0\to\fbox{$\bigcap_{\rho\in\Sigma(1)}F(\rho)=\kk$}\to\bigoplus_{\sigma\in\Sigma(\kdimX)}\hspace{-0.2em}\kk_\sigma\to\hspace{-0.4em}\bigoplus_{\tau\in\Sigma(\kdimX-1)}\hspace{-0.6em}\kk_\tau\to\ldots\to\bigoplus_{\rho\in\Sigma(1)}\hspace{-0.2em}\kk_\rho\to\kk_0\to0
\end{equation}
(recall that the indices record the corresponding copy of $\kk$). Since $\Sigma$ is the normal fan of the ample polytope $\DeltaV$, each $\sigma\in\Sigma$ corresponds to a face $\phi=\DeltaV(\sigma)$ of $\DeltaV$, cf.~\eqref{eq:sigmaCorner}. Accordingly, we also write $\sigma=\sigma(\phi)$ and $F(\sigma)=F(\phi)$. Further, we let $\DeltaV_i$ be the faces of dimension $i$. Our complex can thus be rewritten as
\begin{equation}
\label{eq:DeltaComplex}
\CC(\Delta)^\bullet\colon\quad0 \to\kk_\emptyset\to\bigoplus_{v\in\DeltaV_0}\hspace{-0.2em}\kk_v\to\hspace{-0.2em}\bigoplus_{e\in\DeltaV_1}\hspace{-0.2em}
\kk_e\to\ldots\to\hspace{-0.4em}\bigoplus_{f\in\DeltaV_{\kdimX-1}}\hspace{-0.6em}\kk_f\to\kk_\DeltaV\to0.
\end{equation}
Consequently, $\opn{K}(F)^\bullet$ is the complex of the reduced cohomology of the contractible polytope $\Delta$ and is therefore exact.
\end{example}

\subsubsection{The complex of a fan}
Returning to a general toric sheaf $\CE^+=\CE(\Delta)$ with associated Weil decoration $\CD^+$ and stratification $\Strat$, e.g.\ the canonical one, we take a stratum $\S\subseteq E$ and consider the full subfan $\Sigma^\S$ of $\Sigma$ defined by the rays
\[
\Sigma^\S(1):=\{\rho\in\Sigma(1)\mid \min\langle\CPvar^+(\S),\rho\rangle\lneqR\langle u,\rho\rangle\}.
\]
In the case of a line bundle, this is just the set $\{\rho\in\Sigma(1)\mid F(\rho)=\kk\}$; Example~\ref{exam:LBCase} dealt with $\DeltaP=\CPvar^+(\eta)$ and $\S=\eta=\kk^\times$.

\medskip

In general, the relation $F(\sigma)=\bigcap_{\rho\in\sigma(1)}F(\rho)$ implies that
\begin{equation}
\label{eq:FanChar}
\sigma\in\Sigma^\S\iff\S\subseteq F(\sigma)\iff\overline\S\subseteq F(\sigma).
\end{equation}

\begin{lemma}
\label{lem:Subcomplex}
The subcomplex
\begin{equation}
\label{eq:Subcomplex}
\CC^{\S\bullet}:=\CC(\Sigma^\S)^\bullet
\end{equation}
of $\CC(\Sigma)^\bullet$ in~\eqref{eq:KCNew} is exact. 
\end{lemma}

\begin{proof}
If $\Sigma^\S(1)=\Sigma(1)$, then $\CC^{\S\bullet}=\CC(\Sigma)^\bullet$ and the claim follows from the discussion above. We therefore assume that $\Sigma^\S(1)$ is a proper, nonempty subset of $\Sigma(1)$.

\medskip

The complement $\Sigma'(1)$ of $\Sigma^\S(1)$ in $\Sigma(1)$ can be described as follows. Let 
\[
\DeltaP(u):=\conv\big(\CPvar^+(\S)\cup\{u\}\big)
\]
be the convex hull of $u$ and $\CPvar^+(\S)$. Further, let 
\[
\opn{N}(u):=\{a\in N_\R\mid\langle u,a\rangle\leqR\min\langle\CPvar^+(\S),a\rangle\},
\]
that is, the closure $\overline{\opn{N}(u)}$ is the cone $\normal\big(u,\DeltaP(u)\big)$ of the normal fan $\normal\big(\DeltaP(u)\big)$ dual to the smallest face of $\DeltaP(u)$ containing $u$. Then
\[
\Sigma'(1)=\{\rho\in\Sigma(1)\mid\min\langle\CPvar^+(\S),\rho\rangle\geq_\rho\langle u,\rho\rangle\}=\Sigma(1)\cap\opn{N}(u),
\]
the intersection taking place in $N_\R$. The rays of $\Sigma'(1)$ correspond to the set of facets
\[
\Delta':=\{\phi\leq\Delta\text{ facet}\mid\sigma(\phi)\notin\Sigma^\S\}=\{\phi(\sigma)\leq\Delta\mid\S\not\subseteq F(\sigma)\}
\]
of which we think as a geometric subcomplex of $\partial\Delta$.
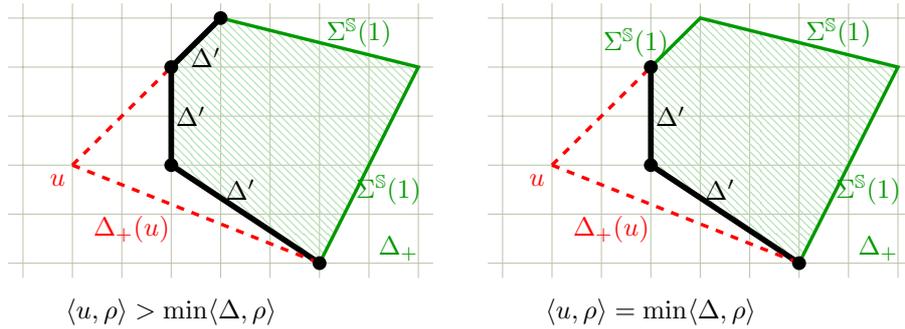
\begin{figure}[ht]
\[
\newcommand{\scaleA}{0.3} 
\newcommand{\scaleB}{0.65}
\newcommand{\spaceA}{\hspace*{5em}}
\begin{tikzpicture}[scale=\scaleB]
\draw[color=oliwkowy!40] (-1.3,-2.3) grid (7.3,3.3);
\fill[pattern color=green!30, pattern=north west lines]
  (5,-2) -- (7,2) -- (3,3) -- (2,2) -- (2,0) -- cycle;
\draw[line width=0.75mm, black] 
   (2,2) -- (2,0) -- (5,-2);
\draw[line width=0.45mm, green] 
  (5,-2) -- (7,2) -- (3,3) -- (2,2);
\draw[line width=0.75mm, black] 
   (2,2) -- (3,3);
\draw[dashed, line width=0.45mm, red]
  (0,0) -- (2,2) (0,0) -- (5,-2);
\draw[thick, color=red]
     (-0.3,-0.3) node{$u$};
\draw[thick, color=green]
     (5.8,2.7) node{$\Sigma^\S(1)$} 
     (6.4,-0.5) node{$\Sigma^\S(1)$} 
     (6.6,-1.7) node{$\DeltaP$};
\draw[thick, color=black]
     
     (2.4,1.0) node{$\Delta'$}
     (3.4,-0.5) node{$\Delta'$}
     (2.7,2.2) node{$\Delta'$};
\draw[thick, color=red]
     (1.2,-1.3) node{$\DeltaP(u)$};
\fill[thick, color=black]
     (5,-2) circle (4pt)
     (2,2) circle (4pt)
     (2,0) circle (4pt);
\fill[very thick, color=black] 
     (3,3) circle (4pt);
\draw[thick, color=black]
     (2,-3) node{$\langle u,\rho\rangle>\min\langle\Delta,\rho\rangle$};
\end{tikzpicture}
\qquad
\begin{tikzpicture}[scale=\scaleB]
\draw[color=oliwkowy!40] (-1.3,-2.3) grid (7.3,3.3);
\fill[pattern color=green!30, pattern=north west lines]
  (5,-2) -- (7,2) -- (3,3) -- (2,2) -- (2,0) -- cycle;
\draw[line width=0.75mm, black] 
   (2,2) -- (2,0) -- (5,-2);
\draw[line width=0.45mm, green] 
  (5,-2) -- (7,2) -- (3,3) -- (2,2);
\draw[dashed, line width=0.45mm, red]
  (0,0) -- (2,2) (0,0) -- (5,-2);
\draw[thick, color=red]
     (-0.3,-0.3) node{$u$};
\draw[thick, color=green]
     (1.7,2.5) node{$\Sigma^\S(1)$}
     (5.8,2.7) node{$\Sigma^\S(1)$} 
     (6.4,-0.5) node{$\Sigma^\S(1)$} 
     (6.6,-1.7) node{$\DeltaP$};
\draw[thick, color=black]
     (2.4,1.0) node{$\Delta'$}
     (3.4,-0.5) node{$\Delta'$};
\draw[thick, color=red]
     (1.2,-1.3) node{$\DeltaP(u)$};
\fill[thick, color=black]
     (5,-2) circle (4pt)
     (2,2) circle (4pt)
     (2,0) circle (4pt);
\draw[thick, color=black]
     (2,-3) node{$\langle u,\rho\rangle=\min\langle\Delta,\rho\rangle$};
\end{tikzpicture}
\]
\caption{The subcomplexes $\Delta'\subset\partial\Delta$ and $\Sigma^\S\subseteq\Sigma$. The left-hand side displays the standard case $\langle u,\rho\rangle>\min\langle\Delta,\rho\rangle$, e.g., $u$ is an inner point of $\Delta$. The right-hand side corresponds to equality.}
\label{fig:SubComplDelta}
\end{figure}
As illustrated in Figure~\ref{fig:SubComplDelta}, $\Delta'$ consists of the ``lower boundary'' of $\CPvar^+(\S)$ considered as a subset of $\DeltaP(u)$. Depending on the position of $u$, cf.\ \eqref{eq:EqRa} and~\eqref{eq:EqRb}, this might also include the facets $\varphi(\rho)$ with $\min\langle\CPvar^+(\S),\rho\rangle=\langle u,\rho\rangle$. At any rate, $\Delta'$ can be contracted to $u$.

\medskip

This brings us back to the question of the exactness of the complex 
\[
\CC^{\S\bullet}\colon\quad 0 \to 0 \to
\bigoplus_{\sigma\in\Sigma^\S(\kdimX)}\hspace{-0.2em}\kk_\sigma
\to \hspace{-0.4em}\bigoplus_{\tau\in\Sigma^\S(\kdimX-1)}\hspace{-0.6em}
\kk_\tau\to\ldots\to\bigoplus_{\rho\in\Sigma^\S(1)}\hspace{-0.2em}\kk_\rho \to \kk \to 0
\]
which fits into the short exact sequence of complexes
\[
0\to\CC^{\S\bullet}\to\CC(\Sigma)^\bullet\to\CC(\Delta')^\bullet\to0.
\]
Here, $\CC(\Delta')^\bullet$ is the complex of the reduced cohomology of the complex $\Delta'$, that is, 
\[
\CC(\Delta')^\bullet\colon\quad0 \to \kk_\emptyset\to\bigoplus_{v\in\Delta'_0}\hspace{-0.2em}\kk_v\to\hspace{-0.2em}\bigoplus_{e\in\Delta'_1}\hspace{-0.2em}
\kk_e\to\ldots\to\hspace{-0.4em}\bigoplus_{F\in\Delta'_{\kdimX-1}}\hspace{-0.6em}\kk_F\to 0 \to 0.
\]
Unlike for the complex $\CC(\Delta)^\bullet$ in~\eqref{eq:DeltaComplex}, there is no $\kk_{\Delta}$-term in $\CC(\Delta')^\bullet$ since the dimension of $\Delta'$ drops by one. Exactness of $\CC(\Sigma)^\bullet$ and $\CC(\Delta')^\bullet$ finally implies exactness of $\CC^{\S\bullet}$.
\end{proof}

\subsubsection{The strata double complex}
For a pair $(\S,\T)$ of strata with $\S\leq\T$, that is, $\S\subseteq\overline \T$, we consider the twisted subcomplex
\begin{equation}
\label{eq:Hom}
\CC^\bullet(\S,\T):=\overline\S\otimes_\kk\CC^{\T\bullet}
\end{equation}
of $E\otimes_\kk\CC(\Sigma)^\bullet$, which is exact by Lemma \ref{lem:Subcomplex}. This definition is similar in vein to~\eqref{eq:rem:CanSeq} of Remark~\ref{rem:CanSeq}. In particular,
$\CC^\bullet(\S,\T)$ behaves like the $\gHom$-functor -- it is covariant in the first and contravariant in the second argument. Indeed, for $\S\leq\S'$ and $\T'\leq\T$ we obtain natural inclusions $\overline\S\subseteq\overline\S'$ and $\Sigma^{\T}\subseteq\Sigma^{\T'}$ which induce commuting inclusion maps of complexes denoted
\[
\iota^I\colon\CC^\bullet(\S,\T)\into\CC^\bullet(\S',\T)\quad\text{and}\quad\iota^{I\!I}\colon\CC^\bullet(\S,\T)\into\CC^\bullet(\S,\T').
\]

\medskip

In order to build a \v Cech-like complex out of the strata, we let $\S_1,\ldots,\S_N$ be the minimal nontrivial strata of the lattice $\Strat$. For every pair of subsets $S\subseteq T\subseteq\{1,\ldots,N\}$ we define
\[
\S_S:=\bigvee\{\S_i\mid i\in S\}\in\Strat\quad\text{and}\quad\opn{C}^\bullet(S,T):=\opn{C}^\bullet(\S_S,\S_T).
\]
Usually, we will replace $\S_S$ by $S$ and $\S_T$ by $T$, respectively, and represent $S$ and $T$ by ordered tuples $(i_1,\ldots,i_a)$ and $(j_1,\ldots,j_b)$ with $1\leq i_k<i_{k+1}\leq N$ and $1\leq j_\ell<j_{\ell+1}\leq N$. We then construct a double complex $\CC^{\bullet,\bullet}$ of (exact) complexes over $\Sigma$ as follows. We set
\[
\opn{C}(a,b):=\bigoplus_{\substack{\scalebox{0.75}{$\sharp S=a,\,\sharp T=b$}\\\scalebox{0.75}{$S\subseteq T$}}}\opn{C}^\bullet(S,T).
\]
for $a\leq b$. Given a pair $S\subseteq T$ we sum for $p=0,\ldots,a$ the maps
\[
\begin{tikzcd}
\opn{C}^\bullet\big((i_0,\ldots,i_{p-1},i_{p+1},\ldots,i_a),(j_1,\ldots,j_b)\big)\ar[r,"(-1)^p\cdot\iota^I"]&\opn{C}^\bullet\big((i_0,\ldots,i_a),(j_1,\ldots,j_b)\big)
\end{tikzcd}
\]
to define the differential $d^I_{a,b}\colon\opn{C}(a,b)\to\opn{C}(a+1,b)$; similarly for $d^{I\!I}_{a,b}\colon\opn{C}(a,b)\to\opn{C}(a,b-1)$. This yields the commutative square
\[
\begin{tikzcd}
\opn{C}(a,b)\ar[r,"d^{I\!I}"]\ar[d,"d^I"']&\opn{C}(a,b-1)\ar[d,"d^I"]\\
\opn{C}(a+1,b)\ar[r]\ar[r,"d^{I\!I}"]&\opn{C}(a+1,b-1)
\end{tikzcd}
\]
Next we reduce this triple complex over the $(a,b)$-variables to the 
total complex $\mc C_\boxx$ defined by 
$\mc C_k:=\bigoplus_b\opn{C}(b,b+k)$, namely
\[
\begin{tikzcd}
\mc C_\boxx\,\colon&0\ar[r]&\mc C_{N-1}\ar[r]&\ldots\ar[r]&\mc C_1\ar[r]&\fbox{$\mc C_0=\bigoplus_b\opn{C}(b,b)$}\ar[r]&0.
\end{tikzcd} 
\]
Recall that  $\mc C_0=\bigoplus_\S\overline\S\otimes\CC^{\S\bullet}$ is itself a complex, cf.~\eqref{eq:Subcomplex}. In fact,~\eqref{eq:FanChar} implies that its direct summands are subcomplexes of  
\[
\opn{K}(F)^\bullet\colon\;0 \to\hspace{-0.2em}\bigcap_{\rho\in\Sigma(1)}\hspace{-0.4em} F(\rho)\to\bigoplus_{\sigma\in\Sigma(\kdimX)}\hspace{-0.2em}F(\sigma)\to\ldots\to\bigoplus_{\rho\in\Sigma(1)}\hspace{-0.2em}F(\rho)\to E\to0, 
\]
cf.~\eqref{eq:FComp} for its definition.

\begin{lemma}
\label{lem:FinalLem}
The extended complex $\mc C_\boxx\to\big[\mc C_{-1}:=\opn{K}(F)^\bullet]$ is exact, that is, $\mc C_\boxx$ is quasi-isomorphic to $\opn{K}(F)^\bullet$. 
\end{lemma}

\begin{proof}
Up to position $-1$, it is sufficient to show that for fixed $a$, the complex given by $d^{I\!I}$, namely
\[
\fbox{$\displaystyle\CC(a,a+r)=\hspace{-5pt}\bigoplus_{\substack{\scalebox{0.75}{$\sharp S=a,\,\sharp T=a+r$}\\\scalebox{0.75}{$S\subseteq T$}}}\hspace{-15pt}\CC^\bullet(S,T)$}\stackrel{d^{I\!I}}{\longrightarrow}\fbox{$\displaystyle\CC(a,a+r-1)=\hspace{-5pt}\bigoplus_{\substack{\scalebox{0.75}{$\sharp S=a,\,\sharp T'=a+r-1$}\\\scalebox{0.75}{$S\subseteq T'$}}}\hspace{-15pt}\CC^\bullet(S,T')$}
\]
is exact. As we mentioned earlier, $d^{I\!I}$ maps the $\sigma$-entries in $\CC^\bullet(S,T)$ to the $\sigma$-entries in $\CC^\bullet(S,T')$ where $T'\subseteq T$. Now the $\sigma$-entry in a given $\CC^\bullet(S,T)$-complex is either $0$ or $\overline \S\otimes\kk_\sigma$ with $\S = \S_S$, where the latter occurs if and only if
\begin{equation}
\label{eq:F}
\sigma\in\Sigma^{\S_T}\iff\S_T\subseteq F(\sigma)\iff u\in\innt_\Delta(\CPvar^+(\S_T)+\sigma^\vee).
\end{equation}
If we let 
\begin{equation}
\label{eq:LocVer}
T_\sigma:=\{i\in\{1,\ldots,N\}\mid u\in\innt_\Delta(\CPvar^+(\S_i)+\sigma^\vee)\},
\end{equation}
where we recall that the $\S_i$ are the minimal nontrivial strata, then $T_\sigma$ is the maximal subset $T\subseteq\{1,\ldots,k\}$ for which $u\in\innt_\Delta(\CPvar^+(\S_T)+\sigma^\vee)$. It follows that $d^{I\!I}$ is the direct sum of the complex maps induced by  
\begin{equation}
\label{eq:FinalComp}
\bigoplus_{\substack{\scalebox{0.75}{$\sharp T=a+r$}\\\scalebox{0.75}{$S\subseteq T\subseteq T_\sigma$}}}\overline \S\otimes_\kk\kk_\sigma\longrightarrow\bigoplus_{\substack{\scalebox{0.75}{$\sharp T'=a+r-1$}\\\scalebox{0.75}{$S\subseteq T'\subseteq T_\sigma$}}}\overline\S\otimes_\kk\kk_\sigma
\end{equation}
for each fixed $S$ and $\sigma$. Hence we have to repeat $\overline\S\otimes\kk_\sigma$ as many times as there are $T$ comprised between $S$ and $T_\sigma$ with $\sharp T=a+r$. To show exactness of this complex we distinguish two cases. 

\medskip

{\em Case 1: $S\subsetneq T_\sigma$.} The complex~\eqref{eq:FinalComp} reduces to
\[
0\to\overline\S\otimes\kk_\sigma\to\ldots\to\overline\S\otimes\kk_\sigma^{\sharp T_\sigma-\sharp S}\to\overline\S\otimes\kk_\sigma\to0,
\]
where the first $\kk_\sigma$ corresponds to $T=T_\sigma$, the penultimate entry comes from the strata $T\supseteq S$ with $\sharp T=\sharp S+1$, and the last 
entry is induced by $T=S$. As a complex, this is just the reduced cohomology complex of the standard simplex defined by $\sharp T_\sigma-\sharp S+1$ points tensored with $\overline\S$. In particular, it is exact.

\medskip

{\em Case 2: $S=T_\sigma$.} Now, $\opn{K}(F)^\bullet=\mc C_{-1}$ finally enters the stage. From~\eqref{eq:F}, we have $F(\sigma)=\overline\T_\sigma:= \S_{T_{\sigma}}$. Since $S=T_\sigma$, the complex $\mc C_\boxx$ reduces to $\overline\T_\sigma\otimes\kk_\sigma$, which is absorbed by $F(\sigma)$. 
\end{proof}

\subsubsection{Conclusion}
Since the complexes $\opn{C}^\bullet(a,b)$ were exact as direct sums of exact complexes $\opn{C}^\bullet(\S,\T)$ from the very beginning , cf.\ Lemma~\ref{lem:Subcomplex}, the cohomology of $\mc C_\boxx$ vanishes, too. In conjunction with Lemma~\ref{lem:FinalLem}, this shows exactness of the complex $\opn{K}(F)^\bullet$ from~\eqref{eq:FComp} and thus proves finally Proposition~\ref{prop:DerLimLoc}.

\section{Computing cohomology}
\label{sec:CompCohom}
In view of computing the cohomology of a given toric sheaf $\CE$, we 
build an analogue of the exact sequence~\eqref{eq:CanRes} for the 
constructible sheaf $\CF(\CE)$. This gives rise to the spectral sequence of \ssect{subsec:ExamCoho} which we apply in Subsections~\ref{subsec:EulerGlobSec},~\ref{subsec:HeightOne} and~\ref{subsec:TangBundle}.

\subsection{A complex for $\CF(\CE)$}
\label{subsec:CompForF}
Recall that the associated constructible sheaf $\CF(\CE)$ from Section~\ref{sec:CohomTS} was defined by
\[
\CF(\CE)(U)=\{e\in E\mid U\subseteq\CPvar^+(e)\}\subseteq E
\]
with stalks
\[
\CF(\CE)_u=\{e\in E\mid u\in\innt_\Delta\CPvar^+(e)\}
\]
at $u\in\Delta\subseteq M_\R$. As before, we shall freely switch between $e$ and the stratum $\S$ containing it, that is, $\CPvar^+(e)=\CPvar^+(\S)$. 

\medskip

For $\S\leq\T$ and the open embedding $j_\T\colon\innt_\Delta\CPvar^+(\T)\into\Delta$, we then consider the $\kk$-sheaves 
\[
\spectF_{\S,\T}=(j_\T)_!\,\underline{\overline\S}=\overline\S\otimes(j_\T)_!\,\underline\kk
\]
with stalks
\[
\renewcommand{\arraystretch}{1.3}
(\spectF_{\S,\T})_u= 
\begin{cases}
\overline\S &\text{if }u\in\innt_\Delta\CPvar^+(\T),\\
0 &\text{if otherwise}.
\end{cases}
\]
For $\S=\T$ we simply write
\[
\spectF_\S=\spectF_{\S,\S}.
\]
As before, this assignment is covariant in $\S$ and contravariant in $\T$, cf.\ Equation~\eqref{eq:Hom}.

\begin{remark}
\label{rem:BackPolytope}
In a sense, $\spectF_{\S,\T}$ is the simplest sheaf reflecting the 
polyhedral structure of $\CE$. Analogously to \ssect{subsubsec:cohomSheafL}, its cohomology is
\[
\gH^\bullet(\Delta, \spectF_{\S,\T}) = 
\overline\S\otimes\gH^\bullet\big(\Delta,(j_{\T})_!\underline\kk\big)= 
\overline\S\otimes\widetilde{\opn H}^{\raisebox{-3pt}{\scriptsize$\bullet\!-\!1$}}\big(\Delta\setminus\innt_\Delta\CPvar^+(\T)\big)
\]
which is exactly the cohomology of the invertible sheaf $\CO_X(\CPvar^+(\T)-\Delta)$ tensored with $\overline\S$.
\end{remark}

Recall that given two strata $\S\leq\T$ we let
$\opn{ch}_\ell(\S,\T)$ be the set of strict chains of length $\ell$ in $\Strat$ starting at $\S_0=\S$ and terminating at $\S_\ell=\T$, cf.~\eqref{eq:DefCH} on page~\pageref{eq:DefCH}. For a toric sheaf $\CE$ of rank $r$, we define a complex
\[
\spectF^\kbb(\CE)\colon\spectF^{-r}\to\ldots\to\spectF^{-1}\to\spectF^0 
\]
by
\[
\spectF^{-\ell}(\CE):=\spectF^{-\ell}:=\bigoplus_{\S\leq\T}\bigoplus_{\opn{ch}_\ell(\S,\T)}\spectF_{\S,\T},\quad 0\leq\ell\leq r.
\]
In particular, $\spectF^{-1}=\bigoplus_{\S<\T}\spectF_{\S,\T}$ and $\spectF^0=\bigoplus_{\S}\spectF_\S$. If $\spectF_{\S_0,\S_\ell}\subseteq\spectF^{-\ell}$ represents the $\ell$-chain $\S_0<\ldots<\S_\ell$, we also write $\spectF_{\S_0,\ldots,\S_\ell}$ for $\spectF_{\S_0,\S_\ell}$. The differentials $\spectF^{-\ell}\to\spectF^{-\ell+1}$, $\ell\ge1$, are induced by the direct sum of the sheaf maps
\[
\begin{tikzcd}[ampersand replacement=\&]
\spectF_{\S_0,\S_\ell}=\spectF_{\S_0,\ldots,\S_\ell}\ar[rr,
"{\left(\begin{smallmatrix}1 & -1 & \ldots & \pm
1\end{smallmatrix}\right)}"]
\&\&\displaystyle\bigoplus_{i=0}^\ell\ \spectF_{\S_0,\ldots
\widehat\S_i\ldots,\S_\ell}=
\spectF_{\S_1,\S_\ell}\oplus\spectF_{\S_0,\S_\ell}^{\oplus(\ell-1)}\oplus
\spectF_{\S_0,\S_{\ell-1}}.
\end{tikzcd}
\]

\begin{theorem} 
\label{thm:FRes}
Let $\CF(\CE)$ be the $\kk$-linear sheaf associated with the toric sheaf $\CE$
of rank $r$. Then the map
\begin{tikzcd}[ampersand replacement=\&]
\bigoplus_{\S} \spectF_\S 
\ar[rr, "{\left(\begin{smallmatrix}1 & \ldots & 1 \end{smallmatrix}\right)}"]
\&\&\CF(\CE)
\end{tikzcd}
induces a quasi-isomorphism between $\spectF^\kbb(\CE)$ and $\CF(\CE)$.
\end{theorem}

\begin{proof}
We briefly sketch the proof which goes along the lines of Lemma~\ref{lem:CanRes}. 

\medskip

Clearly, $\spectF^\kbb$ defines a complex; we let
\[
\spectF^1:=\CF(\CE)
\]
and check exactness of the augmented complex stalkwise. For every $u\in\Delta$, consider the global version of~\eqref{eq:LocVer}, namely
\[
\S(u):=\bigvee\{\S\mid u\in\innt_\Delta\CPvar^+(\S)\} 
\]
which is the maximal stratum $\S$ with $u\in\innt_\Delta\CPvar^+(\S)$. 
The non-zero stalks of the complex are indexed by chains $\S_0<\ldots<\S_\ell$ with $u\in\innt_\Delta\CPvar^+(\S_\ell)$, or equivalently, with $\S_\ell\leq\S(u)$. The complex of stalks takes the form
\[
\hspace{-15pt}0\to\displaystyle\bigoplus_{\mathclap{\S_0<\ldots<\S_{m}\leq\S(u)}}\ (\spectF_{\S_0,\ldots,\S_{m}})_u\to\ldots\to\displaystyle\bigoplus_{\S_0<\S_1\leq\S(u)}\!(\spectF_{\S_0,\S_1})_u\to\displaystyle\bigoplus_{\S\leq\S(u)}(\spectF_\S)_u\to\CF(\CE)_u\mathrlap{\to 0}
\]
where $m$ denotes the maximal length of a chain such that there is $(\spectF_{\S_0,\ldots,\S_{m}})_u\neq0$. By definition, we have
\[
(\spectF_{\S_0,\ldots,\S_\ell})_u=\overline\S_0
\]
for $\S_0<\ldots<\S_\ell\leq\S(u)$ and
\[
\CF(\CE)_u=\overline{\S(u)}. 
\]

\medskip

For $1\leq\ell\leq m$, we decompose the terms of the complex into
\[
C_\ell:=\quad\bigoplus_{\mathclap{\S_0<\ldots<\S_\ell\leq\S(u)}}\;(\spectF_{\S_0,\ldots,\S_\ell})_u=C^=_\ell \oplus C^<_\ell,
\]
where the first sum is over $\S_\ell=\S(u)$ and the second sum is over $\S_\ell<\S(u)$. The differential restricts to an isomorphism $C^=_{\ell+1}\xrightarrow{\sim} C^<_\ell$ being the direct sum of
\[
\big(\spectF_{\S_0,\ldots,\S_\ell,\S(u)}\big)_u \xrightarrow{\pm1}
\big(\spectF_{\S_0,\ldots,\S_\ell}\big)_u
\]
with $\S_\ell<\S(u)$. So starting with the chains of maximal length $m$, we may split off the direct summand $C^=_m=C^<_{m-1}$ twice and consider the homotopic complex:
\[
\begin{tikzcd}[row sep=0ex]
& C^=_{m-1} & C^=_{m-2} \\
0 \to \cancel{C^=_m} \ar[rd, "\sim"'] \ar[ru]
& \oplus \ar[r] & \oplus \ar[r] & \ldots\\
& \cancel{C^<_{m-1}} & C^<_{m-2}
\end{tikzcd}
\]
By repeating this procedure $(m-1)$-times, we are left with the terms
\[
\begin{tikzcd}
0 \to\fbox{$C^=_1 = (\spectF_{\S(u)})_u = \overline{\S(u)}$}\ar[r, "\sim"] &\CF(\CE)_u \to 0.
\end{tikzcd}
\]
As a consequence, the complex is exact.
\end{proof}

\subsection{A spectral sequence for 
$\gH^\kbb(X,\CE)_0=\gH^\kbb(\Delta,\CF(\CE))$}
\label{subsec:ExamCoho}
We let $h=h(\CE)$ be the {\em height} of the poset $\Strat_\CE$ which by 
definition is the maximal length of chains $\S_0<\ldots<\S_h$ of 
nontrivial strata in $E$. Then $h\leq r$, and by Theorem~\ref{thm:FRes},
the sheaf $\CF(\CE)$ on $\Delta$ is quasi-isomorphic to the complex
\[
\spectF^\kbb=\spectF^\kbb(\CE)=[\spectF^{-h}\to\ldots\to\spectF^0] 
\]
with $\spectF^{-\ell}=\bigoplus_{\S<\T}\bigoplus_{\opn{ch}_\ell(\S,\T)}\spectF_{\S_0,\S_\ell}$, $0\leq\ell\leq h$. If we let $\H(X,\cdot)=\R\Gamma_X$ and 
$\H(\Delta,\cdot)=\R\Gamma^\Delta$ denote the 
right-derived global section functors, this means that
\[
\H(X,\CE)_0=\H(\Delta,\CF(\CE))=\H(\Delta,\spectF^\kbb).
\]

\medskip

We shall employ the usual spectral sequence for hypercohomology. For a functor $G$ we consider
\begin{equation}
\label{eq:SpecSeqA}
E_1^{p,q}=(R^qG)(\spectF^p)\then(\R^{p+q}G)(\spectF^\kbb)
\end{equation}
with differential
\[
d_s\colon E_s^{p,q}\to E_s^{p+s,\,q-s+1}.
\]
We have isomorphisms 
$E_{s+1}^{p,q}\stackrel{\sim}{\to}\gH^{p,q}(E_s^{\kbb,\kbb})$, 
and the sequence $E_s^{p,q}$, $E_{s+1}^{p,q},\dots$ becomes 
eventually stationary with limit $E_\infty^{p,q}$. 
Then~\eqref{eq:SpecSeqA} means that $(\R^nG)(\spectF^\kbb)$ admits a descending filtration with factors $\gr^p(\R^nG)(\spectF^\kbb)=E_\infty^{p,n-p}$.

\medskip

In our concrete situation where $G=\kG^\Delta$, we have
\[
E_1^{p,q}=\gH^q(\Delta,\spectF^p) \then \H^{p+q}(\Delta,\spectF^\kbb)=
\gH^{p+q}(\Delta,\CF(\CE)),
\]
with $-r\leq-h(\CE)\leq p \leq 0\leq q\leq n=\rk N$. For $p=-\ell$ we obtain
\[
\spectF^p=\bigoplus_{\S<\T}\bigoplus_{\opn{ch}_p(\S,\T)}\spectF_{\S,\T},
\]
that is,
\[
E_1^{-\ell,q}=\bigoplus_{\S<\T}\bigoplus_{\opn{ch}_\ell(\S,\T)}\gH^q(\Delta,\spectF_{\S_0,\S_\ell}) 
\then 
\gH^{q-\ell}\big(\Delta,\,\CF(\CE)\big),
\]
see Figure~\ref{fig:SpecSeq}.
\begin{figure}[ht]
\newcommand{\scaleA}{-4.9} 
\newcommand{\scaleB}{0.7}
\newcommand{\spaceA}{\hspace*{5em}}
\begin{tikzpicture}[scale=\scaleB]
\draw[color=oliwkowy!40] (-0.3,-2.3) grid (5.3,0.3);
\foreach \x in {0,...,4} \foreach \y in {-2,-1,0} {
  \fill[thick, color=gray] (\x,\y) circle (2pt); }
\draw[thick, black]
  (6,0) node{$E_1^{0,q}$} (6,-2) node{$E_1^{-2,q}$}
  (0,0.5) node{$E_1^{-\ell,0}$} (4,0.5) node{$E_1^{-\ell,4}$};
\draw[thick, red, ->]
  (0,-0.8) -- (0,-0.2); 
\draw[thick, red, ->]
  (0,-1.8) -- (0,-1.2);
\draw[thick, red, ->]
  (1,-0.8) -- (1,-0.2);
\draw[thick, red, ->]
  (1,-1.8) -- (1,-1.2);
\draw[thick, red, ->]
  (2,-0.8) -- (2,-0.2); 
\draw[thick, red, ->]
  (2,-1.8) -- (2,-1.2);
\draw[thick, red, ->]
  (3,-0.8) -- (3,-0.2); 
\draw[thick, red, ->]
  (3,-1.8) -- (3,-1.2);
\draw[thick, red, ->]
  (4,-0.8) -- (4,-0.2); 
\draw[thick, red, ->]
  (4,-1.8) -- (4,-1.2);
\end{tikzpicture}
\spaceA
\begin{tikzpicture}[scale=\scaleB]
\draw[color=oliwkowy!40] (-0.3,-2.3) grid (5.3,0.3);
\foreach \x in {0,...,4} \foreach \y in {-2,-1,0} {
  \fill[thick, color=gray] (\x,\y) circle (2pt); }
\draw[thick, black]
  (6,0) node{$E_2^{0,q}$} (6,-2) node{$E_2^{-2,q}$}
  (0,0.5) node{$E_2^{-\ell,0}$} (4,0.5) node{$E_2^{-\ell,4}$};
\draw[thick, red,->]
  (0.9,-1.9) -- (0.1,-0.1);
\draw[thick, red,->]
  (1.9,-1.9) -- (1.1,-0.1);
\draw[thick, red,->]
  (2.9,-1.9) -- (2.1,-0.1); 
\draw[thick, red,->]
  (3.9,-1.9) -- (3.1,-0.1);
\end{tikzpicture}
\spaceA
\begin{tikzpicture}[scale=\scaleB]
\draw[color=oliwkowy!40] (-0.3,-2.3) grid (5.3,0.3);
\foreach \x in {0,...,4} \foreach \y in {-2,-1,0} {
  \fill[thick, color=gray] (\x,\y) circle (2pt); }
\draw[thick, black]
  (6,0) node{$E_\infty^{0,q}$} (6,-2) node{$E_\infty^{-2,q}$}
  (0,0.5) node{$E_\infty^{-\ell,0}$} (4,0.5) node{$E_\infty^{-\ell,4}$};
\draw[thick, blue]
  (0.1,-0.1) -- (2.2,-2.2)
  (0.8,0.2) -- (3.2,-2.2)
  (1.8,0.2) -- (4.2,-2.2)
  (2.8,0.2) -- (4.2,-1.2)
  (3.9,0.1) -- (4.2,-0.2);
\draw[thick, blue]
  (2.6,-2.6) node{$\gH^0$}
  (3.6,-2.6) node{$\gH^1$}
  (4.6,-2.6) node{$\gH^2$}
  (4.6,-1.5) node{$\gH^3$}
  (4.6,-0.5) node{$\gH^4$};
\draw[thick, green]
  (0.6,-2.6) node{$0$}
  (1.6,-2.6) node{$0$}
  (0.1,-1.1) -- (1.2,-2.2)
  (0.1,-2.1) -- (0.2,-2.2);
\end{tikzpicture}
\caption{The spectral sequence for $h=2$}
\label{fig:SpecSeq}
\end{figure}
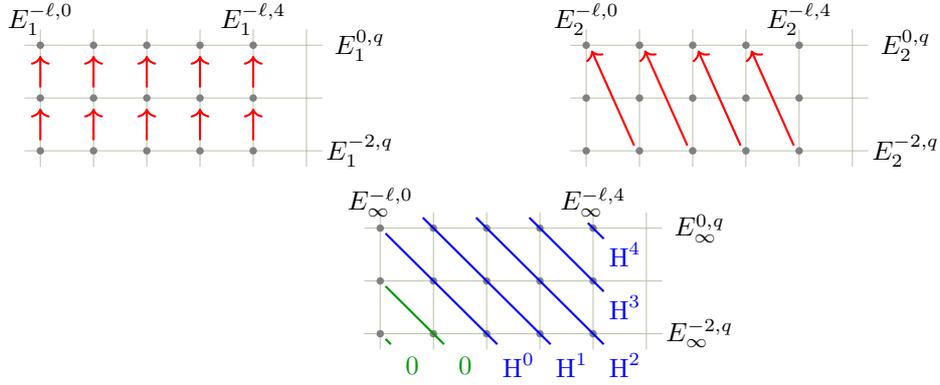

\medskip

Next, we define the closed subsets
\begin{equation}
\label{eq:PDef}
P(\S):=\Delta\setminus\innt_{\Delta}\CPvar^+(\S)
\end{equation}
of $\Delta$. Let
\[
j_\S\colon\innt_\Delta\CPvar^+(\S)\hookrightarrow\Delta\quad\text{and}\quad
\iota_\S\colon P(\S)\hookrightarrow\Delta
\]
be the open and closed embedding of $\innt_\Delta\CPvar^+(\S)$ and $P(\S)$, respectively, with associated Gysin
sequence
\[
\begin{tikzcd}
0\ar[r]&(j_\S)_!\underline\kk\ar[r]&\underline\kk\ar[r]&
(i_\S)_*\underline\kk\ar[r]&0.
\end{tikzcd}
\]
Applying $\H(\Delta,\cdot)$ yields the distinguished triangle
\[
\begin{tikzcd}
\H\big(\Delta,(j_\S)_!\underline\kk\big)\ar[r]&\H(\Delta,\kk)
\ar[r]&\fbox{$\H\big(\Delta,(i_\S)_*\underline\kk\big)
= \H\big(\hspace{-1pt}P(\S),\kk\big)$}
\end{tikzcd}
\]
in the bounded derived category. The mapping cone $\big[\H(\Delta,\kk)\to \H\big(\hspace{-1pt}P(\S),\kk\big)\big]$ is isomorphic to the reduced cohomology which we write $\widetilde{\H}\big(P(\S),\kk\big)$. We thus obtain
\begin{equation}
\label{eq:RedCohom}
\begin{tikzcd}
\H\big(\Delta,(j_\S)_!\underline\kk\,\big)
\ar[r,"\sim"]&
\Big[\H(\Delta,\kk)\to \H\big(\hspace{-1pt}P(\S),\kk\big)\Big][-1]
\ar[r,"\sim"]&
\widetilde{\H}\big(P(\S),\kk\big)[-1]
\end{tikzcd}
\end{equation}
which ultimately leads to the following

\begin{theorem}
\label{thm:SpecSeq}
There exists a spectral sequence
\begin{equation}
\label{eq:SpecSeq}
E_1^{-\ell,q}=\bigoplus_{\S<\T}\bigoplus_{\opn{ch}_\ell(\S,\T)}\overline\S\otimes\widetilde{\opn H}^{\raisebox{-3pt}{\scriptsize$q\!-\!1$}}\big(P(\T),\,\kk\big)\then\gH^{q-\ell}\big(\Delta,\,\CF(\CE)\big)=\gH^{q-\ell}(X,\CE)_0
\end{equation}
with $0\leq\ell\leq h(\Strat)$, $0\leq q\leq n=\dim X$.
\end{theorem}

\begin{remark}
From $q<\ell$ we deduce $\widetilde{\opn H}^{\raisebox{-3pt}{\scriptsize$q\!-\!\ell$}}\big(X,\CE\big)_0=0$. Consequently, as displayed in Figure~\ref{fig:SpecSeq}, all $E_\infty^{-\ell,q}$ must vanish, which also implies constraints on the $E_1$-terms. For instance, in the case of height $h(\CE)=1$ 
we have $E_\infty=E_2$ and thus $E_2^{-1,0}=0$, 
that is, $d_1:E_1^{-1,0}\to E_1^{0,0}$ is injective.
Using the terminology of Subsection~\ref{subsec:HeightOne}, this is also a consequence of $P(\S)\cup P(\T)=P(\eta)$ for $\S\not=\T$.
\end{remark}

\begin{remark}
\label{rem:backToLB}
The terms $\widetilde{\opn H}^{\raisebox{-3pt}{\scriptsize$q\!-\!1$}}\big(P(\T),\kk\big)$ of the spectral sequence~\eqref{eq:SpecSeq} are just the cohomology groups of line bundles $\gH^q\!\big(X,\CD_\CE(\T)\big)_0$. Since for $\S<\T$, the sheaves $\overline\S\otimes\CO_X\big(\CD_\CE(\T)\big)$ define a system of subsheaves of $\CE$, we obtain~\eqref{eq:SpecSeq} equally well from the coherent sheaves on $X$ from Corollary~\ref{coro:CanSeq} instead of the constructible sheaves on $\Delta$. 
\end{remark}

\subsection{Euler characteristic and global sections}
\label{subsec:EulerGlobSec}
Next we turn to the relationship between $\gH^0(X,\CE)$ and the Euler characteristic of $\CE$, as equality is a necessary condition for $\CE$ to be acyclic. In particular, $\gH^0(X,\CE)=\chi(X,\CE)=0$ is necessary for $\CE$ to be immaculate. 

\medskip

As Figure~\ref{fig:SpecSeq} reveals, the spectral sequence is not well-suited for the computation of $\gH^0(X,\CE)$. However, from Remark~\ref{rem:VS} we already know how to recover the global sections from the Weil decoration. Indeed, 
\begin{equation}
\label{eq:GlobSecE}
\gH^0(X,\CE)_0=\overline\S_{\mr{glob}},
\end{equation}
where $\S_{\mr{glob}}$ is the maximal stratum such that $\CPvar^+(\S_{\mr{glob}})\supseteq\Delta$. Note that for general degree $m\in M$, the corresponding stratum $\S_{\mr{glob},m}$ is characterised by $\CPvar^+(\S_{\mr{glob},m})\supseteq\Delta+m$.

\begin{example}
We revisit the tangent bundle $\CT_{\PP^2}$, cf.\ \ssect{subsubsec:cohomSheafTang}. We twisted by the standard simplex $\Delta=\Delta_1$ to obtain the ample Weil decoration
\[
\CPvar^+(\rho_0)=2\Delta,\quad
\CPvar^+(\rho_1)=2\Delta-[1,0],\quad
\CPvar^+(\rho_2)=2\Delta-[0,1],\quad
\CPvar^+(\eta)=\Delta
\]
of $\CT_{\PP^2}(\Delta)$. It follows that $\gH^0(X,\CT_{\PP^2})_0=N_\kk=\overline{\eta}$. We have further contributions from $\kk\cdot\rho_0=\overline{\S(\rho_0)}$ in degrees $[1,0]$ and $[0,1]$, from $\kk\cdot\rho_1=\overline{\S(\rho_1)}$ in degrees $[-1,0]$ and $[-1,1]$ and from $\kk\cdot\rho_2=\overline{\S(\rho_2)}$ in degrees $[0,-1]$ and $[1,-1]$. 
\end{example}

On the other hand, the Euler characteristic in degree $0$ is given by
\begin{align}
\chi(X,\CE)_0&=\sum_n(-1)^n\dim \gH^{n}\big(X,\CE\big)_0=\sum_{\ell,q} (-1)^{q-\ell} \dim E_1^{-\ell,q}\nonumber\\
&=\sum_{\ell\geq0}(-1)^{\ell+1}\!\sum_{\S<\T}\sum_{\opn{ch}_\ell(\S,\T)}\dim(\S)\cdot\widetilde\chi(P(\T))
\label{eq:FirstEC},
\end{align}
where 
we used the reduced Euler characteristic
\[
\widetilde\chi(P(\T))=\sum_q(-1)^q\dim
\widetilde{\opn H}^{\raisebox{-3pt}{\scriptsize$q$}}\big(P(\T),\kk\big)=\chi(P(\T))-1. 
\]

To simplify expression~\eqref{eq:FirstEC} further, we recall that for a finite poset $P$ the {\em associated M\"obius function} $\mu_P\colon P\times P\to\Z$ is defined recursively by 

\begin{align*}
\mu_P(a,a)&=1\quad\text{for all }a\in P,\\
\mu_P(a,b)&=-\sum_{a\leq p<b}\mu_P(a,p)\text{ for all }a<b\text{ in }P. 
\end{align*}
We may regard $\mu_P$ as a matrix in $\Z^{\sharp P\times\sharp P}$ with entries
\[
\renewcommand{\arraystretch}{1.3}
(\mu_P)_{a,b}:=\begin{cases}\mu_P(a,b)&\text{if }a\leq b,\\
0 &\text{if otherwise}.\end{cases}
\]
The celebrated M\"obius inversion formula states that $\mu_P$ is the inverse of the poset's incidence matrix $\Inc_P\in\Mat(P,P;\Z)$, namely
\[
\renewcommand{\arraystretch}{1.3}
(\Inc_P)_{a,b}:=\begin{cases}1 &\text{if }a\leq b,\\
0 &\text{if otherwise;}\end{cases}
\]
see~\cite[Sec.\ 3.7]{stanley} for this and further results.

\medskip

Next we apply this to $\mu_\CE$, the M\"obius function associated with the stratification $\Strat$ of the toric sheaf $\CE$.

\begin{lemma}
\label{lem:MobStrat}
Let $\S$ and $\T\in\Strat$ be two strata with $\S\leq\T$. Then
\[
\mu_\CE(\S,\T)=\sum_{\ell\geq0}\sum_{\opn{ch}_\ell(\S,\T)}(-1)^\ell.
\]
\end{lemma}

\begin{proof}
Set temporarily $f(\S,\T):=\sum_{\ell\geq0}\sum_{\opn{ch}_\ell(\S,\T)}(-1)^\ell$. If $\S=\T$, then $\opn{ch}_\ell(\S,\S)=\varnothing$ unless $\ell=0$ where it equals $\{\S\}$. Hence $f(\S,\S)=1=\mu_\CE(\S,\S)$.

\medskip

Next assume $\S<\T$. By~\eqref{eq:ExamCh}, $f(\S,\T)=\sum_{\ell\geq1}\sum_{\opn{ch}_\ell(\S,\T)}(-1)^\ell$, and
\[
\opn{ch}_\ell(\S,\T)=\bigsqcup_{\S<\T'<\T}\opn{ch}_{\ell-1}(\S,\T').
\]
for $\ell\geq2$. It follows that
\begin{align*}
f(\S,\T)&=\sum_{\ell\geq1}\sum_{\opn{ch}_\ell(\S,\T)}(-1)^\ell=\sum_{\ell\geq1}\sum_{\S\leq\T'<\T}\sum_{\opn{ch}_{\ell-1}(\S,\T')}(-1)^\ell\\
&=-\sum_{\S\leq\T'<\T}\sum_{\ell\geq0}\sum_{\opn{ch}_\ell(\S,\T')}(-1)^\ell=-\sum_{\S\leq\T'<\T}f(\S,\T'),
\end{align*}
which is precisely the recursion formula for the M\"obius function. Hence $f(\S,\T)=\mu_\CE(\S,\T)$ for all $\S\leq\T$. 
\end{proof}

From~\eqref{eq:FirstEC} and Lemma~\ref{lem:MobStrat} we immediately draw the

\begin{corollary}
\label{coro:EulerChar}
The Euler characteristic in degree $0$ of a toric sheaf $\CE$ is given by
\[
\chi(X,\CE)_0=-\sum_{\S\leq\T}\dim(\S)\cdot\mu_\CE(\S,\T)\cdot\widetilde\chi(P(\T)).
\]
\end{corollary}

We note in passing that the M\"obius function also appears in the Euler characteristic of the Grassmannian with finitely many hyperplane sections removed, see~\cite{eulerchar}.

\begin{example}
\label{exam:EulerCharLB}
Consider the line bundle $\CL=\CO_X(D)=\CO(\nablap-\nablam)$ with $\nabla_\pm$ nef. The only strata are $0$ and $\eta$; since $\dim(0)=0$ and $\dim(\eta)=1$ we find indeed
\[
\chi(X,\mc L)_0=\widetilde\chi(P(\eta)) 
\]
where $P(\eta) = \nablam \setminus \innt_{\nablam}(\nablap)$ for the generic stratum $\eta$. In particular, if $D$ is already nef so that we can choose $\nablam=\{0\}$, then
\[
\renewcommand{\arraystretch}{1.3}
\chi\big(X,\CO_X(D)\big)_0=-\widetilde\chi(P(\eta))=\dim\widetilde{\opn H}^{\raisebox{-3pt}{\scriptsize$-1$}}(P(\eta))=
\begin{cases}
1 & \text{if } 0 \in D,\\
0 & \text{if otherwise}.
\end{cases}
\]
\end{example}

\begin{remark}
\label{rem:EquivEC}
Regarding the M\"obius form $\mu_\CE=\Inc_\Strat^{-1}$ as a bilinear form on $\Z^\Strat$ and the data $\dim=\big(\dim(\S)\big)$ and $\widetilde\chi=\big(\widetilde\chi(P(\T))\big)$ as vectors in $\Z^\Strat$, respectively, we can write this more succinctly as  
\[
\chi(X,\CE)_0=-\dim^\top\hspace{-3pt}\cdot\,\mu_\CE\cdot\widetilde\chi=-\dim^\top\hspace{-3pt}\cdot\,\Inc_\CE^{-1}\cdot\widetilde\chi.
\] 
In this form, $\chi(X,\CE)_0$ can be directly read off from the Weil decoration. More generally, we can consider the {\em equivariant Euler characteristic} 
\[
\chi^\TX(X,\CE):=\sum_{m\in M}\chi(X,\CE)_m\cdot x^m\in\kk[M];
\]
only finitely many terms are nontrivial for $X$ is complete. We then recover the usual Euler characteristic by $\chi(X,\CE)=\chi^\TX(X,\CE)(1)$. From Example~\ref{exam:EulerCharLB} we conclude 
\begin{equation}
\label{eq:ECForm}
\chi^\TX(X,\CE)=\sum_{\S\leq\T}\dim(\S)\cdot\mu_\CE(\S,\T)\cdot\chi^\TX\big(X,\CO_X(\CD_\CE(\T))\big)=\dim^\top\cdot\Inc_\CE^{-1}\cdot\chi^\TX
\end{equation}
where $\chi^\TX$ denotes the vector $\big(\chi^\TX(\CD_\CE(\T))\big)$.
\end{remark}

\subsection{Posets of height one}
\label{subsec:HeightOne}
Theorems~\ref{thm:CohomSheafE} and~\ref{thm:FRes} and Equation~\ref{eq:RedCohom} give 
\begin{align*}
&\gH(X,\CE)_0=\\
&\Big[\ldots\to\bigoplus_{\S<\T}\bigoplus_{\opn{ch}_\ell(\S,\T)}
\!\big(\overline\S\otimes\widetilde\H(P(\T),\kk)\big)\to
\bigoplus_{\S<\T}\bigoplus_{\opn{ch}_{\ell-1}(\S,\T)}\!\!\big(\overline\S\otimes\widetilde\H(P(\T),\kk)\big)\to\ldots\Big][-1]
\end{align*}
where the big complex is induced by $\spectF^\kbb=\spectF^\kbb(\CE)$, that is, the displayed terms sit at positions $-\ell$ and $-(\ell-1)=1-\ell$.

\medskip

We illustrate this formula for the case of toric sheaves $\CE$ with stratifications $\Strat_\CE$ of height one, that is, their Hasse diagram of strata is given by
\[
\begin{tikzcd}
& \eta  & & 
\\
\S_1 \ar[ru, no head] & \ldots & \S_N \ar[lu, no head] & 
\end{tikzcd}
.
\]
As there are no chains of length greater than two, the exact sequence 
of Theorem~\ref{thm:FRes} boils down to the short exact sequence 
\begin{equation}
\label{eq:Heightseq}
\begin{tikzcd}
0\ar[r]&\bigoplus_{\indStrat=1}^N\spectF_{\S_\indStrat,\eta}\ar[r]&
\spectF_\eta\oplus
\bigoplus_{\indStrat=1}^N\spectF_{\S_\indStrat}\ar[r]&\CF(\CE)\ar[r]&0.
\end{tikzcd}
\end{equation}
Applying the functor $\H(\Delta,\cdot)=\R\Gamma(\Delta,\cdot)$, we get
\begin{align*}
&\H(X,\CE)_0=\H(\Delta,\CF(\CE))=
\Big[\bigoplus_\indStrat\H(\Delta,\spectF_{\S_\indStrat,\eta})\to
\H(\Delta,\spectF_\eta)\oplus\bigoplus_\indStrat\H(\Delta,\spectF_{\S_\indStrat})\Big]\\
&=\Big[\bigoplus_\indStrat\big(\overline\S_\indStrat\otimes\widetilde\H(P(\eta),\kk)\big)\to 
\big(E\otimes\widetilde\H(P(\eta),\kk)\big)\oplus\bigoplus_\indStrat\big(\overline\S_\indStrat\otimes\widetilde\H(P(\S_\indStrat),\kk)\big)
\Big][-1],
\end{align*}
where $[\cdot]$ denotes the mapping cone
and $P(\S)=\Delta\setminus\innt_{\Delta}\CPvar^+(\S)$, cf.~\eqref{eq:PDef}.

\medskip

Alternatively, this is also a consequence of our spectral sequence from~\ssect{subsec:EulerGlobSec}. The sequence consists of two rows for $\ell=0$ and $1$, and we obtain $E_2^{-1,q}=E_\infty^{-1,q}$ and $E_2^{0,q}=E_\infty^{0,q}$ as kernel and cokernel of $d_1\colon E_1^{-1,q}\to E_1^{0,q}$, respectively. In particular, the short exact sequences
\[
0\to E_\infty^{0,q}\to\gH^{q}\big(X,\CE\big)_0\to E_\infty^{-1,q+1}\to0
\]
lead to the long exact sequence
\begin{equation}
\label{eq:SpectHeightOne}
\ldots\to E_1^{-1,q}\to E_1^{0,q}\to\gH^{q}\big(X,\CE\big)_0\to
E_1^{-1,q+1}\to E_1^{0,q+1}\to\ldots
\end{equation}

\medskip

Next, we assume in addition that the natural linear map
\begin{equation}
\label{eq:ALin}
\begin{tikzcd}
A\colon\bigoplus_{\indStrat=1}^N\overline{\S}_\indStrat
\ar[r,twoheadrightarrow]&E 
\end{tikzcd}
\end{equation}
is surjective. This applies, for instance, to arbitrary toric sheaves of rank $2$ with $N\geq2$ (the bundles with $N\leq1$ are split), or the universal extensions discussed in Section~\ref{sec:ExtNLB}. 

\medskip

In~\eqref{eq:Heightseq}, each $\spectF_{\S_\indStrat,\eta}$ maps naturally
into $\spectF_{\S_\indStrat}$ and, with opposite sign, into $\spectF_\eta$.
Hence we may pass to the subcomplex induced by~\eqref{eq:ALin}
and finally obtain
\begin{equation}
\label{eq:finalHeightOne}
\H(X,\CE)_0=\Big[\ker A\otimes\widetilde{\H}(P(\eta),\kk\big)
\to\bigoplus\limits_{\indStrat=1}^N\overline\S_\indStrat\otimes\widetilde{\H}\big(P(\S_k),\kk\big)\Big][-1]
\end{equation}
in the bounded derived category. Note that by the $[-1]$-shift, the left complex sits at position $0$, while the right one is at position $1$.

\subsection{The tangent bundle continued}
\label{subsec:TangBundle}
We continue the example from \ssect{subsubsec:cohomSheafTang}
and turn to the case $\CT_{\PP^2}(-\ell)=\CT_{\PP^2}(-\ell D_{\rho_0})$ 
with $\ell\geq1$. Here, the Weil decoration consists of the following virtual polytopes
\[
\CPvar(\rho_0)=(-\ell+1)\cdot\Delta_1,\quad
\CPvar(\rho_1)=(-\ell+1)\cdot\Delta_1-[1,0],\quad
\CPvar(\rho_2)=(-\ell+1)\cdot\Delta_1-[0,1].
\]
Since 
\[
\CPvar(\eta)=\CPvar(\rho_0)\cap\CPvar(\rho_1)\cap\CPvar(\rho_2)=
-\ell\cdot\Delta_1,
\]
we have to twist with $\Delta=(\ell+1)\cdot\Delta_1$ in order to get an amply decorated sheaf $\CT_{\PP^2}(-\ell)^+$. In particular, the materialised Weil decoration $\CD^+$ is the same as for $\ell=0$, namely
\[
\CPvar^+(\rho_0)=2\Delta_1,\quad
\CPvar^+(\rho_1)=2\Delta_1-[1,0],\quad
\CPvar^+(\rho_2)=2\Delta_1-[0,1],
\]
see~\eqref{eq:WDzero} on page~\pageref{eq:WDzero}. However, the domain of definition of the associated $\kk$-linear sheaves $\CF_m$ is now $\Delta=(\ell+1)\Delta_1+m$.

\medskip

Consider first the case $\ell=1$. For any shift $\Delta=2\Delta_1+m$, the higher reduced cohomology $\widetilde{\opn H}^{\raisebox{-3pt}{\scriptsize$\geq\!1$}}(P(\S),\kk)$ of the associated polyhedral sets $P(\rho_\indStrat)=
\Delta\setminus\innt_\Delta\CD^+(\rho_\indStrat)$ and $P(\eta)=\bigcup_\indStrat P(\rho_\indStrat)$ vanishes. Thus, the long exact sequence~\eqref{eq:SpectHeightOne} becomes
\[
0 \to E_1^{-1,0}\to E_1^{0,0} \to \gH^{0}\!\big(\PP^2,\CT(-1)\big)_m
\to E_1^{-1,1}\to E_1^{0,1} \to \gH^{1}\!\big(\PP^2,\CT(-1)\big)_m
\to 0.
\]
Since all the $P(\S)$ are connected, we also obtain
\[
\renewcommand{\arraystretch}{1.3}
\widetilde{\opn H}^{\raisebox{-3pt}{\scriptsize$-1$}}(P(\S),\kk) = \begin{cases}
\kk & \mbox{if } P(\S)=\emptyset\\
0 & \mbox{if otherwise}
\end{cases}
.
\]
For each $\indStrat=0$, $1$, $2$, there exists a unique $m_\indStrat\in M=\Z^2$
such that 
$P(\rho_\indStrat)=\emptyset$, i.e., $\Delta_2+m_\indStrat\subseteq\CD^+(\rho_\indStrat)$, namely
\[
m_0=[0,0],\quad m_1=[-1,0],\quad\text{and}\quad m_2=[0,-1].
\]
Here, the previous exact sequence simply yields
\[
\kk\cdot\rho_\indStrat\stackrel{\sim}{\to}\gH^{0}\!\big(\PP^2,\CT_{\PP^2}(-1)\big)_{m_\indStrat}.
\]
In the remaining $m\in M$, all the polyhedral sets $P(\S)$ are non-empty. Therefore, all the $E_1$-terms vanish, and so does the cohomology of $\CT_{\PP^2}(-1)$ in these degrees.

\medskip

Finally, we assume $\ell\geq 2$. 
The most interesting case $\ell=3$ is illustrated in 
Figure~\ref{fig:TangP2Twist}. 
It leads to a non-trivial $\gH^1$ 
and thus to a negative Euler characteristic, cf.~\eqref{eq:EulerChar}.
\begin{figure}[ht]
\newcommand{\scaleA}{0.3} 
\newcommand{\scaleB}{0.9}
\newcommand{\spaceA}{\hspace*{5em}}
\begin{tikzpicture}[scale=\scaleB]
\draw[color=oliwkowy!40] (-1.3,-1.3) grid (3.3,3.3);
\draw[thin, black] 
  (0,0) -- (2,0) -- (0,2) -- (0,0);
\fill[pattern color=green!30, pattern=north east lines]
  (0,0) -- (2,0) -- (0,2) -- (0,0);
\draw[thin, black] 
  (1,0) -- (-1,0) -- (-1,2) -- (1,0);
\fill[pattern color=blue!30, pattern=horizontal lines]
  (1,0) -- (-1,0) -- (-1,2) -- (1,0);
\draw[thin, black]
  (0,1) -- (2,-1) -- (0,-1) -- (0,1);
\fill[pattern color=red!30, pattern=north west lines] 
  (0,1) -- (2,-1) -- (0,-1) -- (0,1);
\draw[thick, color=green]
  (1.5,1.5) node{$\CPvar^+(\rho_0)$};
\draw[thick, color=blue]
  (-2.0,1.0) node{$\CPvar^+(\rho_1)$};
\draw[thick, color=red] 
  (0.7,-0.5) node{$\CPvar^+(\rho_2)$};
\draw[line width=0.75mm, black]
  (-1,-1) -- (3,-1) -- (-1,3) -- (-1,-1);
\draw[thick, color=black] 
  (1.1,2.7) node{$4\Delta_1-[1,1]$};
\fill[thick, red] 
  (0,0) circle (3pt);
\end{tikzpicture}
\caption{The Weil decoration of the twisted tangent bundle $\CT_{\PP^2}(-3)=\Omega_{\PP^2}$}
\label{fig:TangP2Twist}
\end{figure}
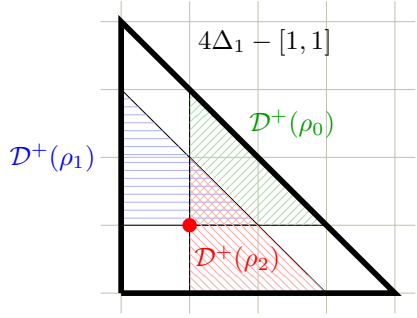
For all $\ell$
the polyhedral sets $P(\rho_\indStrat)$ and $P(\eta)=\bigcup_{\indStrat=0}^2P(\rho_\indStrat)$ are connected and non-empty. As a result, only $\widetilde{\opn H}^{\raisebox{-3pt}{\scriptsize$1$}}$ might be
non-trivial. Moreover, for any stratum $\S$ we have
\[
\widetilde{\opn H}^{\raisebox{-3pt}{\scriptsize$1$}}(P(\S),\kk)=\kk \iff 
\widetilde{\opn H}^{\raisebox{-3pt}{\scriptsize$1$}}(P(\S),\kk)\neq 0 \iff
\CD^+(\S)\subseteq\innt\Delta
\]
with $\Delta=\Delta_{\ell+1}+m$ and $\innt\Delta$ the actual interior of $\Delta$. This yields the exact sequence
\begin{equation}
\label{eq:FinalTang}
0\to \gH^{1}\!\big(\PP^2,\CT_{\PP^2}(-\ell)\big)_m \to
E_1^{-1,2} \to E_1^{0,2} \to 
\gH^{2}\!\big(\PP^2,\CT_{\PP^2}(-\ell)\big)_m \to 0
\end{equation}
with 
\[
E_1^{-1,2} = \bigoplus_{\indStrat=0}^2\big(\kk\cdot\rho_\indStrat\otimes 
\widetilde{\opn H}^{\raisebox{-3pt}{\scriptsize$1$}}(P(\eta),\kk)\big)
\]
and
\[
E_1^{0,2} = \big(N_\kk\otimes\widetilde{\opn H}^{\raisebox{-3pt}{\scriptsize$1$}}(P(\eta),\kk)\big) \;\oplus\;
\bigoplus_{\indStrat=0}^2\big(\kk\cdot\rho_\indStrat\otimes\widetilde{\opn H}^{\raisebox{-3pt}{\scriptsize$1$}}(P(\rho_\indStrat),\kk)\big).
\]
If any of the polytopes $\CD^+(\rho_\indStrat)$ is contained in $\innt\Delta=\innt\Delta_{\ell+1}+m$, then so is $\CD^+(\eta)$. Let 
\[
\gH^k_m=\gH^k\!\big(\PP^2,\CT_{\PP^2}(-\ell)\big)_m,\quad k=1,\,2.
\]
We now distinguish the following cases.

\medskip

{\em Case 1: $\CD^+(\rho_\indStrat)\subseteq\innt\Delta$ for $\indStrat=0$, $1$, $2$.}
Then the exact sequence \eqref{eq:FinalTang} is
\[
0\to\gH^{1}_m\to\big(\oplus_{\indStrat=0}^2\kk\cdot\rho_\indStrat\big) \to N_\kk\oplus\big(\oplus_{\indStrat=0}^2\kk\cdot\rho_\indStrat\big)\to\gH^{2}_m\to 0,
\]
that is, $\gH^{1}_m=0$ and $\gH^{2}_m=N_\kk$.

\medskip

{\em Case 2: $\CD^+(\rho_\indStrat)\subseteq\innt\Delta$ for precisely two indices, say $\indStrat=0$ and $1$.} 
The central map $E_1^{-1,2}\to E_1^{0,2}$ becomes
\[
\big(\oplus_{\indStrat=0}^2\kk\cdot\rho_\indStrat\big)
\to N_\kk \oplus \big(\kk\cdot\rho_0 \oplus \kk\cdot\rho_1\big).
\] 
As this is injective, its kernel $\gH^{1}_m$ must vanish and its cokernel $\gH^{2}_m\cong\kk$ is one-dimensional.

\medskip

{\em Case 3: $\CD^+(\rho_\indStrat)\subseteq\innt\Delta$ happens exactly 
once, say for $\indStrat=0$.} 
Here, the central map $\oplus_{\indStrat=0}^2\kk\cdot\rho_\indStrat\to N_\kk \oplus\kk\cdot\rho_0$ is an isomorphism and yields $\gH^1_m=\gH^2_m=0$. 

\medskip

{\em Case 4. $\CD^+(\eta)\not\subseteq\innt\Delta$.}
The polyhedral sets $\CD^+(\rho_\indStrat)$ cannot be contained in
$\innt\Delta$. Hence the whole complex vanishes and we obtain again $\gH^1_m=\gH^2_m=0$.

\medskip

{\em Case 5. $\CD^+(\eta)\subseteq\innt\Delta$, but $\CD^+(\rho_\indStrat)\not\subseteq\innt\Delta$ for $\indStrat=0$, $1$, $2$.}
This is the most interesting case. It happens only for $\ell=3$ and
$m=[-1,-1]$, see Figure~\ref{fig:TangP2Twist} for illustration. The central map is the canonical surjection $A\colon\oplus_{\indStrat=0}^2\kk\cdot\rho_\indStrat\twoheadrightarrow N_\kk$. Hence $\gH^{1}_m=\ker A=\kk\cdot(1,1,1)$ and $\gH^2_m$ vanishes.

\medskip

Finally, we compute the equivariant Euler characteristic via the Formula~\eqref{eq:ECForm} of Remark~\ref{rem:EquivEC}. Since $\dim^\top\cdot\Inc^{-1}_{\Tang_{\PP^2}(-\ell D_0)}=[0,1,1,1,-1]$ we find
\begin{align*}
\chi^\TX(\Tang_{\PP^2}(-\ell D_0))&=\sum_{i=0}^2\chi^\TX(\CO_{\PP^2}(\CD(\rho_i))-\chi^\TX(\CO_{\PP^2}(\CD(\eta))\\
&=3\chi^\TX(\CO_{\P^2}(-\ell+1))-\chi^\TX(\CO_{\P^2}(-\ell)),
\end{align*}
which is precisely the result we also obtain by using the (twisted) dual Euler sequence from~\eqref{eq:DES} on page~\pageref{eq:DES}. In particular, it follows that

\section{Final remarks}
\label{sec:FinRem}
In~\eqref{eq:GlobSecE} on page~\pageref{eq:GlobSecE} we noticed that
\[
\gH^0(X,\CE)_m=\overline\S_{\mr{glob},m},
\]
where $\S_{\mr{glob},m}$ is the maximal stratum with 
$\CPvar^+(\S_{\mr{glob},m})\supseteq\Delta+m$ for an ample Weil 
decoration $\CD^+$. 
This result continues to hold if $\CE^+$ is merely positively decorated. 
In this case, we could replace the positive Weil decoration $\CD^+=\CD+\Delta$ by $\CD^{++}:=\CD^++\Delta=\CD+2\Delta$ and use
\[
\CPvar^+(\S)\supseteq\Delta+m\iff\CPvar^+(\S)+\Delta\supseteq2\Delta+m
\]
for all strata $\S\in\Strat$.

\medskip

Of course, this cannot work for Theorem~\ref{thm:CohomSheafE}, cf.\ Remark~\ref{rem:AmpleNec} or Figure \ref{fig:PosAmpWDeco} below, where the left hand side represents the positive Weil decoration of the line bundle $\CE=\CO_X(\CD^+(1)-\Delta)$.
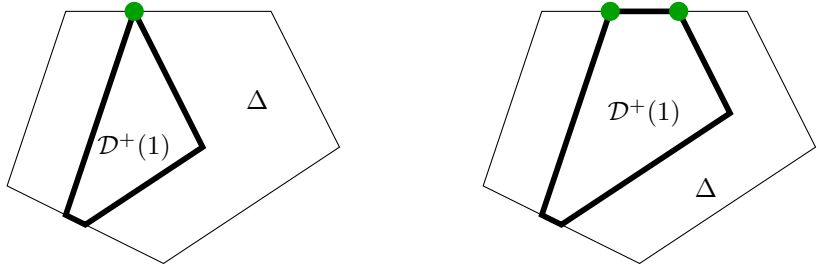
\begin{figure}[ht]
\newcommand{\scaleA}{0.3}
\newcommand{\scaleB}{0.9}
\newcommand{\spaceA}{\hspace*{5em}}
\begin{tikzpicture}[scale=\scaleB]
\draw[thin, black]
  (1/7,3/7) -- (17/7, -5/7) -- (5,1) -- (4,3) -- (1,3) -- cycle;
\draw[thick, color=black]
  (2.0,1.0) node{$\CPvar^+(1)$} (3.8,1.7) node{$\Delta$};
\draw[line width=0.75mm, black]
  (2,3) -- (1,0)-- (9/7,-1/7) -- (3,1) -- cycle;
\fill[thick, darkgreen]
  (2,3) circle (4pt);
\end{tikzpicture}
\spaceA
\begin{tikzpicture}[scale=\scaleB]
\draw[thin, black]
  (1/7,3/7) -- (17/7, -5/7) -- (5,1) -- (4,3) -- (1,3) -- cycle;
\draw[thick, color=black]
  (2.5,1.5) node{$\CPvar^+(1)$} (3.4,0.4) node{$\Delta$};
\draw[line width=0.75mm, black]
  (3,3) -- (2,3) -- (1,0)-- (9/7,-1/7) -- (15/4,3/2) -- cycle;
\fill[thick, darkgreen]
  (2,3) circle (4pt) (3,3) circle (4pt);
\end{tikzpicture}
\caption{Positive and ample Weil decorations}
\label{fig:PosAmpWDeco}
\end{figure}
Here, $\gH^1(\CE)=\widetilde{\opn H}^{\raisebox{-3pt}{\scriptsize$0$}}\big(\Delta\setminus\CD^+(1)\big)=\kk$ for the difference of the two polytopes has two components. However, $\widetilde{\opn H}^{\raisebox{-3pt}{\scriptsize$0$}}\big(\Delta\setminus\innt_\Delta\CD^+(1)\big)=0$ as the green point of the upper tip on the left hand side does not belong to $\innt_\Delta \CD^+(1)$, that is, $\Delta\setminus\innt_\Delta\CD^+(1)$ is connected.

\medskip

Instead of looking
at the closed set $\Delta\setminus\innt_\Delta\CD^+(1)$, we could equally well consider the difference $\innt\Delta\setminus\innt\CD^+(1)$ by taking the 
honest interior of $\Delta$ and $\CD^+(1)$ with respect to $M_\R$. In the example of the left hand side of Figure~\ref{fig:PosAmpWDeco}, this causes the difference to be disconnected, thus leading to the correct result.

\medskip

In general, we can ask whether or not Theorem~\ref{thm:CohomSheafE} continues to hold if we restrict the sheaf $\CF$ to $\innt\Delta$. Technically, this offers the advantage to define $\CF$ once and for all over the whole of $M_\R$
before restricting it to $\innt\Delta+m$, $m\in M$, in order to compute $\gH^\ell(X,\CE)_m$, the cohomology in degree $m$.

\begin{proposition}
\label{prop:CohomSheafEInt}
Let $X$ be a smooth projective toric variety. If $\CE^+=\CE(\Delta)$ is amply
decorated with $\Delta\in\Pol^+(\Sigma)$, then
\[
\gH^\ell(X,\CE)_0\cong\gH^\ell\hspace{-0.5ex}\big(\innt\Delta,\CF(\CE)\big).
\]
\end{proposition}

\begin{proof}
This proposition is in fact a corollary to Theorem~\ref{thm:CohomSheafE}. Indeed, we can compare both statements by considering the open embedding $j\colon\innt\Delta\hookrightarrow\Delta$ together with the map $\CF\to (\R j_*)j^{-1}\CF$ induced by adjunction. We obtain maps
\[
\H(\Delta,\CF)\to \H(\Delta,(\R
j_*)j^{-1}\CF)=\H(\innt\Delta,\CF|_{\innt\Delta})
\] 
as well as between the respective $E_1^{-\ell,q}$, see \ssect{subsec:ExamCoho}. Therefore, we may assume without loss of generality that $\CF$ is induced by an invertible sheaf $\CE=\CO_X(\DeltaP-\DeltaM)$. We thus need to prove that if $\DeltaP$ ample and $j_{\DeltaP}\colon\innt_\Delta\DeltaP\hookrightarrow\Delta$ denotes the open embedding of the interior of $\DeltaP$ relative to $\Delta$, then $\id\to(\R j_*)j^{-1}$ is an isomorphism on
\[
\CF:=(j_{\DeltaP})_!\kk_{\DeltaP}. 
\]
We proceed stalkwise over $\Delta$. Since the isomorphism property is clear for interior points $x\in\innt\Delta$, we pick a boundary point $x\in\partial\Delta=\Delta\setminus\innt\Delta$. For an open neighbourhood $U\subseteq\Delta$ of $x$ we consider first
\[
\kG(U,\CF)\to\H\big(U\cap\innt\Delta, \,
\CF\big).
\]
The left hand side is $\kk$ (instead of $0$) if and only if $U\subseteq \innt_\Delta\DeltaP$, which is equivalent to $U\subseteq \DeltaP$. For the right hand side, this holds precisely if $U\cap\innt\Delta\subseteq \innt_\Delta\DeltaP$, or equivalently, if $U\cap\innt\Delta\subseteq\DeltaP$. Both conditions are obviously equivalent for $\DeltaP$ is closed in $M_\R$. It remains to show that
\[
\varinjlim_{U\ni x}\gH^k\!\big(U\cap\innt\Delta, 
\CF\big)
=(R^k j_*)j^{-1}\big(\CF\big)_x=0
\]
for $k\geq 1$. The Gysin sequence implies
\[
\gH^k\!\big(U\cap\innt\Delta,\CF\big)=
\widetilde{\opn H}^{\raisebox{-3pt}{\scriptsize$k\!-\!1$}}\big((U\cap \innt\Delta)\setminus\DeltaP\big).
\]
If $x\notin\DeltaP$, we can chose $U$ disjoint to $\DeltaP$ so that these groups vanish. 
On the other hand, ampleness of the Weil decoration entails that every $x\in\DeltaP\cap\partial\Delta$ is in the closure of the set $\innt_{\partial\Delta}(\DeltaP\cap\partial\Delta)$, cf.\ the sketch on the right hand side of Figure~\ref{fig:PosAmpWDeco}. For a sufficiently small neighbourhood $U$ of $x$, $(U\cap \innt\Delta)\setminus\DeltaP$ is contractible, whence the claim.
\end{proof}

\bibliographystyle{alpha}
\bibliography{ctvb}
\end{document}